\documentclass[a4paper, 11pt]{article}
\usepackage[latin1]{inputenc}    
\usepackage[T1]{fontenc}
\usepackage{lmodern}
\usepackage[a4paper]{geometry}
\usepackage{graphicx}
\usepackage{subfigure}
\usepackage{amsmath, amsfonts, amssymb, amsthm} 
\usepackage{ulem}
\usepackage{bbold}
\usepackage{framed}
\usepackage{color}
\usepackage{tikz}
\usepackage[english]{babel}
\usepackage[citecolor=blue,colorlinks=true]{hyperref}

\numberwithin{equation}{section} 

\newtheorem{define}{Definition}[section]
\newtheorem{proposition}{Proposition}[section]
\newtheorem{theorem}{Theorem}
\newtheorem{corollary}{Corollary}[section]
\newtheorem{remark}{Remark}[section]
\newtheorem{lemma}{Lemma}[section]
\newtheorem{assumption}{Assumption}

\newcommand*\myref[1]{{\normalfont(\ref{#1})}}
\newcommand*\mycite[1]{{\normalfont \cite{#1}}}
\newcommand*\B[0]{{\cal B}}

\newcommand*\eps{\varepsilon}
\newcommand*\tld{\widetilde}
\newcommand*\E{\mathbb{E}}

\newcommand*\blangle{\Big \langle}
\newcommand*\brangle{\Big \rangle}

\newcommand*\llangle{\langle \langle}
\newcommand*\rrangle{\rangle \rangle}

\DeclareMathOperator{\sign}{sign}

\title{Diffusion-approximation for a kinetic spray-like system with random forcing}
\date{}
\author{ Arnaud Debussche, \; \; Angelo Rosello, 
{\vspace{1mm}} \\
\textit{
Univ Rennes, CNRS, IRMAR - UMR 6625, F-35000 Rennes, France
}
{\vspace{2mm}} \\
\; \; Julien Vovelle 
{\vspace{1mm}} \\
\textit{
UMPA, UMR 5669 CNRS, ENS de Lyon
}
}

\begin{document}

\maketitle

\begin{abstract}
We study a kinetic toy model for a spray of particles immersed in an ambient fluid, subject to some additional random forcing given by a mixing, space-dependent Markov process.
Using the perturbed test function method, we derive the hydrodynamic limit of the kinetic system. The law of the limiting density satisfies a stochastic conservation equation in Stratonovich form, whose drift and diffusion coefficients are completely determined by the law of the stationary process associated with the Markovian perturbation.
\medskip

\end{abstract}

\tableofcontents

\begin{section}{Introduction}

\begin{subsection}{Overview of the kinetic model}

We consider a toy model for spray-like behavior, describing the motion of some particles immersed in an ambient fluid, given by the following conservation equation
\begin{align}
\partial_t f + \eps v \cdot \nabla_x f + \nabla_v \cdot \left( F(t,x,v) f \right) = 0.
\label{micro_eq}
\end{align}
The forcing term
\begin{align*}
F(t,x,v) = (u_t(x) - v) + m_t(x)
\end{align*}
expresses how the speed of a particle tends to align with the local speed $u(x)$ of the ambient fluid. It is perturbed by some random forcing naturally modeled by a space-dependent, centered Markov process $(m_t(x))_{t \ge 0}$. Assuming that the particles have very little impact on the behavior of the ambient fluid, $u \equiv u_t(x)$ evolves according to
\begin{align}
\partial_t u - \Delta u = \eps^2 \int_v (v-u) f dv .
\label{micro_eq_u}
\end{align}
We could easily add a nonlinear term of Burger's type $u \cdot \nabla_x u$ in equation \myref{micro_eq_u}. This would not yield any mathematical difficulty but would unnecessarily give heavier proofs.

In order for the system  \myref{micro_eq}-\myref{micro_eq_u} to be well-posed and admit a solution $(f_t(x,v),u_t(x))$, we need, in fact, to restrict to the one-dimensional case: 
$$
x \in \mathbb{T} = \mathbb{R}/\mathbb{Z}, \hspace{5mm} v \in \mathbb{R}.
$$
One can refer to section \ref{app} below (and in particular Remark \ref{well_posedness_remark}) for details regarding the well-posedness of this system. The rescaled variables
\begin{align*}
f^\eps_t(x,v) := f_{\eps^{-2}t}(x,v),
\hspace{10mm}
u^\eps_t(x) :=\eps^{-1} u_{\eps^{-2}t}(x),
\hspace{10mm}
m^\eps_t(x) := m_{\eps^{-2} t} (x),
\end{align*}
then satisfy
\begin{equation}
\left\{
\begin{array}{l}
\displaystyle{ 
\partial_t f^\eps + \eps^{-1} v   \partial_x f^\eps + \eps^{-2} \partial_v   \Big[ ( \eps u^\eps(x) - v + m^\eps(x)) f^\eps \Big] = 0 ,
} \\ 
\displaystyle{ \partial_t u^\eps - \partial^2_x u^\eps = \int (v- \eps u^\eps) f^\eps dv } .
\end{array}
\right.
\label{eq}
\end{equation}
Introducing the integrated quantity
\begin{align}
\rho^\eps(x) = \int f^\eps(x,v) dv,
\end{align}
we are interested in the behavior of  $(\rho^\eps, u^\eps)$ as $\eps$ goes to zero. In essence, we wish to describe how the random forcing $m$ added on a "mesoscopic level" translates on a "macroscopic level" in the scaling limit. More precisely, we shall determine the stochastic partial differential equation satisfied by the limiting law of $\rho^\eps$ as $\eps$ goes to $0$.

This work is part of a program on  stochastic diffusion-approximation in the context of kinetic equation and fluid limits (see \cite{ddv, dv1, dv2}). In the case considered here, it is to be noted that due to the absence of collision term in the kinetic equation of \myref{eq}, the limiting SPDE satisfied by $\rho^\eps$ will be of order one: it is in a fact a stochastic conservation equation when written in Stratonovich form. 

This latter stochastic equation is linear but our proof provides a solution which has very low regularity and uniqueness in this context 
 is not trivial. We need to use a duality argument and backward SPDEs to prove uniqueness.
 
The study of the scaling limit relies essentially on the perturbed test function method, first introduced in \cite{papa} and developed further in \cite{wave}.
In contrast with previous works on the matter, particular attention will be paid to the rigorous construction of the correctors involved in the perturbed test function method. This will be achieved by handling infinitesimal generators and their domains in an overall more cautious way, and by exploiting the mixing assumptions made on the driving Markov process to properly solve Poisson equations associated with these generators.

\end{subsection}

\begin{subsection}{Main results}

Throughout our study, we will be led to work with the following spaces:
\begin{align*}
& L^p_{x,v} := L^p(\mathbb{T} \times \mathbb{R}), \hspace{5mm} L^p_x := L^p(\mathbb{T}), 
\\
& W^{s,p}_x := W^{s,p}(\mathbb{T}), \hspace{8mm} H^\sigma_x := W^{2, \sigma}(\mathbb{T}).
\end{align*}
We will denote by ${\cal P}(E)$ the set of probability measures on a metric space $E$.
When it is well defined, we will denote
\begin{align*}
\langle f,g \rangle = \int f(x,v) g(x,v) dx dv \; \text{ or } \; \int f(x) g(x) dx.
\end{align*}
depending on the context. Finally, we shall sometimes use the handy notation $A(z) \lesssim B(z)$ to signify that there exists a constant $C > 0$ such that $A(z) \le C B(z)$ for all $z$.
\vspace{3mm}

Let us now state our diffusion-approximation result.

\begin{theorem} \label{chap1-thm1}
Let $(\Omega, {\cal F}, ({\cal F}_t)_{t \ge 0}, \mathbb{P})$ be a filtered probability space equipped with a Markov process $m = (m_t)_{t \ge 0}$ satisfying Assumptions \ref{ass1} through \ref{ass4}, described in section \ref{m_ass} below. Let $T > 0$.
Assume that the initial data $(f_0^\eps, u_0^\eps)_{\eps > 0}$ satisfy
\begin{align*}
& f_0^\eps \ge 0 \text{ and } \int \int f^\eps_0(x,v) dx dv = 1,
\\
&  \rho^\eps_0 \to \rho_0 \text{ in $L^1_x$}, 
\hspace{5mm}
 u^\eps_0 \to u_0 \text{ in $H^\alpha_x$ for some $\alpha \in (1/2,3/2]$}
\end{align*}
and, for all fixed $\eps > 0$,
\begin{align} \label{fixed}
\int_x \int_v (1 + |v|^4) |f_0^\eps|^2 dx dv + \int_x \int_v (1+|v|) | \nabla_{x,v} f^\eps_0 | dx dv < \infty, 
\hspace{5mm}
u_0^\eps  \in H^\eta_x \text{ for} \eta(\eps) > 3/2.
\end{align}
Assume the following uniform bound:
\begin{align}
 \sup_{\eps > 0} \int_x \int_v |v|^3 f_0^\eps dx dv < \infty.
\end{align}
Let $(f^\eps,u^\eps)$ be the unique solution of \myref{eq} with initial data $(f_0^\eps, u_0^\eps)$ in the sense given in Proposition \ref{ex_un} below. 
Then for all $\gamma > 1/2$ and $\beta < \alpha$, there exists a subsequence $(\rho^{\eps_k}, u^{\eps_k})_{k \ge 1}$ such that
\begin{align*}
(\rho^{\eps_k}, u^{\eps_k}) \to (\rho, u) \text{ in law in } C([0,T] ; H^{-\gamma}_x \times H^\beta_x)
\end{align*}
where $u$ satisfies the heat equation
$$
\partial_t u - \partial_x^2 u = 0, \hspace{5mm} \text{with } \; u(0) = u_0,
$$
and $\rho$ is a martingale solution (see \mycite{da_prato}, Chapter 8, for the definition) of 
\begin{align}
d \rho = \partial_x \Big[ (a-u) \rho \Big]dt + \partial_x \Big[ \rho Q^{1/2} \circ dW_t \Big], 
\hspace{5mm} \text{with } \; \rho(0) = \rho_0,
\label{edps}
\end{align}
satisfying
$ \sup_{t \in [0,T]} \| \rho_t \|_{ L^1_x}  \le 1$ a.s.
In \myref{edps}, $(W_t)_{t \ge 0}$ denotes a cylindrical Wiener process on $L^2_x$. The coefficient $a \equiv a(x)$  and the trace operator $Q$ on $L^2_x$ are completely determined by the law of the stationary process $(\tld m_t)_{t \ge 0}$ associated with the random driving term $m$ (see section \ref{m_ass} below for a definition of $\tld m$):
\begin{align}
 a(x) =  \frac{1}{2} \int_0^\infty \E \Big[ \partial_x \tld m_0(x) \tld m_t(x) -  \tld m_0(x) \partial_x \tld m_t(x) \Big] dt 
+ \int_0^\infty e^{-t} \E \Big[ \tld m_0(x) \partial_x \tld m_t(x) \Big] dt,
\label{exp_a}
\\
 Qf(x) = \int_y k(x,y)f(y)dy,
\; \text{ with } \; k(x,y) = \int_{\mathbb{R}} \E \Big[ \tld m_0(x) \tld m_t(y) \Big] dt.
\hspace{13mm}
\label{exp_Q}
\end{align}

\end{theorem}

Note that assumption \myref{fixed} is only necessary to ensure that the system \myref{eq} can indeed be uniquely solved for fixed $\eps > 0$, as will be stated later on, in Proposition \ref{ex_un}. 
The result of
Theorem \ref{chap1-thm1} holds when the underlying probability space $\Omega$ is equipped with the probability~$\mathbb{P}_\mu$ $-$ so that $\mathbb{P}_\mu[ m_0 \in A] = \mu(A)$ $-$
for any initial law $\mu \in {\cal P}(E)$.
For instance, one may consider the case $\mu = \delta_n$ where the process $(m_t)_{t \ge 0}$ starts at some fixed $n \in E$ ; or one may consider the case where $\mu = \nu$ is the invariant measure of $m$, so that $(m_t)_{t \ge 0} := (\tld m_t)_{t \ge 0}$ is a stationary process (see again section \ref{m_ass} below).

The weak existence result established in Theorem \ref{chap1-thm1} is completed by the following strong uniqueness result.

\begin{theorem} \label{chap1-thm2}
Let $\rho_0 \in L^1_x$ and $u_0 \in W^{4,\infty}_x$. 
Assume furthermore that the coefficient $a$ and the kernel $k$ defined in \myref{exp_a}, \myref{exp_Q} satisfy
\begin{align}
a \in W_x^{4,\infty}, \hspace{10mm} \sup_{y \in \mathbb{T}} \| k(\cdot, y) \|_{W_x^{10+\delta, \infty}} < \infty 
\; 
\text{ for some } \delta > 0.
\label{regularity}
\end{align}
Then a solution $\rho \in C([0,T] ; H^{-\gamma}_x)$ of \myref{edps} satisfying the bound
$$
\sup_{t \in [0,T]} \| \rho_t \|_{L^1_x} \le 1 \text{ a.s}
$$
is path-wise unique.
\end{theorem}

As a consequence of Theorem \ref{chap1-thm2}, Yamada-Watanabe's Theorem guarantees that uniqueness in law ({\it i.e.} weak uniqueness in the probabilistic 
sense) holds for \myref{edps} and the whole sequence converges in Theorem \ref{chap1-thm1}.
The regularity properties \eqref{regularity} will be satisfied under some adequate assumptions on the driving process $(m_t)_{t \ge 0}$ (see the end of section~\ref{m_ass}).

In the limit where $\eps$ goes to $0$, the rather complex random process $(\rho^\eps_t)_{t \ge 0}$ driven by a general Markov process $m=(m_t)_{t \ge 0}$ can therefore be approximated by a much simpler diffusion $(\rho_t)_{t \ge 0}$ whose drift and covariance are explicitely determined by the law of the associated stationary process $\tld m = (\tld m_t)_{t \ge 0}$. 
Let us give a quick overview of the proof of Theorems \ref{chap1-thm1} and \ref{chap1-thm2}.
\vspace{3mm}

In section \ref{chap1-2} we start by introducing a rigorous framework for solving Poisson equations associated with generators of Markov processes, which will be of great use in section \ref{5} in particular. The tools and results presented here are all very classical, however the authors could not find a reference really suited for the specific setup adopted in the present work. We then describe the assumptions made on the driving process $m$ and their consequences in section \ref{m_ass}.
In section \ref{3}, we establish some estimates regarding the solutions of \myref{eq} and rigorously determine the generator of the process $(f^\eps,u^\eps,m^\eps)$.
The essential work is developed in section \ref{5} where we use the perturbed test function method to determine the limiting equation \myref{edps} and pave the way towards the convergence of the martingale problem. Section \ref{6} concludes the proof by proving the convergence announced in Theorem \ref{chap1-thm1}. Finally, the path-wise uniqueness of \myref{edps} stated in Theorem \ref{chap1-thm2} is looked upon in section \ref{7}.
\end{subsection}

\end{section}

\begin{section}{Generators and Poisson equations} \label{chap1-2}

\begin{subsection}{Generalities on generators} \label{markov}

On a filtered probability space $(\Omega, {\cal F}, ({\cal F}_t)_{t \ge 0})$, let us consider a càdlàg, time-homogeneous Markov process  $X=(X_t)_{t \ge 0}$  taking values in some separable metric space ${\cal X}  \equiv({\cal X},d_{\cal X})$. Throughout this paper, we will naturally use the handy notation $(X_t(x))_{t \ge 0}$ to denote $(X_t)_{t \ge 0}$ under the probability $\mathbb{P}_{\delta_x}$ (that is $X_0(x) = x$ a.s). 
All the processes we shall encounter in this work satisfy the following assumptions:
\begin{itemize}
\item $X$ is stochastically continuous:
\begin{align}
\forall x \in {\cal X}, \; \; \forall t \ge 0, \; \; \forall \varepsilon > 0, \; \; \mathbb{P} \Big[ d_{\cal X}(X_s(x),X_t(x)) > \varepsilon \Big] \xrightarrow[s \to t]{} 0.
\label{stoc_cont}
\end{align}
 \item $X$ is locally bounded (in space and time):
\begin{align}
\forall r > 0, \forall T > 0, \exists R > 0, \forall x \in B(0,r), \; 
\mathbb{P} \Big[ \forall t \in [0,T],  \; \; X_t(x) \in B(0,R)  \Big] = 1.
\label{loc_bound}
\end{align}
\end{itemize}

Let us denote by $(Q_t)_{t \ge 0}$ the semigroup associated to $(X_t)_{t \ge 0}$. The local boundedness assumption \myref{loc_bound} allows to define it over locally bounded test functions: let us introduce the spaces
 \begin{align}
& \B^{loc}({\cal X}) = \{ \psi : {\cal X} \to \mathbb{R} \text{ measurable}, \;  \psi \text{ maps bounded sets onto bounded sets } \} \\
& C_b^{loc}({\cal X}) = \{ \psi : {\cal X} \to \mathbb{R} \text{ continuous}, \;  \psi \text{ maps bounded sets onto bounded sets } \} 
\end{align}
and define the semigroup $Q_t : \B^{loc}({\cal X}) \to \B^{loc}({\cal X})$ as
\begin{align*}
\forall \psi \in \B^{loc}({\cal X}), \; \; Q_t \psi(x) = \E[\psi(X_t(x))] .
 \end{align*} 
Note that, using assumption \myref{stoc_cont}, the map $t \mapsto Q_t \psi(x)$ is continuous whenever $\psi \in C_b^{loc}({\cal X})$. 
The Markov property
\begin{align*}
\forall \psi \in \B^{loc}({\cal X}), \; \; \; \E [ \psi(X_{t+s}) | {\cal F}_s ] = Q_t \psi(X_s) \; \text{ a.s}
\end{align*}
yields in particular the semigroup property $Q_{t+s} = Q_t Q_s$ on $\B^{loc}({\cal X})$.

 \begin{define}[Infinitesimal generator] 
The infinitesimal generator of $X$ is the operator
$$A : D(A) \subset \B^{loc}({\cal X}) \to C_b^{loc}({\cal X})$$ 
defined as follows: for all $\psi \in \B^{loc}({\cal X})$ and $\theta \in C_b^{loc}({\cal X})$,
\begin{align*}
\psi \in D(A) \text{ and } A \psi = \theta \; \; \text{ if and only if } \; \;  \frac{Q_t \psi - \psi}{t} \xrightarrow[ t \to \, 0^+]{b.p.c} \theta \in C_b^{loc}({\cal X})
\end{align*}
where $b.p.c$ denotes the (locally) bounded pointwise convergence:
\begin{align*}
\theta_t  \xrightarrow[]{b.p.c} \theta  \Longleftrightarrow \left\{
\begin{array}{l}
\forall x \in {\cal X}, \; \; \theta_t(x) \to \theta(x) \\
\forall B \subset {\cal X} \text{ bounded}, \; \; \displaystyle{ 
\sup_{t > 0}  \sup_{x \in B} | \theta_t (x) |  < \infty
} .
\end{array}
\right.
\end{align*}
\end{define}

We then have the following classical result.

\begin{proposition}[Kolmogorov's equation] \label{chap1-kolmo}
For all $\psi \in D(A)$, for all $t \ge 0$,
\begin{align*}
\frac{Q_{t+s} \psi - Q_t \psi}{s}  \xrightarrow[ s \to 0]{b.p.c}  Q_t A \psi .
\end{align*}
In particular, for all $x \in {\cal X}$,
\begin{align*}
\frac{d}{dt}Q_t \psi(x) = Q_t A \psi(x).
\end{align*}
\end{proposition}

\begin{remark}
Provided that $Q_t A \psi \in C_b^{loc}({\cal X})$, this states that $Q_t \psi \in D(A)$ and $A Q_t \psi = Q_t A \psi$.
\end{remark}

\begin{proof}[Proof]
For $s > 0$, we may write
\begin{align*}
\frac{Q_{t+s} \psi - Q_t \psi}{s}(x) = Q_t \Big[ \frac{Q_{s} \psi - \psi}{s} \Big](x)
= \E \Big[ \frac{Q_{s} \psi (X_t(x)) - \psi(X_t(x))}{s} \Big].
\end{align*}
Making use of \myref{stoc_cont} and \myref{loc_bound}, we may use dominated convergence as $s$ goes to $0$ to conclude for the right derivative. As for the left derivative, we simply note that
\begin{align*}
\frac{Q_t \psi - Q_{t-s} \psi}{s} = Q_{t-s} \Big[ \frac{Q_s \psi - \psi}{s} - A\psi \Big] + Q_{t-s} A \psi
\end{align*}
and conclude in the same way,  keeping in mind that $A\psi~\in~C_b^{loc}({\cal X})$ so that  $u \mapsto Q_u A\psi(x)$ is continuous. 
. 
\end{proof}

The following property will be of great interest, and could in fact be an alternative way of defining the generator $A$.

\begin{proposition}[Martingale problem]
For all $\psi \in D(A)$, for all $x \in {\cal X}$,
\begin{align*}
 M_\psi(t) = \psi(X_t(x)) - \int_0^t A \psi (X_s(x)) ds, \; \; t \ge 0
\end{align*}
defines an $({\cal F}_t)_{t \ge 0}$-martingale.
\end{proposition}

\begin{proof}[Proof]
Let us write $X_t := X_t(x)$ for simplicity.
We need to check that, for all $h$ bounded, ${\cal F}_s$-measurable,
\begin{align*}
\E \Big[ \Big( \psi(X_t) - \psi(X_s) \Big) h \Big] = \E \Big[ \Big( \int_s^t A \psi(X_u) du \Big) h \Big]  .
\end{align*}
Using the Markov property (twice) and Proposition \ref{chap1-kolmo}, the right hand side can be written as
\begin{align*}
 \E \Big[ \E \Big[ \int_s^t A \psi(X_u) du | {\cal F}_s \Big] h \Big] & =
 \E  \Big[ \Big( \int_s^t \Big[ \E A \psi(X_u) | {\cal F}_s \Big] du \Big)  h \Big] 
  = \E  \Big[ \Big( \int_s^t Q_{u-s} A \psi (X_s) du \Big)  h \Big] \\
 & = \E \Big[ \Big( \int_0^{t-s} \frac{d}{du} Q_u \psi(X_s) du \Big) h   \Big] 
  = \E \Big[ \Big(  Q_{t-s} \psi(X_s) du - \psi(X_s) \Big) h   \Big]  \\
 & = \E \Big[ \Big( \psi(X_t) - \psi(X_s) \Big) h \Big] .
\end{align*}
Again, we have used the fact $u \mapsto Q_u A\psi(x)$ is continous since $A\psi$ is continuous.
\end{proof}

Let us now give a natural context in which one can interchange generator and integral.
\begin{proposition}[Generator under the integral sign] \label{gen_int} 
Let $(t,x) \in \mathbb{R}^+ \times {\cal X} \mapsto \psi_t(x) \in \mathbb{R}$ be a measurable map, such that $\psi_t \in D(A)$ for all $t \ge 0$.
Assume that, for all $B \subset {\cal X}$ bounded,
\begin{align}
& \int_0^\infty \sup_{x \in B} | \psi_t(x) | dt < \infty 
\label{int1}\\
& \int_0^\infty \sup_{x \in B} | A \psi_t(x) | dt < \infty 
\label{int2}
\end{align}
Then $\displaystyle{\psi = \int_0^\infty \psi_t dt \in D(A)}$ and $ \displaystyle{ A \psi = \int_0^\infty A \psi_t dt }$.
\end{proposition}

\begin{proof}[Proof]
Let us not get into the questions of measurability. First, $\psi$ is well defined, in $\B^{loc}(X)$, and we have
$
Q_s \psi (x)= \int_0^\infty Q_s \psi_t(x) dt$.
The interversion is justified by Fubini's theorem since, using \myref{int1},
\begin{align*}
\E \Big[ \int_0^\infty | \psi_t | (X_s(x)) dt  \Big] \le \int_0^\infty \sup_{y \in B} |\psi_t(y) | dt <\infty
\end{align*}
where $B \subset {\cal X}$ is a bounded set such that $\forall \sigma \in [0,s], \; \; X_\sigma(x) \in B, \; \mathbb{P}$ a.s. It follows that 
\begin{align*}
\frac{ Q_s \psi(x) - \psi(x) }{s} = \int_0^\infty \frac{Q_s \psi_t(x) - \psi_t(x)}{s} dt .
\end{align*}
For $x \in B(0,r)$ for some $r > 0$, for $s \le 1$, we have the bound
\begin{align*}
\Big|  \frac{Q_s \psi_t(x) - \psi_t(x)}{s} \Big| \le \sup_{\sigma \in [0,s]} \Big| \frac{d}{d\sigma} Q_\sigma \psi_t(x) \Big| = \sup_{\sigma \in [0,s]} \Big| Q_\sigma A \psi_t(x) \Big| \le \sup_{y \in B} \Big| A \psi_t(y) \Big|
\end{align*}
where $B \subset {\cal X}$ is bounded and such that $\forall x \in B(0,r)$, $\forall \sigma \in [0,1], \; \; X_\sigma(x) \in B$, $\mathbb{P}$-a.s. Dominated convergence then gives the b.p.c of $ \int_0^\infty s^{-1}(Q_s \psi_t(x) - \psi_t(x)) dt$ to $\int_0^\infty A \psi_t(x) dt$. 
Assumption \myref{int2} and the fact that $A \psi_t \in C_b^{loc}({\cal X})$ for all $t \ge 0$ guarantee that $\int_0^\infty A \psi_t dt$ is in $C_b^{loc}({\cal X})$, so that we can indeed state $\psi \in D(A)$.
\end{proof}

\begin{proposition}[Resolution of a Poisson equation $A \psi = \theta$] \label{poisson} 
Let $\theta \in D(A) \cap C_b^{loc}({\cal X})$. Assume that, for all $B \subset {\cal X}$ bounded,
\begin{align}
\int_0^\infty \sup_{x \in B} | Q_t \theta(x) | dt < \infty  \label{integrabilityPoisson} \\
\int_0^\infty \sup_{x \in B} | Q_t A \theta(x) | dt < \infty
\label{int22}
\end{align}
Then $\displaystyle{\psi = - \int_0^\infty Q_t \theta dt \in D(A)}$ and $A \psi = \theta$.
\end{proposition}

\begin{proof}[Proof]
Since $Q_t A \theta$ is not necessarily continuous, we cannot exactly state that $\psi_t = Q_t \theta \in D(A)$. Let us hence be a little more careful: using the same arguments as the previous proof, we can show that
\begin{align*}
\frac{Q_s \psi(x) - \psi(x)}{s} \xrightarrow[s \to 0]{b.p.c} - \int_0^\infty Q_t A \theta(x) dt = \lim_{T \to \infty} - \int_0^T \frac{d}{dt} Q_t \theta(x) dt
= \theta(x) - \lim_{T \to \infty} Q_T \theta(x).
\end{align*}
Hence, $\lim_{T \to \infty} Q_T \theta(x)$ exists and assumption \myref{integrabilityPoisson} implies $\lim_{T \to \infty} Q_T \theta(x) = 0$. We have shown
\begin{align*}
\frac{Q_s \psi - \psi}{s} \xrightarrow[s \to 0]{b.p.c}  \theta
\end{align*}
and $\theta$ is in $C_b^{loc}({\cal X})$ by assumption, so that we have indeed $\psi \in D(A)$.
\end{proof}

\begin{remark} \label{rem}
Assuming \myref{integrabilityPoisson} only, one can still define the $\B^{loc}({\cal X})$ function $
\psi = - \int_0^\infty Q_t \theta dt$. If for some reason we know in advance that $\psi \in D(A)$, then we can conclude immediately that $A \psi = \theta$ without assumption \myref{int22}. Indeed,
\begin{align*}
Q_s \psi(x) = -  \int_0^\infty Q_{t+s} \theta(x) dt = - \int_s^\infty Q_t \theta(x) dt
\end{align*}
which gives, since $t \mapsto Q_t \theta(x)$ is continuous,
\begin{align*}
\frac{d}{ds^+} Q_s \psi(x) \Big|_{s = 0} = Q_0 \theta(x) = \theta(x)
\end{align*}
and using Kolmogorov's equation, we know that 
\begin{align*}
\frac{d}{ds^+} Q_s \psi(x) \Big|_{s = 0} = A \psi(x).
\end{align*}
\end{remark}

\end{subsection}

\begin{subsection}{The random driving term} \label{m_ass}

We now describe the assumptions made on the process $m = (m_t)_{t \ge 0}$ driving equation \myref{eq}. Let us define its canonical filtration
\begin{align*}
{\cal F}_t = \sigma(m_s, \; 0 \le s \le t), \; \; \; t \ge 0.
\end{align*}
We assume that $m=(m_t)_{t \ge 0}$ is a càdlàg, stochastically continuous Markov process taking values in some separable, complete normed space $(E, \| \cdot \|_E$), satisfying the continuous embedding
\begin{align}
m_t \in E \subset W_x^{2,\infty}, \; \text{ with } \;   \| \cdot \|_{W_x^{2,\infty}} \lesssim \| \cdot \|_E .
\end{align}
This enables us to consider linear test functions of the form
\begin{align*}
\psi_h : n \in E \mapsto \langle h,n \rangle, \; \; \; h \in (W^{2,\infty}_x)'.
\end{align*}
\begin{assumption}[Bounded state space] \label{ass1}
For some $C^* > 0$,
\begin{align}
\| n \|_{E} \le C^*
\; \; \text{for all $n \in E$}.
\end{align}
 \end{assumption}
We denote by $(P_t)_{t \ge 0}$ the semigroup associated to $m$ and  by $(M,D(M))$ its infinitesimal generator, in the sense of the b.p.c convergence introduced in section \ref{markov}. 

\begin{assumption}[Stationary measure] \label{ass2}
There exists some measure $\nu \in {\cal P}(E)$ satisfying:
\begin{align}
& \forall t \ge 0, \; \forall \psi \in \B(E), \; \; \; \int_E P_t \psi(n)  d\nu(n) = \int_E \psi(n) d\nu(n)
\tag{Stationary}
\\
& \forall h \in (W^{2,\infty}_x)', \; \;  \int \psi_h(n) d\nu(n) = 0 .
\tag{Centered}
\end{align}
We shall denote by $(\tld m_t)_{t \ge 0}$ the process with initial law $\nu$. 
\end{assumption}

Note that it is possible to build a càdlàg stationary process $(\widehat m_t)_{t \in \mathbb{R}}$ indexed by $t \in \mathbb{R}$: one may indeed use the Kolmogorov extension theorem with the finite dimensional distributions
$$
\mathbb{P} \Big[ \widehat m_s \in A_0, \;  \widehat m_{s+t_1} \in A_1, \; \ldots \;
\widehat m_{s+t_n} \in A_n \Big] = 
\mathbb{P} \Big[ \tld m_0 \in A_0, \;  \tld m_{t_1} \in A_1, \; \ldots \;
\tld m_{t_n} \in A_n \Big] 
$$
for $s \in \mathbb{R}$ and $t_1, \ldots, t_n \ge 0$.
This extended stationary process will still be denoted by 
 $(\tld m_t)_{t \in \mathbb{R}}$.
\vspace{3mm}
 
Furthermore, we require that $m$ satisfies the following mixing property.

\begin{assumption}[Mixing] \label{ass3}
There exists a non-increasing, integrable function $\gamma_{\text{mix}} \in L^1(\mathbb{R}^+)$ such that, for any initial state $n \in E$, there exists a càdlàg coupling $(m_t^*(n),\tld m_t^*)_{t \ge 0}$, whose marginal laws are respectively those of $(m_t(n))_{t \ge 0}$ and $(\tld m_t)_{t \ge 0}$, satisfying
\begin{align}
\E \| m_t^*(n) - \tld m_t^* \|_{E} \le \gamma_{\text{mix}}(t), \; \; \; t \ge 0 . \label{mixing}
\end{align}
\end{assumption}
An immediate consequence is the following: for any Lipschitz-continuous function $\psi : E \to \mathbb{R}$ such that $\int_E \psi(n) d\nu(n) = 0$,
\begin{equation}
\Big| P_t \psi(n) \Big| = \Big| \E[  \psi(m^*_t(n)) - \psi(\tld m^*_t) ]  \Big|
\le \| \psi \|_{Lip} \gamma_{mix}(t) \in L^1_t(\mathbb{R}^+).
\label{integrability}
\end{equation}
As a result, the coefficient $a(x)$ and the kernel $k(x,y)$ introduced in \myref{exp_a}, \myref{exp_Q} are well-defined: for $k(x,y)$ we may for instance write (using the stationarity of $\tld m$)
\begin{align*}
\int_\mathbb{R} \Big| \E \Big[ \tld m_0(x) \tld m_t(y) \Big] \Big| dt
= \int_0^\infty \Big| \E \Big[ \tld m_0(x) \tld m_t(y) \Big] \Big| dt  + \int_0^\infty \Big| \E \Big[ \tld m_0(y) \tld m_t(x) \Big] \Big| dt
\end{align*}
and note that, using the Markov property,
\begin{align*}
\int_0^\infty  \Big| \E \Big[ \tld m_0(x) \tld m_t(y) \Big] \Big| dt =
\int_0^\infty \Big|  \int_E n(x) \E \Big[ m_t(n)(y) \Big] d\nu(n) \Big| dt
\le C^* \int_0^\infty \gamma_{mix}(t) dt < \infty.
\end{align*}
As discussed in section \ref{markov}, the property \myref{integrability} allows to solve a number of Poisson equations.  We may indeed define, for $h \in (W^{2,\infty}_x)'$,
\begin{align*}
M^{-1}\psi_h(n) = -\int_0^\infty P_t \psi_h (n) dt = -\int_0^\infty \E \blangle h , m_t(n) \brangle dt .
\end{align*}
It is then natural to introduce the function $n \in E \mapsto M^{-1}I(n)$ defined by
\begin{align*}
\forall h \in (W^{2,\infty}_x)', \; \; \blangle M^{-1}I(n) , h \brangle = M^{-1} \psi_h(n).
\end{align*}
In other words,
\begin{equation}
M^{-1}I(n) = - \int_0^\infty \E [ m_t(n) ] dt \in W_x^{2,\infty}.
\label{M-1}
\end{equation}
Note that the function $M^{-1}I$ is bounded: since $\E [ \tld m_t] = \int_E n d\nu(n) = 0$ in $W^{2,\infty}_x$, we have
\begin{align*}
\| M^{-1} I(n) \|_{W^{2,\infty}_x} & \le \int_0^\infty  \| \E [  m_t(n) ] \|_{W^{2,\infty}_x} dt
= \int_0^\infty  \| \E [  m_t^*(n) - \tld m_t^* ] \|_{W^{2,\infty}_x} dt
\\ &
\le \int_0^\infty \E \| m_t^*(n) - \tld m_t^* \|_E dt \le \int_0^\infty \gamma_{mix}(t) dt < \infty.
\end{align*}

\begin{assumption}[Generator] \label{ass4} 
For all $h \in (W_x^{2,\infty})'$,
$
M^{-1} \psi_h \in D(M)
$
and $M^{-1}I$ satisfies
\begin{align}
\| M^{-1}I(n) - M^{-1}I(n') \|_{W^{2,\infty}_x} \lesssim \|n-n'\|_E.
\label{lipschitz}
\end{align}
Furthermore, for all $h \in L^1_x$,
\begin{align}
| M^{-1} \psi_h|^2 \in D(M) \; \; \text{ and } \; \;  \Big| M |M^{-1} \psi_h|^2 (n) \Big| \lesssim \| h \|^2_{L^1_x}.
\label{chap1-quadratic}
\end{align}
\end{assumption}
Note that, as stated in Remark \ref{rem}, the assumption $M^{-1} \psi_h \in D(M)$ guarantees that
$$M \blangle M^{-1}I(n) , h \brangle = \langle n,h \rangle.$$ 
It is typically obtained under the assumptions of Proposition \ref{poisson}. Combining \myref{lipschitz} with the mixing property \myref{mixing} essentially allows to define $M^{-2}I$, which shall prove necessary in the calculations developed later on.

Let us conclude by discussing the additional assumption \myref{regularity} made in Theorem \ref{chap1-thm2}. This requires of course some more regularity for the stationary process $\tld m$. A natural setup would be the following: assume that the stationary measure $\nu$ is reversible (that is $(\tld m_{-t})_{t \ge 0} = (\tld m_t)_{t \ge 0}$ in law) and supported in $W^{10+\delta, \infty}_x$, with
\begin{align*}
\E \| \tld m_t \|_{W_x^{10+\delta, \infty}}^2  =\int_{E} \| n \|^2_{W_x^{10+\delta, \infty}} d\nu(n) < \infty.
\end{align*}
Making use of the time-reversibility \myref{exp_a} and \myref{exp_Q} become
\begin{align}
a(x) = \int_0^\infty e^{-t} \E \Big[ \tld m_0(x) \partial_x \tld m_t(x) \Big] dt,
\hspace{10mm}
k(x,y) = 2 \int_0^\infty \E \Big[ \tld m_0(x) \tld m_t(y) \Big] dt
\end{align}
and it follows that
\begin{align*}
& \| a \|_{W_x^{4,\infty}} 
\lesssim \Big( \int_0^\infty e^{-t} dt  \Big)  \Big( \int_E \| n \|_{W_x^{5,\infty}}^2 d\nu(n) \Big) < \infty,
\end{align*}
and
\begin{align*}
 \| k( \cdot, y) \|_{W_x^{10+\delta, \infty}} & \le 2 \int_0^\infty \Big\| \int_E n \E \Big[ m_t(n)(y) \Big] d\nu(n) \Big\|_{W_x^{10+\delta, \infty}} dt
 \\
 & \le 2 \Big( \int_0^\infty \gamma_{mix}(t) dt \Big) \Big( \int_E \| n \|_{W_x^{10+\delta, \infty}} d\nu(n) \Big) < \infty.
\end{align*}
Alternatively, one could drop the reversibility assumption on $\nu$ and assume that the whole process $(m_t)_{t \ge 0}$ takes values in $E \subset W_x^{10+\delta, \infty}$ with $\| \cdot \|_{W_x^{10 + \delta, \infty}} \lesssim \| \cdot \|_E$.

\end{subsection}

\begin{subsection}{An elementary example}

Let us give a simple example of a Markov process $m=(m_t)_{t \ge 0}$ satisfying the assumptions detailed above, in the form of a jump process. Consider a discrete-time Markov chain 
$(M_i)_{i \in \mathbb{N}}$ whose  (bounded) countable state space $E$ consists of sufficiently smooth functions:
$$
E = \{ n^j, j \in \mathbb{N}^* \} \subset C^k(\mathbb{T}_x), \hspace{5mm}
 \| \cdot \|_E := \| \cdot \|_{C^k _x},
$$
for some $k$ large enough, say $ k > 10$. Let us denote by $P=(P(n,n'))_{n,n' \in E}$ its transition matrix. Let us consider a probability space $(\Omega, {\cal F}, \mathbb{P})$ where an i.i.d sequence $(U_i)_{i \in \mathbb{N}^*}$ of random variables uniformly distributed in $[0,1)$ is given. The trajectory of $(M_i(n))$ for any initial data $n \in E$ can then classically be simulated in the following manner: $M_0(n) = n$ and, for all $ i \in \mathbb{N}$,
$$ 
M_{i+1}(n) = T(M_i(n) ; U_{i+1} ),
$$
where the jump function is given by
$$
T(M ; U) = n^j \text{ when } \sum_{\ell =1}^{j-1} P(M, n^\ell) \le U < \sum_{\ell =1}^{j} P(M, n^\ell).
$$
We may define the time-continuous, constant-rate jump process associated with the discrete Markov chain $(M_i)$:  considering a Poisson process $(N(t))_{t \ge 0}$ of rate $1$ with $N(0)=0$, independent of $(U_i)$, we simply define
$$
m_t(n) = M_{N(t)}(n), \hspace{5mm} t \ge 0.
$$

Let us assume that the transition matrix $P=(P(n,n'))$ is irreducible, aperiodic and positive recurrent. As a result, let us introduce the unique invariant law $\nu \in {\cal P}(E)$ of the chain $(N_i)_{i \in \mathbb{N}}$ (which we may assumed to be centered, in accordance with Assumption \ref{ass2}). 
\vspace{3mm}

Let us prove that the mixing property given in Assumption \ref{ass3} is satisfied. To this intent, we introduce a coupling similar to the one developed in the classical proof of the ergodic theorem for Markov chains.
Let us consider a random variable $\tld n \in E $ of distribution $\nu$ and an i.i.d sequence $(\tld U_i)_{i \in \mathbb{N}^*}$ of uniform random variables, mutually independent and independent of $(U_i)_{i \in \mathbb{N}^*}$. One can then introduce the Markov chain $(\tld M_i(\tld n))_{i \in \mathbb{N}^*}$ defined by
$\tld M_0(\tld n) = \tld n$ and, for all $i \in \mathbb{N}$,
$$
\tld M_{i+1}( \tld n) = T ( \tld M_i(n) ; \tld U_{i+1} ).
$$
As a result, $(\tld M_i(\tld n))$ defines a stationary Markov chain of transition matrix $P$, independent of $(M_i(n))$. For fixed $n \in E$, we then define the stopping time
$$
\tau = \inf \left\{ i \ge 0, \; \; M_i(n) = \tld M_i(\tld n) \right\}
$$
and the process
$$
M_i^*(n) = \left\{
\begin{array}{l}
M_i(n) \text{ if } i \le \tau,
\\
\tld M_i(\tld n) \text{ if } i > \tau.
\end{array}
\right.
$$
One can easily see that $(M_i^*(n))$ defines a Markov chain of transition matrix $P$, so that
$$
\Big( M_i^*(n), \; i \ge 0  \Big) \sim \Big(  M_i(n), \; i \ge 0 \Big) \; \text{ in law}.
$$
Moreover, $(C_i)_{i \in \mathbb{N}} = (M_i(n) , \tld M_i(\tld n) )_{i \mathbb{N}}$ defines a Markov chain on $E \times E$, of transition matrix $P \otimes P$. The assumptions made on $P$ guarantee that $P \otimes P$ is irreducible, and since $\nu \otimes \nu$ is an invariant probability, it is also positive recurrent. As a result, since $\tau$ is the first hitting time of the diagonal $D = \{ (n,n), \; \; n \in E \}$ by $(C_i)$, we deduce that $\E [ \tau ] < \infty $. Now, defining the continuous-time processes
$$
m^*_t(n) = M^*_{N(t)}(n), \hspace{5mm} \tld m_t^* = \tld M_{N(t)} (\tld n), \hspace{5mm} t \ge 0,
$$
the coupling $(m^*_t(n) , \tld m^*_t)_{t \ge 0}$ has the expected marginal laws and 
$$
\mathbb{P} \Big[ m_t^*(n) \neq \tld m_t^*  \Big] = \mathbb{P} \Big[ N(t) < \tau  \Big]
= \E \Big[ e^{-t} \sum_{k=0}^{\tau -1} \frac{t^k}{k!} \Big].
$$
This tends to zero as $t$ goes to infinity by dominated convergence. Regarding the integrability in time, simple calculations give
$$
\int_0^\infty \mathbb{P} \Big[ m_t^*(n) \neq \tld m_t^*  \Big]  dt = \sum_{k=0}^\infty (k+1) \mathbb{P}(\tau > k) = \frac{1}{2} \E [ \tau(\tau+1) ].
$$
The mixing Assumption \ref{ass3} is therefore satisfied as soon as the hitting time $\tau$ admits a second moment. This holds for instance if the state space $E$ is finite: it is indeed a well-known fact that hitting times of Markov chains with finite state space have exponential tails (see e.g \mycite{roch}, Lemma 3.25). Another possible setup is the one where the transition matrix is given by $P(n,n') = \nu(n')$, that is when $M_{i+1}$ is drawn on $E$ according to the law $\nu$, independently of $M_i$.
\\
In the general case, the generator of $(m_t)_{t \ge 0}$ is given by
$$
M\psi(n) = \sum_{n' \in E} P(n,n') \psi(n') - \psi(n)
$$
and we derive that $M^{-1} I = (M^{-1}I(n))_{n \in E}$ satisfies the linear system
$$
\forall n \in E, \; \; \; \sum_{n' \in E} P(n,n') M^{-1}I(n') - M^{-1}I(n) = n.
$$
When the state space $E$ is finite, it is clear that $M^{-1}I$ satisfies the Assumption \ref{ass4}. In the case where $P(n,n') = \nu(n')$, since $\nu$ is centered, we get simply
$$
M^{-1}I(n) = -n.
$$

\end{subsection}

\end{section}

\begin{section}{Solutions of the kinetic system} \label{3}

For $k \ge 0$, let us introduce the Banach space $({G^k}, \| \cdot \|_{{G^k}})$ defined by
\begin{align}
{G^k} := \left\{ f \in L^1_{x,v}, \; \; \| f \|_{{G^k}} := \int \int (1 + |v|^k) |f(x,v)| dx dv < \infty \right\}.
\label{Gk}
\end{align}
When $f \in G_2$, we shall denote
\begin{align*}
\rho(f)(x) = \int f(x) dv, \hspace{7mm}
J(f)(x) = \int v f(x) dv, \hspace{7mm}
K(f) = \int v^2 f(x) dv.
\end{align*}
Since solutions $f^\eps$ of \myref{eq} are of course expected to be probability densities, we also introduce the complete normed space $(G_0^k, \| \cdot \|_{{G^k}})$ defined by
\begin{align}
G_0^k := {G^k} \cap \left\{  f \ge 0, \; \; \int \int f(x,v) dx dv = 1 \right\} .
\label{Gk0}
\end{align}

\begin{subsection}{Path-wise weak solutions} 
Let us start by stating the well-posedness of the PDE system  \myref{eq}.
Taking $\eps=1$ for simplicity, it may be written as
\begin{equation}
\left\{
\begin{array}{l }
\partial_t f + v   \partial_x f + \partial_v \Big[ (m + u - v) f \Big] = 0, \\
\partial_t u - \partial^2_x u = J(f) - \rho u.
\end{array}
\right.
\label{system1}
\end{equation}
We consider a fixed $\omega \in \Omega$, so that a trajectory $m \equiv m(\omega)$ belongs to $D([0,T] ; E)$, space of càdlàg functions taking values in the separable complete space $E$.
\\
Whenever $w \in D([0,T] ; E)$ is given, the solution $f[w]$ to the linear conservation equation
\begin{equation*}
\left\{
\begin{array}{l}
\partial_t f + v   \partial_x f + \partial_v \Big[ (w - v) f \Big] = 0, \\
f_0 \in L^1_{x,v},
\end{array}
\right. 
\end{equation*}
is naturally expressed as
\begin{align}
f_t[w](x,v) = e^t f_0 \circ \Phi_0^t(x,v), \label{char}
\end{align}
where $ \Phi_s^t(x,v) = (X_s^t(x,v) , V_s^t(x,v))$ is the flow associated to the characteristics:
\begin{equation}
\left\{
\begin{array}{l l}
\displaystyle{ \frac{d}{ds} X^t_s = V^t_s, } &  X^t_t(x,v) = x ,
\vspace{2mm}
\\
\displaystyle{ 
 \frac{d}{ds} V^t_s = w_s(X^t_s) - V^t_s, } & V^t_t(x,v) = v,
\end{array}
\right.
\label{flow}
\end{equation}
which is clearly well defined since $w$ is globally Lipschitz-continuous in the space variable.
Note that the fact that $w$ is only càdlàg in time causes no issue since this  differential problem should be thought of in integral form. It is in fact easy to check that, for $ x \in D([0,T] ; \mathbb{R})$,
\begin{align*}
\forall t \in [0,T] , \; \; y_t = \int_0^t x_s ds  \; \; \text{ if and only if } \; \; 
\forall t \in [0,T], \; \;  \frac{dy}{dt^+} = x_t, \; \; \frac{dy}{dt^-} = x_{t^-} .
\end{align*}
In the following, we will sometimes write "full" derivatives $\frac{d}{dt}$ for simplicity, but rigorously speaking one should rather consider left and right derivatives.
We may now state our  well-posedness result.

\begin{proposition} \label{ex_un}  Assume that
\begin{align} 
\int_x \int_v (1 + |v|^4) |f_0|^2 dx dv + \int_x \int_v (1+|v|) | \nabla_{x,v} f_0 | dx dv < \infty, 
\hspace{5mm}
u_0  \in H^\eta_x \text{ for } \eta \in (3/2,2).
\label{hyp_ex_un}
\end{align}
Then, for all $\beta \in (3/2 , \eta)$, there exists a unique couple $(f,u)$ with $u \in C([0,T] ; H^\beta_x)$ solution of the system \myref{system1}
in the sense that
\begin{equation}
\left\{
\begin{array}{l}
 f_t = f_t[m+u] \; \; \text{ as defined in \myref{char}},
\vspace{1mm} \\ 
\displaystyle{
 u_t = S(t) u_0 + \int_0^t S(t-s) \Big[ J(f_s) - \rho_s u_s \Big] ds,
 }
\end{array}
\right.
\label{sense}
\end{equation}
where $S(t) = e^{t \partial^2_x}$ denotes the semigroup associated to the heat equation.
For any $k \ge 0$, assuming $f_0 \in G_0^k$, we have additionally $f \in C([0,T] ; G_0^k)$.
\end{proposition}
The proof of Proposition \ref{ex_un} is based on a rather classical fixed point argument. Since it is not a central aspect of the present work, it is postponed to section \ref{app}. 
\begin{remark} \label{rem_measurable}
The uniqueness of the solutions via an explicit iterative construction method classically guarantees that $\omega \mapsto (f_t(\omega), u_t(\omega)) \in G_0^k \times H^\beta_x$ is ${\cal F}_t = \sigma ( m_s, 0 \le s \le t )$-measurable. 
It is of course to be expected that $(f_t, u_t, m_t)_{t \ge 0}$ defines a Markov process.
We will not however attempt to rigorously prove it, since it is in fact not necessary for the purpose of this work.
\end{remark}

Let us simply justify here that the unique solution $(f,u)$ in the sense of \myref{sense} satisfies some weak formulation of \myref{system1}. In particular, we derive some  estimates that will be useful to rigorously determine the generator of the process $(f_t,u_t,m_t)_{t \ge 0}$ later on.
\vspace{3mm}

In order to properly express the weak formulation of the equation satisfied by $f_t \in {G^k}$, we now introduce the space ${\cal C}_k$ which should be thought of as the natural dual of ${G^k}$ "along" equation \myref{system1}.

\begin{define} \label{norm_k}
For $k \ge 0$ and  $\xi \in C^1(\mathbb{T} \times \mathbb{R})$ let us introduce the norm
\begin{align}
& N_k(\xi) = \Big\| \frac{ \xi}{1 + |v|^k} \Big\|_{L^\infty_{x,v}} + \Big\| \frac{(1 + |v|) \nabla_{x,v} \xi }{1 + |v|^k} \Big\|_{L^\infty_{x,v}} 
\end{align}
and the associated normed space, given by the closure
$$
{\cal C}_k = \overline{ \;  \left\{ \xi \in C^1(\mathbb{T}^d \times \mathbb{R}^d), \; \; N_k(\xi) < \infty \right\} \; } .
$$
\end{define}

\begin{proposition} \label{weak_f}
Let $k \ge 0$ and assume that $f_0 \in {G^k}$. Then for all  $\xi \in {\cal C}_k$, 
\begin{equation*}
\langle f_t, \xi \rangle - \langle f_0 , \xi \rangle = \int_0^t  \blangle f_s,  v\partial_x \xi + (u_s + m_s-v)\partial_v \xi \brangle ds 
\end{equation*}
and the following estimate holds: for all $t,s \in [0,T]$,
\begin{align}
\Big|  \frac{(f_{t+s}, \xi) - (f_t, \xi)}{s} \Big| \le C(k, C^*, \| u \|_{L^\infty_{t,x}} ,T)  N_k(\xi) \| f_0 \|_{{G^k}}.
\end{align}
\end{proposition}

\begin{proof}[Proof]
Let us denote $\Phi_t := \Phi^0_t$ for simplicity. Since
\begin{align*}
 \frac{\langle f_{t+s}, \xi \rangle - \langle f_t, \xi \rangle}{s}  = \int \int f_0(x,v) \frac{\xi \circ \Phi_{t+s}(x,v) - \xi \circ \Phi_t(x,v)}{s} dxdv,
\end{align*}
we want to dominate the integrand
uniformly for $t,t+s \in [0,T]$. For all $t \ge 0$,
\begin{align*}
\frac{\xi \circ \Phi_{t+s}(x,v) - \xi \circ \Phi_t(x,v)}{s} = \frac{1}{s}
\int_t^{t+s} \frac{d}{d\sigma}\xi(\Phi_\sigma(x,v)) d\sigma . 
\end{align*}
with
$$
\frac{d}{dt}\xi(\Phi_t(x,v)) =  V_t (x,v)    \partial_x \xi \circ \Phi_t(x,v)
+  ((u_t+m_t)(X_t(x,v)) - V_t(x,v))    \partial_v \xi \circ \Phi_t(x,v)
$$
which, by design, can be bounded by
$$
(1 + C^* + \| u \|_{L^\infty_{t,x}}) N_k(\xi) (1+ | V_t(x,v) |^k).
$$
From the sub-linearity of the equation of characteristics \myref{flow}, we easily derive
$$
\forall t \in [0,T], \; \;  |V_t(x,v)|^k  \le C(C^*, \| u \|_{L^\infty_{t,x}} ,T) (1 + |v|^k) .
$$
Combining these bounds, we have dominated
\begin{align*}
\Big| f(x,v) \frac{\xi \circ \Phi_{t+s}(x,v) - \xi \circ \Phi_t(x,v)}{s} \Big|
\end{align*}
 uniformly for $t,t+s \in [0,T]$ by an expression of the form
$$
C(k,C^*, \| u \|_{L^\infty_{t,x}} ,T) | f_0(x,v) |(1 + |v|^k).
$$
This is exactly the expected estimate, and the dominated convergence theorem leads to
\begin{align*}
\frac{d}{dt} \langle f_t, \xi \rangle & = \int_{x,v}f(x,v) \frac{d}{dt}  \xi \circ \Phi_t(x,v) dx dv
= \blangle f_t, v\partial_x \xi + (u_t + m_t-v)\partial_v \xi \brangle .
\end{align*}
\end{proof}
Similarly, $u$ is a weak solution of the corresponding equation:
\begin{proposition} \label{weak_u}
For all $\xi \in C^2(\mathbb{T})$,
\begin{align*}
\langle u_t, \xi \rangle - \langle u_0, \xi \rangle = \int_0^t \langle u_s, \partial^2_x \xi \rangle  + \blangle J(f_s) - \rho_s u_s, \xi \brangle ds .
\end{align*}
We deduce the estimate: for all $t,s \in [0,T]$,
\begin{align}
\Big| \frac{\langle u_{t+s},\xi \rangle - \langle u_t,\xi \rangle}{s} \Big| \le C( \| u \|_{L^\infty_{t,x}}, \| f_0 \|_{G_1}, T) \Big( \| \xi \|_\infty + \| \partial^2_x \xi \|_\infty \Big) .
\end{align}
\end{proposition}

\begin{proof}[Proof]
The weak formulation is classically derived from the mild form \myref{sense} (\textit{cf.} \cite{Ball77} for example). The estimate is easily deduced.
\end{proof}

\end{subsection}

\begin{subsection}{Estimates uniform in \texorpdfstring{$\eps$}{}}

We now establish some uniform bounds for the solution $(f^\eps, u^\eps)$ of \myref{eq}.

\begin{proposition} \label{moments_eps}
For all $\eps \in (0, 1)$, we have (almost surely)
$$\forall t \in [0,T], \; \;  \| f^\varepsilon_t \|_{L^1_{x,v}} = \| f^\eps_0 \|_{L^1_{x,v}} = 1
$$
and
\begin{align*}
& \sup_{t \in [0,T]} \| f^\eps_t \|_{G_1} \le C(T,C^*, \| u^\eps_0 \|_{L^\infty_x}, \| f^\eps_0 \|_{G_1}),  \\
& \sup_{t \in [0,T]}  \| u^\eps_t \|_{L^\infty_x} \le C(T,C^*, \| u^\eps_0 \|_{L^\infty_x}, \| f^\eps_0 \|_{G_1}).
\end{align*}
\end{proposition}

\begin{proof}
The conservation of the $L^1$ norm of $f^\eps$ is clear from the transport form \myref{sense}. Assuming $f_0^\eps \ge 0$, $f_0^\eps \in C^\infty_c$ (we conclude in the general case by density), we get
\begin{align*}
\frac{d}{dt} \int_x \int_v |v| f^\eps_t(x,v) dx dv = - \eps^{-2} \int_x \int_v |v| \partial_v \Big[ (m^\eps_t + \eps u^\eps_t -v) f^\eps_t \Big]  dx dv
\end{align*}
hence denoting $\bar J(f) = \int_x \int_v |v| f dx dv$, (so that $\| f \|_{G_1} = \| f \|_{L^1_{x,v}} + \bar J(f)$) we get
\begin{align*}
\frac{d}{dt} \bar J(f^\eps_t) \le \eps^{-2} ( \| m^\eps_t \|_{L^\infty_x} +  \eps \| u^\eps_t \|_{L^\infty_x} ) \  - \eps^{-2} \bar J(f^\eps_t)
\end{align*}
and Grönwall's lemma gives
\begin{align*}
\bar J(f^\eps_t) & \le e^{- \eps^{-2} t}\bar J(f^\eps_0)+ \int_0^t \eps^{-2} e^{-\eps^{-2} (t-s)} 
( \| m^\eps_s \|_{L^\infty_x} + \eps \| u^\eps_s \|_{L^\infty_x}) ds .
\end{align*}
We may note that
$
\int_0^t \eps^{-2} e^{-\eps^{-2}s} ds = 1-e^{-\eps^{-2} t} \le 1
$
so that we get
\begin{align}
\bar J(f^\eps_t)\le \bar J(f^\eps_0) + (C^* + \eps \sup_{s \in [0,t]} \| u_s^\eps \|_{L^\infty_x}).
\label{estim_J}
\end{align}
On the other hand, using the mild form \myref{sense} for $u^\eps$, usual heat-semigroup estimates, and making use of the embedding $H_x^{1/2+\delta} \subset L^\infty_x$ for any $\delta > 0$, we have
\begin{align*}
\| u^\eps_t \|_{L^\infty_x} & \le \| u^\eps_0 \|_{L^\infty_x} + C \int_0^t |t-s|^{-{1/2-\delta}} \| J(f^\eps_s) - \eps \rho^\eps_s u^\eps_s \|_{L^1_x} ds \\
& \le \| u^\eps_0 \|_{L^\infty_x}  + C \int_0^t | t-s|^{-1/2 - \delta} \| J(f^\eps_s) \|_{L^1_x} + C \eps  \int_0^t |t-s|^{-1/2 - \delta} \| u_s^\eps \|_{L^\infty_x} ds \\
& \le \| u_0^\eps \|_{L^\infty_x} + C(  \| J(f_0^\eps) \|_{L^1_x} +C^*)
+  C \eps  \int_0^t |t-s|^{-1/2-\delta} \sup_{\sigma \in [0,s]} \| u^\eps_\sigma \|_{L^\infty_x} ds,
\end{align*}
using \myref{estim_J}.
We finally get an estimate of the form 
\begin{align*}
\sup_{s \in [0,t]} \| u^\eps \|_{L^\infty_x} \le C(\| u^\eps_0 \|_{L^\infty_x}, \| f_0^\eps \|_{G_1},C^*) \Big( 1 + \eps \int_0^t |t-s|^{-1/2-\delta} \sup_{\sigma \in [0,s]} \| u^\eps_\sigma \|_{L^\infty_x} ds \Big)
\end{align*}
and deduce the desired estimates using a singular Grönwall inequality.
\end{proof}

The higher $v$-moments for $f$ are obtained by induction using the same arguments. For the sake of brevity, we omit the proof.
\begin{proposition}[Higher moments for $f$] For all $k \ge 1$, we have (almost surely)
\begin{align*}
\sup_{t \in [0,T]} \| f^\eps_t \|_{{G^k}} \le C(k,T,C^*, \|u_0^\eps\|_{L^\infty_x}, \|f_0^\eps\|_{{G^k}}).
\end{align*}
\end{proposition}

Let us finally state some compactness estimate for $u^\eps$.

\begin{proposition}[Compactness for $u$] \label{u_compact1}
For $\alpha \in (1/2, 3/2)$, we have (almost surely)
\begin{align*}
& \| u^\eps_t \|_{H_x^{\alpha}} \le C(\alpha, T,C^*, \| u^\eps_0 \|_{H_x^{\alpha}}, \| f^\eps_0 \|_{G_1})
\\
& \forall \beta < \alpha \; \; \; \| u^\eps_t - u^\eps_s \|_{H_x^{\beta}} \le C(\alpha, \beta, T,C^*, \| u^\eps_0 \|_{H_x^{\alpha}}, \| f^\eps_0 \|_{G_1}) \int_s^t \theta(\sigma) d\sigma,
\end{align*}
where $\theta \in L^1([0,T])$ is independent of $\eps$. 
\end{proposition}
\begin{proof}
Consider the mild form
\begin{align}
& u^\eps_t = S(t)u_0^\eps + \int_0^t S(t-\sigma) g_\sigma^\eps d\sigma, \; \; \text{ where } g^\eps_\sigma = J(f_\sigma^\eps) - \eps \rho_\sigma^\eps u_\sigma^\eps .
\label{mildy}
\end{align}
The embedding $L^1_x \subset H^{-1/2-\delta}_x$ for all $\delta > 0$, and a classical regularity estimate for the heat kernel $S(t)$ lead to
\begin{equation}
\| S(t) w \|_{H^\alpha_x} \lesssim (1 + t^{-(\alpha+1/2 + \delta)/2}) \| w \|_{L^1_x},
\label{classic}
\end{equation}
from which we derive
$$
\| u^\eps_t \|_{H^{\alpha}_x} \lesssim \| u_0^\eps \|_{H^\alpha} +  \int_0^t (1 + |t-\sigma|^{-(\alpha + 1/2 + \delta)/2}) \; \| g^\eps_\sigma \|_{L^1_x} d\sigma.
$$
Using Proposition \ref{moments_eps}, we have $\| g_t^\eps \|_{L^1_{x,v}} \le C(T,C^*, \| u^\eps_0 \|_{L^\infty}, \| f^\eps_0 \|_{G_1})$, which leads to the first estimate, since $ (\alpha + 1/2 + \delta) / 2 < 1$ for $\delta$ small enough.
\\
For the second estimate, we first write, for all $\xi \in H^{-\beta}_x$,
\begin{align*}
\Big| \langle S(t)u_0^\eps - S(s)u_0^\eps, \xi \rangle \Big| & = \Big| \int_s^t \blangle S(\sigma) u_0^\eps , \partial^2_x \xi \brangle d\sigma \Big| \le \| \xi \|_{H^{-\beta}_x} \int_s^t \| S(\sigma) u_0^\eps \|_{H^{2+\beta}_x} d\sigma \\
&
\le C \| \xi \|_{H^{-\beta}_x} \| u_0^\eps \|_{H^{\alpha}_x}  \int_s^t \Big( 1 +\sigma^{-(2+\beta - \alpha)/2} \Big) d\sigma,
\end{align*}
which gives $\| S(t) u_0^\eps - S(s) u_0^\eps \|_{H^\beta_x} \le \| u_0^\eps \|_{H^{\alpha}_x} \int_s^t \theta_1(\sigma) d\sigma$ with $\theta_1 \in L^1$ since $(2+\beta - \alpha)/2 < 1$. 
Finally, 
\begin{align*}
\Big\| \int_s^t S(t-\sigma) g_\sigma^\eps d\sigma \Big\|_{H^\beta_x} \le \int_s^t \| S(t-\sigma) g_\sigma^\eps d\sigma \|_{H^\beta_x} d\sigma
\end{align*}
and we use again \myref{classic} with $\beta$ instead of $\alpha$ to conclude from the mild form \myref{mildy}.
\end{proof}

\end{subsection}

\begin{subsection}{Generator of the process}

As previously mentioned, $(f^\varepsilon_t, u^\eps)_{t \ge 0}$ is not a Markov process in itself, however the process $(f^\varepsilon_t, u^\eps, m^\varepsilon_t)_{t \ge 0}$ is. In view of \myref{eq}, its generator is naturally expected to be
\begin{equation}
{\cal L}^\varepsilon \psi= \frac{1}{\varepsilon^2} {\cal L}_\# \psi + \frac{1}{\varepsilon} {\cal L}_\flat \psi
+ {\cal L}_0 \psi + \eps {\cal L}_1 \psi,
\end{equation}
where
\begin{equation}
\left\{
\begin{array}{l}
 {\cal L}_\# \psi(f,u,n) = - \blangle \partial_v[(n-v)f ] , D_f \psi(f,u,n) \brangle + M \psi(f,u,n), \\
 {\cal L}_\flat \psi(f,u,n) = - \blangle v \partial_x f + u \partial_v f , D_f \psi(f,u,n) \brangle ,\\
 {\cal L}_0 \psi(f,u,n) = \blangle \partial^2_x u + J(f), D_u \psi(f,u,n) \brangle, \\
 {\cal L}_1 \psi(f,u,n) = - \blangle \rho u , D_u \psi(f,u,n) \brangle.
\end{array}
\right.
\label{generator}
\end{equation}
Let us quickly get into the details of the notation $"D"$ used here.
For $\psi : {G^k} \to \mathbb{R}$, let $\text{Diff}_f \psi$ denote the usual differential in $f \in {G^k}$ of $\psi$, that is \begin{align*}
\forall f \in {G^k}, \; \; \; \text{Diff}_f \psi(f) \in (G^k)' = L^1\Big( \mathbb{T}^d \times \mathbb{R}^d , (1+|v|^k)dxdv \Big)' \simeq L^\infty_{x,v} .
\end{align*}
There exists $h(f) \in L^\infty_{x,v}$ such that 
\begin{align*}
\forall g \in {G^k}, \; \; \text{Diff}_f \psi(f)(g) = \blangle g, h(f) \brangle_{{G^k}, (G^k)'} = \blangle g , (1+|v|^k) h(f) \brangle.
\end{align*}
The differential $\text{Diff}_f\psi(f)$ is naturally identified with the $L^\infty_{x,v}$ function $h(f)$. For convenience, we shall rather denote by $D_f\psi(f)$ the function $(1+|v|^k)h(f)$, so that
$$
\forall g \in {G^k}, \; \; \; \text{Diff}_f \psi(f)(g) = \blangle g , D_f\psi(f)  \brangle.
$$
To avoid cluttering notation, we do not indicate the dependence of $D_f\psi(f)$ on the choice of the parameter $k \ge 1$. In section~\ref{5}, where the perturbed test-function method is developed, the value $k=3$ is fixed for good.
The natural bound on $D_f \psi$ is consequently
$$
 \Big\| \frac{D_f\psi(f)}{1+|v|^k} \Big\|_{L^\infty_{x,v}}  = \| \text{Diff}_f \psi(f) \|_{{G^k}'}   < \infty .
$$
 Similarly, for $\psi : H^\beta_x \to \mathbb{R}$, $D_u \psi$ is defined so that
$$
\forall w \in H^\beta_x, \; \; \; \text{Diff}_u \psi(u)(w) = \blangle w , D_u\psi(u)  \brangle.
$$

We will not properly identify the complete domain of the infinitesimal generator ${\cal L}^\varepsilon$: for the purpose of this paper, it will in fact be sufficient to show that ${\cal L}^\eps$ satisfies the martingale problem associated with $(f^\eps, u^\eps,m^\eps)$ for a large enough class of test functions.

\begin{define}[Good test function] 
Let $k \ge 1$ and $\alpha \in (1/2, 3/2)$ be given. Recalling the space $G_0^k$ introduced \myref{Gk0}, define the state space for the process 
$(f_t^\eps,u_t^\eps,m_t^\eps)_{t \ge 0}$:
\begin{align}
{\cal X}_k^\alpha = G^k_0 \times H^\alpha_x \times E.
\end{align}
A function 
$\psi : (f,u,n) \in G^k \times H^\alpha_x \times E \to \mathbb{R}$
is said to be a "good test function" on ${\cal X}_k^\alpha$ when: 
\begin{enumerate}
\item $\psi \in \B^{loc}({\cal X}_k^\alpha)$ 

\item For all $(f,u) \in {G_0^k} \times H^\alpha_x$, $\psi(f,u, \cdot) \in D(M)$ and $M \psi \in C_b^{loc}({\cal X}_k^\alpha)$

\item For all $(u,n) \in H^\alpha_x \times E$, $\psi( \cdot ,u,n)$ is ${G^k}$-differentiable and
$$
\begin{array}{l c l l}
D_f \psi :&  {\cal X}_k^\alpha & \longrightarrow & ({\cal C}_k, N_k) \\
& (f,u,n) & \longmapsto & D_f \psi(f,u,n)
\end{array} \text{ is } C_b^{loc}({\cal X}_k^\alpha ; {\cal C}_k),
$$
where the normed space $({\cal C}_k, N_k)$ has been introduced in Definition \myref{norm_k}

\item For all $(f,n) \in {G_0^k} \times E$, $\psi(f, \cdot ,n)$ is $H^\alpha_x$-differentiable and
$$
\begin{array}{l c l l}
D_u \psi :&  {\cal X}_k^\alpha & \longrightarrow & C^2(\mathbb{T}) \\
& (f,u,n) & \longmapsto & D_u \psi(f,u,n)
\end{array} \text{ is } C_b^{loc}({\cal X}_k^\alpha ; C^2(\mathbb{T})).
$$
\end{enumerate}
\label{good}\end{define}
Note that the local boundedness required in point 3. can be expressed as
$$
\left\{
\begin{array}{l}
|D_f\psi(f,u,n)| \le C \Big(\| f \|_{{G^k}}, \| u \|_{H^\alpha_x} \Big)(1+ |v|^k),  \\
|v \partial_x D_f\psi(f,u,n)| \le  C \Big(\| f \|_{{G^k}}, \| u \|_{H^\alpha_x} \Big)(1+ |v|^k), \\
|(1+|v|)  \partial_v D\psi(f,n)| \le  C \Big(\| f \|_{{G^k}}, \| u \|_{H^\alpha_x} \Big)(1+ |v|^k),
\end{array}
\right.
$$
which is a natural condition to ensure that ${\cal L^\eps \psi}$ is well-defined and $C_b^{loc}( {\cal X}_k^\alpha)$. This suggests that, when $f \in {G_0^k}$, typical good test functions involve moments of order at most $k-1$ for $f$. Later on, we will for instance consider the test function
\begin{align*}
\psi(f,n) = \blangle J_{k-1}(f), ( \partial_x M^{-1}I(n)) \xi \brangle, \; \; \; \xi \in W^{1,\infty}_x,
\end{align*}
and linear combinations of such functions.
Introducing the canonical filtration of $m^\eps$
$$
{\cal F}^\eps_t = \sigma \Big( m^\eps_s, \; 0 \le s \le t \Big) = {\cal F}_{\eps^{-2} t}, \; \; \; t \ge 0,
$$
the process $t \mapsto (f_t^\eps, u_t^\eps) \in {G_0^k} \times H^\alpha_x$ is naturally $({\cal F}^\eps_t)$-adapted, as mentioned in Remark \ref{rem_measurable}. 
We can now state the following result:

\begin{proposition} \label{prop_gen} Let $k \ge 1$, and $\alpha \in (1/2,3/2)$ be given.
Let  $\psi \equiv \psi(f,u,n)$ be a good test function on ${\cal X}_k^\alpha$. 
 Then $\psi$ satisfies the martingale problem for ${\cal L}^\eps$ in the sense that
$$
\forall (f,u,n) \in {\cal X}_k^\alpha, \;  \; {\cal L}^\varepsilon \psi(f,u,n) \text{ is well-defined}, \; \; 
{\cal L}^\varepsilon \psi \in C_b^{loc}({\cal X}_k^\alpha)
$$
and for any $(f^\eps_0,u^\eps_0,n) \in {\cal X}_k^\alpha$ (with $f_0^\eps, u_0^\eps$ satisfying the assumption \myref{hyp_ex_un}), 
$$
M^\eps_\psi(t) = \psi(f^\eps_t, u^\eps_t,m^\eps_t) - \int_0^t {\cal L}^\eps \psi(f^\eps_s,u^\eps_s,m^\eps_s) ds, 
\; \; \; t \ge 0
$$
defines a càdlàg, $L^2$ martingale with respect to the filtration $({\cal F}^\eps_t)_{t \ge 0}$  under 
$\mathbb{P}_{(f^\eps_0,u^\eps_0,n)}$. 
\\
Moreover, if $|\psi|^2$ is also a good test function, the predictable quadratic variation is given by:
$$
\Big[ M_\psi^\eps \Big](t) = \frac{1}{\eps^2} \int_0^t  \Big( M |\psi|^2 - 2 \psi M \psi \Big)(f^\eps_s,u^\eps_s,m^\eps,s) ds .
$$

\end{proposition}

\begin{remark}
The predictable quadratic variation $\Big[ M \Big](t)$ is the unique càdlàg, predictable, increasing process such that $|M(t)|^2 - \Big[ M \Big](t)$ defines a martingale. When $M(t)$ is moreover continuous, this coincides with the standard quadratic variation (see \mycite{JacodShiryaev03}, p.38).
\end{remark}

\begin{proof}[Proof] 
The assumptions on $\psi$ clearly guarantee that ${\cal L}^\varepsilon \psi$ is well defined and $C_b^{loc}({\cal X}_k^\alpha)$. 
To simplify, we shall consider the case $\varepsilon = 1$ and write $(f_t,u_t,m_t)$ instead of $(f^\eps_t,u^\eps_t,m^\eps_t)$.
The estimates established in the previous subsections show that, $(f_0,u_0)$ being fixed,
\begin{align*}
\forall T > 0, \; \forall t \in [0,T], \; \; \; \| f_t \|_{{G^k}} + \|u_t\|_{H^\alpha_x} \le C(T), \; \; \;  \; \; \; \; 
\|m_t\|_E \le C^*,
\end{align*}
almost surely under the probability $\mathbb{P}_{(f^\eps_0,u^\eps_0,n)}$. If follows that $M^\eps_\psi(t)$ is well defined and in $L^\infty(\Omega)$ for $t \ge 0$. 
We now need to prove that, for $s < t$ and $h$ bounded and ${\cal F^\eps}_s-$measurable,
\begin{align}
\E \Big[ \Big( \psi(f_t,u_t,m_t) - \psi(f_s,u_s,m_s) - \int_s^t {\cal L}^\eps \psi(f_\sigma,u_\sigma,m_\sigma) d\sigma \Big) h \Big] = 0 . \label{mart}
\end{align}
To this purpose, let us formulate two natural lemmas:
\begin{lemma} \label{le_lemma} We have
\begin{align*}
 \frac{d}{dt} \psi(f_t,u,n) & = "\blangle D_f \psi(f_t,u,n) , \partial_t f_t  \brangle" 
\\ &
:= \blangle f_t, v  \partial_x D_f\psi(f_t,u,n) +  (u_t + m_t - v)   \partial_v D_f\psi(f_t,u,n) \brangle.
\end{align*}
\end{lemma}

\begin{proof}[Proof of the lemma]
One can simply write
\begin{align*}
\frac{\psi(f_{t+s},u,n) - \psi(f_t,u,n)}{s} & = 
\blangle D_f\psi(f_t,u,n) , \frac{f_{t+s} - f_t}{s} \brangle + R_s
\end{align*}
with
\begin{align*}
R_s = \int_0^1 \blangle D_f\psi(f_t + \theta (f_{t+s} - f_t) , u,n) - D_f\psi(f_t,u,n)  , \frac{f_{t+s} - f_t}{s}\brangle d\theta .
\end{align*}
Since $D_f \psi(f_t,u,n) \in {\cal C}_k$, Proposition \ref{weak_f} guarantees that
\begin{align*}
\blangle D_f\psi(f_t,u,n) , \frac{f_{t+s} - f_t}{s} \brangle \xrightarrow[s \to 0]{}"\blangle D_f \psi(f_t,u,n) , \partial_t f_t  \brangle"
\end{align*}
and that
\begin{align*}
| R_s | \lesssim \sup_{\theta \in [0,1]} N_k\Big( D_f \psi(f_t +\theta(f_{t+s}-f_t),u,n) - D_f \psi(f_t,u,n) \Big)
\end{align*}
which tends to $0$ as $s \to 0$, since $f_{t+s} \to f_t$ in ${G_0^k}$ and $D_f \psi$ is continuous.
\end{proof}
This second lemma is proved in a similar fashion.
\begin{lemma} We have
\begin{align*}
\frac{d}{dt} \psi(f,u_t,n) & = "\blangle D_u \psi(f,u_t,n) , \partial_t u_t  \brangle"  \\
& := \blangle u_s, \partial_x^2 D_u \psi(f,u_t,n) \brangle + \blangle J(f_t)-\rho_t u_t, D_u \psi(f,u_t,n) \brangle.
\end{align*}
\end{lemma}

Introducing a subdivision $t_0=s < t_1 < ... < t_N = t$ of $[s,t]$ of step $\delta = \max_i |t_{i+1}-t_i|$, we may split \myref{mart} into 3 terms:
\begin{align*}
& A = \E \Big[ \Big( \sum_i \psi(f_{t_i},u_{t_i},{m_{t_{i+1}}}) - \psi(f_{t_i},u_{t_i},{m_{t_{i}}})
- \int_{t_i}^{t_{i+1}}  M \psi(f_\sigma,u_\sigma,m_\sigma) d\sigma \Big)  h \Big],
\\
& B = \E \Big[ \Big( \sum_i \psi(f_{t_{i+1}},u_{t_{i+1}},{m_{t_{i+1}}}) - \psi(f_{t_i},u_{t_{i+1}},{m_{t_{i+1}}})
- \int_{t_i}^{t_{i+1}}  \blangle D_f\psi(f_\sigma,u_\sigma,m_\sigma) , \partial_\sigma f_\sigma \brangle d\sigma \Big) h \Big] ,
\\
& C = \E \Big[ \Big( \sum_i \psi(f_{t_{i}},u_{t_{i+1}},{m_{t_{i+1}}}) - \psi(f_{t_i},u_{t_{i}},{m_{t_{i+1}}})
- \int_{t_i}^{t_{i+1}}  \blangle D_u\psi(f_\sigma,u_\sigma,m_\sigma) , \partial_\sigma u_\sigma \brangle d\sigma \Big) h \Big].
\end{align*}
Let us treat the first term $A$: conditioning with respect to ${\cal F}_{t_i}^\eps$ and using the Markov property, 
\begin{align*}
\E \Big[ \Big(\psi(f_{t_i},u_{t_i},{m_{t_{i+1}}}) - \psi(f_{t_i},u_{t_i},{m_{t_{i}}}) \Big) h \Big]
= \E \Big[ \Big( P_{t_{i+1}-t_i} [ \psi(f_{t_i},u_{t_i}, \cdot ) ](m_{t_i}) - \psi(f_{t_i},u_{t_i},{m_{t_{i}}}) \Big) h \Big].
\end{align*}
We note that, for all $t \ge 0$ and $(f,u,n) \in {\cal X}_k^\alpha$,
$$
P_t[\psi(f,u,.)](n) - \psi(f,u,n)  = \int_0^t P_\sigma M \psi(f,u,n) d\sigma = 
\E_n  \int_0^t M \psi(f,u,m_\sigma) d\sigma,
$$
where we have used Kolmogorov's forward equation from Proposition \ref{chap1-kolmo} and the interversion $\E \int = \int \E$ is justified since $M \psi$ is locally bounded.
We are hence led to
\begin{align*}
\E \Big[ \Big(\psi(f_{t_i},u_{t_i},{m_{t_{i+1}}}) - \psi(f_{t_i},u_{t_i},{m_{t_{i}}}) \Big) h \Big]
& = \E \Big[ \Big(  \E_{m_{t_i}} \int_0^{t_{i+1}-t_i} M \psi(f_{t_i},u_{t_i},m_\sigma) d\sigma  \Big) h \Big] \\
& = \E \Big[ \Big( \int_0^{t_{i+1}-t_i} M\psi(f_{t_i},u_{t_i},m_{t_i + \sigma}) d\sigma \Big) h \Big] \\
& =  \E \Big[ \Big( \int_{t_i}^{t_{i+1}} M\psi(f_{t_i},u_{t_i},m_{\sigma}) d\sigma \Big) h \Big] 
\end{align*}
where we have once again used the Markov property. It follows that $A$ can be rewritten as
$$
A=\E \Big[ \Big (\int_s^t a_\delta(\sigma) d\sigma \Big) h \Big]
$$
with
$$
a_\delta(\sigma) = \sum_i \mathbb{1}_{[t_i,t_{i+1}]}(\sigma) \Big( M \psi(f_{t_i},u_{t_i},m_\sigma) - M \psi(f_\sigma, u_\sigma, m_\sigma) \Big).
$$
For fixed $\sigma \in [s,t]$, since $t \mapsto (f_t,u_t) \in {G_0^k} \times H^\alpha_x$ and $M\psi$ are continuous, it is clear that, as $\delta \to 0$, we have
$
a_\delta(\sigma) \to  0
$ almost surely (hence in probability).
Moreover, $M\psi$ being locally bounded, we get $a_\delta \in L^\infty([s,t] \times \Omega)$ with a bound uniform in $\delta$. In particular, $\E[a_\delta(\sigma) h]$ tends to $0$ and is uniformly integrable in $\sigma \in [s,t]$, so that we may conclude $A \xrightarrow[\delta \to 0]{} 0$. 
\vspace{3mm}

Using the two lemmas introduced earlier, one can write in a similar fashion
\begin{align*}
B = \E \Big[ \Big( \int_s^t b_\delta(\sigma) d\sigma \Big) h \Big], \; \; \; \; 
C = \E \Big[ \Big( \int_s^t c_\delta(\sigma) d\sigma \Big) h \Big],
\end{align*}
with
\begin{align*}
& b_\delta(\sigma) = \sum_i \mathbb{1}_{[t_i,t_{i+1}]}(\sigma) 
\blangle D_f\psi(f_\sigma,u_{t_{i+1}},m_{t_{i+1}}) - D_f\psi(f_\sigma,u_\sigma,m_\sigma) , \partial_\sigma f_\sigma \brangle \\
& c_\delta(\sigma) = \sum_i \mathbb{1}_{[t_i,t_{i+1}]}(\sigma) 
\blangle D_u\psi(f_{t_i},u_\sigma,m_{t_{i+1}}) - D_u\psi(f_\sigma,u_\sigma,m_\sigma) , \partial_\sigma u_\sigma \brangle
\end{align*}
and conclude in the same way using the estimates from Propositions \ref{weak_f} and \ref{weak_u} and the fact that $(m_t)_{t \ge 0}$ is stochastically continuous and $D_f\psi$, $D_u\psi$ are $C_b^{loc}({\cal X}_k^\alpha)$. 
\vspace{3mm}

If $| \psi|^2$ is also a good test function, then ${\cal L}^\eps |\psi|^2$ is locally bounded, and $M^\eps_{|\psi|^2}(t)$ defines a martingale. Using these facts, (we may refer to the proof of Theorem B.3 in \cite{dv2} for details) one can show that the predictable quadratic variation of $M^\eps_\psi$ is given by
\begin{align*}
d[M^\eps_\psi](t) = \Big( {\cal L}^\eps |\psi|^2 - 2 \psi {\cal L}^\eps \psi \Big) (f^\eps_t,u^\eps_t,m^\eps_t) dt
\end{align*}
which, in our case, boils down to $\eps^{-2} (M |\psi|^2 - 2 \psi M \psi)$ since ${\cal D}|\psi|^2 = 2 \psi {\cal D} \psi$ for any first order differential operator ${\cal D}$.
\end{proof}

\end{subsection}

\end{section}

\begin{section}{The auxiliary process} \label{4}

In section \ref{5}, we will be led to solve Poisson equations relative to the partial generator ${\cal L}_\#$ defined in \myref{generator}. To this purpose, it is natural to study the auxiliary Markov process $(g_t , m_t)_{t \ge 0}$ whose generator is expected to be ${\cal L}_\#$.  We hence define $g_t \equiv g_t(f,n)$ as the solution of
\begin{equation}
\left\{
\begin{array}{l}
\displaystyle{  \partial_t g_t + \partial_v[ (m_t(n) - v) g_t ] = 0 }, 
\vspace{1mm} \\ 
 g_0(f,n) = f , \; \; \; m_0(n) = n ,
 \end{array}
 \right.
 \label{kinetic g}
\end{equation}
and let $(Q_t)_{t \ge 0}$ be the corresponding semigroup
$$
\forall \psi \in \B^{loc}({G_0^k} \times E), \; \; Q_t \psi(f,n) = \E[ \psi(g_t(f,n), m_t(n))  ] .
$$
The characteristics associated with the conservation equation \myref{kinetic g} solve
$$ \dot V_t(v)  = m_t(x) - V_t(v),$$
leading to the explicit expression
$$
V_t(v) =e^{-t} \Big( v + \int_0^t e^u m_u(x) du \Big).
$$

\begin{define}
The solution of $\myref{kinetic g}$ is defined by
\begin{align}
 g_t(f,n)(x,v) =e^t f \Big( x, e^t \Big[ v - w_t(n)(x) \Big] \Big) , \; \; t \ge 0 \label{def_g}
 \end{align}
 where
\begin{align}
 w_t(n) = \int_0^t e^{-(t-s)}m_s(n) ds = \int_0^t e^{-s} m_{t-s}(n) ds \in W^{2,\infty}_x .
 \label{w_t}
\end{align}
\end{define}
It is easy to verify that, for all $f \in {G_0^k}$, we have $g(f,n) \in C([0,T] ; {G_0^k})$ with a bound
$$
\| g_t(f,n) \|_{{G^k}} \le C(T, C^*, \| f \|_{{G^k}} ).
$$
Additionally, $g_t(f,n)$ defines a weak solution of \myref{kinetic g} just as in Proposition \ref{weak_f}.

\begin{subsection}{Generator of the auxiliary process}
With the explicit expression \myref{def_g}, one could easily prove that the auxiliary process $(g_t,m_t)_{t \ge 0}$ rigorously defines a (locally bounded) Markov process on the state space ${G_0^k} \times E$. 
In order to solve Poisson equations later, it is of interest this time to truly characterize the generator in terms of b.p.c convergence.
\begin{proposition} \label{generator_g}
Let $k \ge 1$ be given, and
let  $\psi \equiv \psi(f,n)$ be a good test function on ${G_0^k} \times E$. Then $\psi \in D({\cal L}_\#)$ in the sense that ${\cal L}_\# \psi \in C_b^{loc}({G_0^k} \times E)$ and
$$
\frac{Q_t \psi (f,n) - \psi(f,n)}{t}
\xrightarrow[t \to \, 0^+]{ b.p.c} {\cal L}_\#\psi(f,n) .
$$
\end{proposition}

\begin{proof}[Proof]
The properties satisfied by the good test function $\psi$ clearly guarantee that ${\cal L}_\# \psi$ is well defined and $C_b^{loc}({G_0^k} \times E)$.
Let us denote $(g_t(f,n),m_t(n))$ by $(g_t,m_t)$ for simplicity. We split the ratio into two terms 
$$
\frac{Q_t \psi (f,n) - \psi(f,n)}{t} =
\E \Big[ \frac{\psi(f,m_t) - \psi(f,n)}{t} \Big] + \E \Big[ \frac{\psi(g_t,m_t) - \psi(f,m_t)}{t}   \Big] . \label{chap1-split}
$$
The first term is simply $t^{-1}(P_t \psi(f, n) - \psi(f,n))$ which converges to $M \psi(f,n)$ and can be bounded as $t \to 0$ by
\begin{align*}
\sup_{t \in [0,1]} \Big| \frac{d}{ds} P_s \psi(f, n) \Big| = \sup_{t \in [0,1]} \Big| P_t M \psi(f,n) \Big|
\end{align*}
which is locally bounded in $(f,n)$. For the second term, we write
\begin{equation}
  \frac{\psi(g_t,m_t) - \psi(f,m_t)}{t}  = \Big \langle D_f \psi(f,n) , \frac{g_t -f}{t} \Big \rangle + R_t
\label{ratio1}
\end{equation}
 with
 $$
 R_t = \int_0^1 \blangle D_f\psi(f+\theta(g_t-f),m_t) - D_f \psi(f,n) , \frac{g_t - f}{t} \brangle d\theta.
 $$
As in the proof of Lemma \ref{le_lemma}, we see that \myref{ratio1} converges to $\Big( D_f \psi(f,n) , - \partial_v [ (n-v) f ] \Big)$ and is uniformly locally bounded in $(f,n)$ as $t \to 0$. We conclude by dominated convergence.
\end{proof}

\end{subsection}

\begin{subsection}{Limiting measures and mixing properties}

We will now determine to some extent the limiting law of $(g_t, m_t)(f,n)$ as $t$ goes to infinity, which naturally corresponds to an invariant measure for the semi-group $(Q_t)_{t \ge 0}$. More precisely, we shall describe how the mixing assumption \myref{mixing} on the driving process $(m_t)_{t \ge 0}$ impacts the speed of convergence of $Q_t$ to this invariant measure, along a certain type of test functions. From the explicit form \eqref{def_g} for example, we see that 
$$
\rho(g_t(f,n))=\rho(f),
$$
for all $t\geq 0$. The density $\rho$ being invariant along the trajectory $(g_t, m_t)(f,n)$, the limiting law (which we define in \eqref{defmurho}) will therefore be naturally parametrized by $\rho$.

\begin{define}
For $\rho : \mathbb{T} \to \mathbb{R}_+$ with $\int \rho(x) dx = 1$ and $w : \mathbb{T} \to \mathbb{R}$, we define the probability measure $\rho \otimes \delta_{w}$ on $\mathbb{T} \times \mathbb{R}$ as
$$
\forall \xi \in C_b(\mathbb{T} \times \mathbb{R}),  \; \;
 \langle \xi , \rho \otimes \delta_{ w} \rangle = \int_{x} \xi(x,  w(x)) \rho(x) dx,
$$
\end{define}
Keeping in mind expression \myref{def_g}, we easily have
$$
e^{t} f \Big( x, e^t [v-w(x)] \Big) \xrightarrow[t \to \infty]{} \rho(x) \otimes \delta_{w(x)}
$$ in the distributions sense. Moreover, we will see that
$w_t(n) = \int_0^t e^{-s}m_{t-s}(n) ds \xrightarrow[t \to \infty]{} 
 \tld w $
in law, where
\begin{align}
\tld w = \int_{-\infty}^0 e^{s} \tld m_{s} ds.
\label{w}
\end{align}
We therefore expect the limiting measures to have the following form.
\begin{define}[Invariant measures] 
For $\rho : \mathbb{T} \to \mathbb{R}_+$ with $\int \rho(x) dx = 1$, we (partially) define the probability measure $\mu_\rho$ on ${\cal P}(\mathbb{T} \times \mathbb{R}) \times E$ by
\begin{align}\label{defmurho}
\forall \psi \equiv \psi(f,n), \; \; \; 
\llangle \psi , \mu_\rho \rrangle := \int \psi(f,n) \, d\mu_\rho(f,n)  = \E \Big[ \psi( \rho \otimes \delta_{ \tld w} , \tld m_0 ) \Big] .
\end{align}
when this expression makes sense.
\end{define}
Remaining on a formal level, one may compute the first moments
\begin{align*}
& \rho( \rho \otimes \delta_{\tld w}) := \int \rho \otimes \delta_{\tld w} \, dv = \rho, \\
& J( \rho \otimes \delta_{\tld w}) := \int v  \rho \otimes \delta_{\tld w} \, dv = \tld w \rho, \\
& K( \rho \otimes \delta_{\tld w}) := \int v^2 \rho \otimes \delta_{\tld w} \, dv = \tld w^2 \rho .
\end{align*}
We are now ready to state the following result.
\begin{proposition}[Mixing speed for the auxiliary process] \label{mix}
Let us work with $f \in {G_0^k}$ for $k \ge 3$. Let $\psi$ be a function of the form
\begin{align}
\psi(f,n) = \blangle \xi_1(n) \rho , \xi_2(n) \brangle \; \text{ or } \; \blangle J(f) , \xi_1(n) \brangle \;  \text{ or } \; \blangle K(f) , \xi \brangle
\end{align}
with $\xi_1, \xi_2 : E \to W^{1,\infty}_x$ Lipschitz-continuous functions:
\begin{align*}
\; \; \; \; \;  \| \xi_{i}(n) - \xi_{i}(n') \|_{W^{1,\infty}_x} \lesssim \| n - n' \|_E
\end{align*}
or $\xi \in W^{1,\infty}_x$. 
Then for all $(f,n) \in {G_0^k} \times E$, letting $\rho = \rho(f)$, we have
$$
Q_t \psi(f,n) \xrightarrow[t \to \infty]{} \llangle \psi, \mu_\rho \rrangle.
$$
More precisely, for all $B \subset {G_0^k} \times E$ bounded, 
\begin{align*}
& \sup_{(f,n) \in B} \Big| Q_t \psi(f,n) - \llangle \psi, \mu_\rho \rrangle \Big| \in L^1(\mathbb{R}^+),
\\
& \sup_{(f,n) \in B} N_k \Big(  D_f \Big[ Q_t \psi(f,n) - \llangle \psi, \mu_\rho \rrangle \Big]  \Big) \in L^1(\mathbb{R}^+).
\end{align*}
\end{proposition}

\begin{remark} 
\begin{itemize}
\item Note that in particular, when $\llangle \psi, \mu_\rho \rrangle = 0$, we have 
\begin{align*}
& \int_0^\infty \sup_{(f,n) \in B} \Big| Q_t \psi(f,n) \Big| dt < \infty , \hspace{8mm} \int_0^\infty \sup_{(f,n) \in B} N_k( D_f Q_t \psi(f,n) ) dt < \infty.
\end{align*}
This allows to define $\psi_1(f,n) = \int_0^\infty Q_t\psi(f,n) dt$ which then satisfies
 \begin{align*}
 \psi_1 \in \B^{loc}({G_0^k} \times E) \text{ and } D_f \psi_1 \in C_b^{loc}({G_0^k} \times E ; {\cal C}^1_k)
 \end{align*}
 as in Definition \ref{good} of a good test function.
\item Later on, we will typically consider $\xi_i(n) = n, M^{-1}I(n), \partial_x M^{-1}I(n)$ which are indeed Lipschitz-continuous functions according to the assumptions made on the driving process.
\item The proof of this result will in fact emphasize that it holds for any test function $\psi$ such that $\psi \Big( g_t(f,n),m_t(n) \Big)$ depends "at most quadratically on $n$" with the general rule of thumb: 
$$
\begin{array}{c c c }
\xi(n) \; \;  (\text{ with $\xi$ Lipschitz }) & \longleftrightarrow & n, \\
 \rho & \longleftrightarrow & 1, \\
 J(f) & \longleftrightarrow  & n, \\
  K(f) & \longleftrightarrow & n^2.
\end{array}
$$
For instance, the test function $\psi(f,n) = \blangle \rho, \xi(n) \brangle \blangle J(f), \xi \brangle$ can be considered.
\end{itemize}
\end{remark} 

\begin{proof}[Proof] 
Recalling \myref{def_g}, simple computations lead to
\begin{align*}
& \rho(g_t(f,n)) = \rho(f) = \rho \\
& J(g_t(f,n)) = e^{-t} J(f) + w_t(n) \rho(f) \\
& K(g_t(f,n)) = e^{-2t} K(f) + 2e^{-t}  w_t(n)   J(f) + w_t(n)^2  \rho(f).
\end{align*}
Let us introduce the coupling $(m_t^*(n),\tld m^*_t)_{t \ge 0}$ given by the mixing Assumption \ref{ass3}, and define
\begin{align*}
w^*_t(n) := \int_0^t e^{-s} m^*_{t-s}(n) ds, \; \; \; \;
 \tld w^*_t := \int_{-\infty}^0 e^{s} \tld m^*_{t+s} ds.
\end{align*}
Recalling \myref{w_t} and \myref{w}, one can observe that, since $(\tld m_t)_{t \in \mathbb{R}}$ is stationary,
\begin{align*}
 (w_t(n), m_t(n)) \sim (w^*_t(n), m^*_t(n) ) \; \; \text{in law}, \hspace{8mm}
 (\tld w , \tld m_0) \sim (\tld w_t^*, \tld m_t^*) \; \; \text{in law} .
\end{align*}
The mixing assumption gives
\begin{align*}
  \E \| w_t^*(n) - \tld w^*_t \|_E  
  & \le \int_0^t e^{-s} \E \| m_{t-s}^*(n) - \tld m^*_{t-s} \|_E ds 
  + \int_t^\infty e^{-s} \E \| \tld m^*_{t-s} \|_E ds \\
  & \lesssim \int_0^t e^{-s} \gamma_{\text{mix}}(t-s) ds + e^{-t}.
 \end{align*}
Noting that
\begin{align*}
& \int_0^t e^{-s} \gamma_{\text{mix}}(t-s) ds \le \int_0^\infty e^{-s} \gamma_{\text{mix}}(t-s) ds 
 \xrightarrow[t \to \infty]{} 0,
\\
& \int_{t=0}^\infty  \int_{s=0}^t e^{-s} \gamma_{\text{mix}}(t-s) ds dt
  =  \int_{s=0}^\infty e^{-s}  \int_{t=s}^\infty  \gamma_{\text{mix}}(t-s) dt ds 
  = \int_0^\infty \gamma_{mix}(t) dt < \infty,
 \end{align*} 
we have shown that 
$$t \mapsto \E \| w_t^*(n) - \tld w^*_t \|_E  \in L^1(\mathbb{R}^+)$$
 and tends to $0$ as $t \to \infty$.
Let us first consider $\psi(f,n)=\blangle \xi_1(n) \rho, \xi_2(n) \brangle$. In this case,
\begin{align*}
& Q_t \psi(f,n) -  \llangle \psi, \mu_\rho \rrangle = \E \blangle \xi_1(m^*_t(n)) \rho, \xi_2(m^*_t(n)) \brangle
- \E \blangle \xi_1(\tld m^*_t) \rho, \xi_2(\tld m^*_t) \brangle,
\\
& D_f \Big[ Q_t \psi(f,n) -  \llangle \psi, \mu_\rho \rrangle \Big] = \E \Big[  \xi_1(m^*_t(n)) \xi_2(m^*_t(n)) - \xi_1(\tld m_t^*) \xi_2(\tld m_t^*) \Big] .
\end{align*}
The Lipschitz assumption on $\xi_i$ and the inequality
\begin{align}
\Big| \langle h,k \rangle -\langle \tld h, \tld k \rangle \Big| \le  \| h \| \| k - \tld k \| + \| k \| \| h - \tld h \| \label{identity}
\end{align} 
lead to
\begin{align*}
 \Big| Q_t \psi(f,n) - \llangle \psi, \mu_\rho \rrangle \Big| \lesssim \| f \|_{G_0} \gamma_{mix}(t),
 \hspace{10mm}
N_k \Big( D_f \Big[ Q_t \psi(f,n) -  \llangle \psi, \mu_\rho \rrangle \Big] \Big) \lesssim  \gamma_{mix}(t).
\end{align*}

Let us now consider $\psi(f,n)=\blangle J(f), \xi_1(n) \brangle$. In this case,
\begin{align*}
& Q_t \psi(f,n) -  \llangle \psi, \mu_\rho \rrangle = e^{-t} \E \blangle J(f), \xi_1(m_t(n)) \brangle
+ \Big[
 \E \blangle w^*_t (n)\rho , \xi_1(m^*_t(n)) \brangle - \E \blangle  \tld w^*_t(n) \rho, \xi_1( \tld m^*_t) \brangle
 \Big] 
 \\
 &D_f \Big[ Q_t \psi(f,n) -  \llangle \psi, \mu_\rho \rrangle \Big] = e^{-t} \E[v \xi_1(m_t(n)) ] +
 \E \Big[ w^*_t(n) \xi_1(m^*_t(n)) - \tld w_t^*(n) \xi_1(\tld m_t^*(n)) \Big] .
\end{align*}
The first terms can be bounded by some $C e^{-t}$, and we can use \myref{identity} again to bound the second terms. 
Finally, let us consider $\psi(f,n)=\blangle K(f), \xi \brangle$. In this case,
\begin{align*}
 Q_t \psi(f,n) -  \llangle \psi, \mu_\rho \rrangle  = & e^{-2t} \E \blangle K(f), \xi \brangle
+ 2 e^{-t} \E \blangle w_t(n)   J(f) , \xi \brangle
\\ & \hspace{5mm}
+ \Big[
 \E \blangle w^*_t(n)^2 \rho , \xi \brangle - \E \blangle (\tld{w}^*_t)^2  \rho , \xi \brangle 
 \Big],
\\
 D_f \Big[ Q_t \psi(f,n) -  \llangle \psi, \mu_\rho \rrangle \Big]  & = e^{-2t} v^2 \xi 
+ 2 e^{-t} \E [ w_t(n)  ] v \xi + 
 \E \Big[ w^*_t(n)^2   - (\tld w^*_t)^2 \Big] \xi .
 \end{align*}
We conclude in the same way.
\end{proof}

Let us complete this section by pointing out some identities regarding the laws of $\tld w$, $\tld m_0$,  and $M^{-1}I(\tld m_0)$ which will be useful for calculations later on.

\begin{proposition} \label{relation} 
We have the following identities:
\begin{align}
 & \E \Big[ \tld w(x)   \tld w(y) \Big]  =
\frac{1}{2} \int_{\mathbb{R}} e^{-|t|} \E \Big[ \tld m_0(x) \tld m_t(y) \Big] dt,
 \label{relation1}
 \\
& \E \Big[ \tld w(x)   M^{-1}I(\tld m_0)(y) \Big]  = 
- \int_{t=0}^\infty (1 - e^{-t}) \E \Big[ \tld m_0(x) \tld m_t(y) \Big] dt \label{relation2}
\\
&  \hspace{37mm} = \E \Big[ \tld m_0(x) M^{-1} I(\tld m_0)(y) + \tld w(x) \tld m_0(y) \Big]
.
\label{relation3}
\end{align}

\end{proposition}

\begin{proof}[Proof] Recalling Definition \ref{w} and using the fact that $\tilde m$ is stationary, the quantity 
$\E\Big[ \tld w(x) \tld w(y)  \Big]$ 
can be written
$$
 \int_{s=-\infty}^0 \int_{t=-\infty}^0 e^{t+s}  \E[ \tld m_s(x) \tld m_t(y) ] dt ds 
 = \int_{s=-\infty}^0 \int_{t=-\infty}^0 e^{t+s} \E[ \tld m_0  (x) \tld m_{t-s}  (y) ] dt ds.
$$
We set $r=t-s$ and use Fubini's Theorem to obtain the expression
$$
 \int_{r=-\infty}^0 \Big( \int_{s=-\infty}^0 e^{2s} ds \Big) e^r \E [ \tld m_0(x) \tld m_r(y)] dr + 
\int_{r=0}^\infty \Big( \int_{s=-\infty}^{-r} e^{2s} ds \Big) e^r \E [ \tld m_0(x) \tld m_r(y)] dr, 
$$
which gives \eqref{relation1}.
As for the second relation \myref{relation2}, we have
\begin{align}
\E \Big[ \tld w  (x)  M^{-1}I(\tld m_0)  (y) \Big] =  \int_{-\infty}^0 e^s 
\E \Big[ \tld m_s  (x)  M^{-1}I(\tld m_0)  (y) \Big] ds.
\label{torelation2}
\end{align}
For $s \le 0$, recalling definition \myref{M-1} and using the Markov property,
\begin{align}
 \E \Big[ \tld m_s  (x)  M^{-1}I(\tld  m_0)  (y) \Big] & = 
 \int_E \E \Big[  m_s  (n)(x) M^{-1}I(n) \Big] d\nu(n) \nonumber \\
& = - \int_E \int_0^\infty \E\Big[ m_s  (n)(x) \E[ m_t  (n)(y) | {\cal F}_0  ]  \Big] dt d\nu(n)\nonumber   \\
& = - \int_E \int_0^\infty \E \Big[ m_s  (n)(x) m_t  (n)(y) \Big] dt d\nu(n) \nonumber  \\
& = - \int_0^\infty \E \Big[ \tld m_s  (x) \tld m_t  (y) \Big] dt . \label{here}
\end{align}
It follows that \eqref{torelation2} is equal to
\[
- \int_{s=-\infty}^0 \int_{t=0}^\infty e^s \E \Big[ \tld m_s  (x) \tld m_t  (y) \Big] ds dt 
 = - \int_{s=-\infty}^0 \int_{t=0}^\infty e^s \E \Big[ \tld m_0  (x) \tld m_{t-s}  (y) \Big] ds dt.
\]
Setting again $r=t-s$ and using Fubini's Theorem leads to the expression 
\[
 - \int_{r=0}^\infty (1-e^{-r} ) \E \Big[ \tld m_0  (x) \tld m_r  (y) \Big] dr
\]
for \eqref{torelation2}. Using \myref{here}, we also recognize this to be \myref{relation3}
\end{proof}

 \end{subsection}

\end{section}

\begin{section}{The perturbed test function method} \label{5}

From this point on, we shall work in the state space $(f,u,n) \in {\cal X}^\alpha_k = {G_0^k} \times H^\alpha_x \times E$, for $k=3$ and $\alpha \in (1/2, 3/2)$.
The aim of this section is to identify a "limiting generator" ${\cal L}$ satisfying, for a large class of test functions $\phi$,
\begin{equation}
{\cal{L}}^\varepsilon \phi^\varepsilon (f,u,n) = {\cal{L}} \phi(\rho,u) + O(\varepsilon)
\label{limit}
\end{equation}
where $\phi^\varepsilon$ is a perturbed version of $\phi$ of the form
$$
\phi^\varepsilon(f,u,n) = \phi(f,u) + \varepsilon \phi_1(f,u,n) + \varepsilon^2 \phi_2(f,u,n)
$$
and the correctors $\phi_1$, $\phi_2$ are good test functions to be identified. The $O(\eps)$ is to be understood as some $R^\eps(f,u,n)$ such that $\sup_{ B} | R^\eps(f,u,n) | = O(\eps)$ as $\eps \to 0$ for all bounded  $B \subset {\cal X}^\alpha_k$.
Let us start with a good test function $\phi \equiv \phi(f,u) \in \B^{loc}({\cal X}^\alpha_k)$ of the form
\begin{align}
\phi(f,u) = \Phi \Big( \langle f,\xi \rangle ; \langle u,\zeta \rangle \Big) = \Phi \Big( \langle \rho,\xi \rangle ; \langle u,\zeta \rangle \Big) \label{form} 
\end{align}
with $\Phi \in C^3(\mathbb{R}^2)$, and $\xi \equiv \xi(x)$, $\zeta \equiv \zeta(x) \in C^2(\mathbb{T})$.
Since, for fixed initial data $f$, the quantity $\rho(x)$ is conserved along the auxiliary equation \myref{kinetic g}, in the following we shall often write $\Phi$, $\partial_i \Phi$ (etc...) instead of specifying 
$\Phi \Big(  \langle \rho,\xi \rangle ; \langle u,\zeta \rangle   \Big)$, 
$\partial_i \Phi\Big( \langle \rho,\xi \rangle ; \langle u,\zeta \rangle \Big)$ (etc...).
Identifying the different orders in $\varepsilon$ in $\myref{limit}$, we get the set of equations
\begin{equation*}
\left\{
\begin{array}{l l}
 \varepsilon^{-2}  : & {\cal L}_\# \phi(f,u,n) = 0, \\
 \varepsilon^{-1}  :  & {\cal L}_\# \phi_1(f,u,n)  + {\cal L}_\flat \phi(f,u,n) = 0, \\
 \varepsilon^{0}  :  & {\cal L}_\# \phi_2(f,u,n)  + {\cal L}_\flat \phi_1(f,u,n) + {\cal L}_0 \phi(f,u,n)= {\cal L} \phi(\rho,u) ,
\end{array}
\right.
\end{equation*}
as well as
\begin{equation}
\left\{
\begin{array}{l l}
 \varepsilon^1  :  & {\cal L}_1 \phi(f,u,n) + {\cal L}_0 \phi_1(f,u,n) +  {\cal L}_\flat \phi_2(f,u,n) = O(1), \\
 \eps^2 : & {\cal L}_1 \phi_1(f,u,n) + {\cal L}_0 \phi_2(f,u,n) = O(\eps^{-1}), \\
 \eps^3 : &{\cal L}_1 \phi_2(f,u,n) = O(\eps^{-2}) . 
\end{array}
\right.
\label{positive_order}
\end{equation}

\begin{subsection}{Order \texorpdfstring{$\varepsilon^{-2}$}{}}

Given the form \myref{form} for $\phi$, we have $\phi(f,u) = \phi(\rho(f),u)$, which is invariant along the trajectories of the auxiliary process $(g_t, m_t)_{t \ge 0}$. It follows that, indeed $ {\cal L}_\# \phi(f,u,n)=0$. Furthermore, one may observe that
\begin{align*}
{\cal L}_\flat \phi(f,u) = \blangle J(f) , \partial_x D_f\phi(f,u) \brangle = \blangle J(f) , \partial_x \xi \brangle \partial_1 \Phi
\end{align*}
is a good test function satisfying the assumptions of Proposition \ref{mix}.

\end{subsection}

\begin{subsection}{Order \texorpdfstring{$\varepsilon^{-1}$}{}: first corrector}
We wish to solve the Poisson equation in $\phi_1$:
$$
{\cal L}_\# \phi_1(f,u,n) = -{\cal L}_\flat \phi(f,u) .
$$
The solution is expected to be given by
$$
\phi_1(f,u,n) = \int_0^\infty Q_t [ {\cal L}_\flat \phi ](f,u,n) dt .
$$
Proposition \ref{mix} guarantees that this corrector is indeed well-defined and locally bounded as long as
$
\llangle {\cal L}_\flat \phi , \mu_\rho \rrangle = 0 
$, and indeed, since $\tld m$ is centered,
\begin{align*}
\llangle {\cal L}_\flat \phi , \mu_\rho \rrangle & = 
\E \blangle \rho \tld w , \partial_x D_f \phi(\rho, u) \brangle = 0.
\end{align*}
 The first corrector $\phi_1$ can in fact be calculated explicitly:
\begin{align*}
\phi_1(f,u,n) = \int_0^\infty \E \blangle f, w_t(n)   \partial_x D_f\phi(f,u) \brangle dt = \blangle f , \int_0^\infty \E[w_t(n)] dt  \; \partial_x D_f\phi(f,u) \brangle,
\end{align*}
with
\begin{align*}
\int_0^\infty \E[w_t(n)] dt & = \int_{t=0}^\infty \int_{s=0}^t e^{-(t-s)} \E[m_s(n)] ds dt 
 = \int_{s=0}^\infty e^{s} \E[m_s(n)] \int_{t=s}^\infty e^{-t} dt ds \\
& = \int_{s=0}^\infty \E[m_s(n)]ds = -M^{-1}I(n) .
\end{align*}
We hence get the expression
\begin{equation}
\begin{array}{r l}
\phi_1(f,u,n) = \blangle \partial_x[ M^{-1}I(n)\rho - J(f)],  D_f\phi(f,u) \brangle,
\end{array}
\label{1st}
\end{equation}
that is, with \myref{form},
$$
\phi_1(f,u,n) =  \blangle J(f) - M^{-1}I(n) \rho, \partial_x \xi \brangle \partial_1 \Phi,
$$
which is clearly a good test function, thus in the domain of generator ${\cal L}_\#$. Making use of Remark~\ref{rem}, we have hence rigorously proven the following.
\begin{proposition}
The first corrector given by $\myref{1st}$ is a good test function which satisfies
$$
{\cal L}_\# \phi_1(f,u,n) = -{\cal L}_b \phi(f,u).
$$
\end{proposition}
A simple calculation then leads to
\begin{equation}
\begin{array}{l l}
\blangle D_f \phi_1(f,u,n) , h  \brangle = & \blangle \partial_x [M^{-1}I(n)\rho(h) - J(h)], D_f\phi(\rho,u) \brangle \\
& + D_f^2\phi(\rho,u) \Big( \rho(h), \partial_x [ M^{-1}I(n)\rho - J(f) ] \Big) .
\end{array} \label{Dphi1}
\end{equation}

\end{subsection} 

\begin{subsection}{Order \texorpdfstring{$\varepsilon^{0}$}{}: Limiting generator and second corrector}
We now wish to solve the Poisson equation in $\phi_2$:
$$
{\cal L}_\# \phi_2(f,u,n) = -{\cal L}_\flat \phi_1(f,u,n) - {\cal L}_0 \phi(f,u,n) + {\cal L}\phi(\rho,u) .
$$
Again, the solution is expected to be given by
$$
\phi_2(f,u,n) = \int_0^\infty Q_t \Big[ {\cal L}_\flat \phi_1 + {\cal L}_0 \phi - {\cal L}\phi(\rho,u)  \Big](f,u,n) dt .
$$
which should converge whenever
\begin{align}
{\cal L}\phi(\rho,u) = \llangle {\cal L}_\flat \phi_1, \mu_\rho \rrangle
+ \llangle {\cal L}_0 \phi , \mu_\rho \rrangle .
\label{condition}
\end{align}
This last condition determines the form of the limiting generator ${\cal L} \phi$. 

\begin{subsubsection}{Limiting generator}

We first compute ${\cal L}_\flat \phi_1(f,u,n)$: letting $h=-(v   \partial_x f + u   \partial_v f)$ in $\myref{Dphi1}$, noting that
\begin{align*}
 \rho(h) = - \partial_x J(f), \hspace{10mm} J(h) = - \partial_x K(f) + u \rho
\end{align*}
leads to the expression
\begin{align}
{\cal L}_\flat \phi_1(f,u,n) = & \blangle \partial_x \Big[ - M^{-1}I(n) \partial_x J(f) + \partial_x K(f) - u \rho \Big], D_f\phi(\rho,u) \brangle \nonumber \\
& + D_f^2\phi(\rho,u) \Big( - \partial_x J(f) , \partial_x[M^{-1}I(n)\rho - J(f) ] \Big) . \label{L_beta}
\end{align}
According to \myref{condition}, the first contribution ${\cal L}_\rho$ to the limiting generator is hence given by
\begin{equation}
\begin{array}{l l}
{\cal L}_\rho \phi(\rho,u) := 
\llangle {\cal L}_\flat \phi_1, \mu_\rho \rrangle 
= &  \E \blangle \partial_x\Big[  \partial_x[  (\tld w)^2  \rho ] - M^{-1}I(\tld m_0) \partial_x[\tld w \rho] - u\rho \Big], D_f\phi(\rho,u)  \brangle  \\
 & + 
 \E D_f^2\phi(\rho,u) \Big( \partial_x[\tld w \rho], \partial_x\Big[ \Big( \tld w  - M^{-1}I(\tld m_0) \Big) \rho \Big] \Big)  .
\end{array} 
\label{contr1}
\end{equation}
Similarly, recalling that
\begin{align}
{\cal L}_0\phi(f,u) = \blangle \partial^2_x u + J(f) , D_u \phi(f,u) \brangle
\label{L_0}
\end{align}
the second contribution ${\cal L}_u$ to the limiting generator is given by
\begin{align}
{\cal L}_u \phi(\rho,u) := \llangle {\cal L}_0 \phi , \mu_\rho \rrangle 
= \blangle \partial^2_x u , D_u \phi(\rho,u) \brangle .
\label{contr2}
\end{align}
Defining the operator ${\cal L}$ as
\begin{align}
{\cal L} \phi(\rho,u) = {\cal L}_\rho \phi(\rho,u) + {\cal L}_u \phi(\rho,u),
\label{limit generator}
\end{align}
let us now describe the corresponding process.
\end{subsubsection}

\begin{subsubsection}{Limiting SPDE } \label{limit_spde}

The limiting process corresponding to the generator ${\cal L}$, will be denoted by $(\rho_t, u_t)_{t \ge 0}$ for simplicity in this subsection only. Since it appears to be a diffusion process, we are naturally led to determine the associated SPDE. 

Firstly, note that this limiting process $(\rho_t, u_t)_{t \ge 0}$
is decoupled: the contribution ${\cal L}_u \phi$ given by \myref{contr2} corresponds to the simple deterministic PDE
\begin{align*}
\partial_t u = \partial_x^2 u .
\end{align*}
Let us hence focus on the contribution ${\cal L}_\rho \phi$ given by \myref{contr1} which describes the evolution of $(\rho_t)_{t \ge 0}$. The second order differential in $\phi$ is associated to to the diffusive part of the SPDE: we must identify the corresponding covariance operator. Let us assume that $D_f^2\phi(\rho,u)$ is associated to some symmetric kernel $\Psi$. Then,
\begin{align}
& \E \Big[ D_f^2\phi(\rho,u) \Big( \partial_x[\tld w \rho], \partial_x\Big[ \Big( \tld w  - M^{-1}I(\tld m_0) \Big) \rho \Big] \Big) \Big] \nonumber \\
& = \E \int_x \int_y \Psi(x,y) \partial_x[\tld w \rho](x) \partial_y \Big[ (\tld w - M^{-1}I(\tld m_0) ) \rho  \Big](y)dx dy  \nonumber \\
& = \E \int_x \int_y \Psi(x,y) \partial_x  \Big[
\partial_y \Big[ \tld w(x)   \Big( \tld w - M^{-1}I(\tld m_0) \Big)(y) \rho(x) \rho(y)  \Big]
\Big] dx dy \label{sym}  .
\end{align}
We are thus naturally led to study the kernel
$\E \Big[ \tld w(x)   \Big( \tld w(y) - M^{-1}I(\tld m_0)(y) \Big) \Big] $. Note that since $\Psi$ is symmetric in \myref{sym}, we may as well consider the symmetrized version
\begin{align*}
\E \Big[ \tld w(x) \tld w(y) \Big] - \frac{1}{2} 
\Big(  
\E \Big[ \tld w(x) M^{-1}I(\tld m_0)(y) \Big]
+ \E \Big[ \tld w(y) M^{-1}I(\tld m_0)(x) \Big]
\Big).
\end{align*}
Making use of the expressions obtained in Proposition \ref{relation} as well as the stationary character of $({\tilde m}_t)_{t \in \mathbb{R}}$, we easily derive
\begin{align}
\E \Big[ \tld w(x) \tld w(y) \Big] - \frac{1}{2} 
\Big(  
\E \Big[ \tld w(x) M^{-1}I(\tld m_0)(y) \Big]
+ \E \Big[ \tld w(y) M^{-1}I(\tld m_0)(x) \Big]
\Big)
= \frac{1}{2} k(x,y)
\label{weget}
\end{align}
where $k(x,y)$ is the kernel defined in \myref{exp_Q}.
Let us denote by $Q$ the linear operator on $L^2_x$ associated to this kernel:
\begin{align}
\forall f \in L^2_x, \; \; \; Qf(x) = \int_y k(x,y)f(y)dy.
\end{align}

\begin{proposition}
The operator $Q$ is self-adjoint, compact and non-negative:
$$
\forall f \in L^2(\mathbb{T} ; \mathbb{R}) \; \; \; (Qf,f ) \ge 0.
$$
\end{proposition}
We refer to \cite{dv1}, Lemma 1 for a proof. 
As a result, we can introduce the square root $Q^{1/2}$ of the operator $Q$, and denote by $q$ the associated kernel:
$$
 Q^{1/2}f(x) = \int_y q(x,y) f(y) dy, \hspace{10mm} k(x,y) = \int_z q(x,z) q(z,y) dz .
$$
We may now define the operator on $L^2_x$
\begin{align}
& \partial_x \Big[ \rho Q^{1/2} \Big] (f) (x)  := \int_y \partial_x \Big[ \rho(x) q(x,y) f(y) \Big] dy, \label{chap1-operator}
\end{align}
whose adjoint is given by
$$
\partial_x \Big[ \rho Q^{1/2} \Big]^*(g)(y) = \int_x \partial_x \Big[ \rho(x) q(x,y) \Big] g(x) dx .
$$
A straight-forward calculation then leads to
$$ 
\partial_x \Big[ \rho Q^{1/2} \Big]^* D^2\phi(\rho,u) \partial_x \Big[ \rho Q^{1/2} \Big] (f) (t)
 = \int_z N(t,z) f(z)  dz
$$
 where
$$
N(t,z)  = \int_x \int_y \partial_x \Big[ \rho(x) q(x,t) \Big]
\Psi(x,y)
\partial_y \Big[ \rho(y) q(y,z) \Big]   dx dy .
$$
We can now get back to \myref{sym} and derive
\begin{align*}
&  \E \Big[  D_f^2\phi(\rho, u ) \Big( \partial_x[\tld w \rho], \partial_x\Big[ \Big( \tld w  - M^{-1}I(\tld m_0) \Big) \rho \Big] \Big)  \Big]
 = \frac{1}{2} \int_x \int_y \Psi(x,y) \partial_x \Big[ \partial_y \Big[ \rho(x) \rho(y) k(x,y) \Big] \Big] dx dy \\
&  = \frac{1}{2} \int_x \int_y \int_z \Psi(x,y) 
\partial_x \Big[ \partial_y \Big[ \rho(x) \rho(y) q(x,z)q(z,y) \Big] \Big] dx dy dz 
 = \frac{1}{2}  \int_z N(z,z) dz \\
 & = \frac{1}{2} 
 \text{Tr} \Big( \partial_x \Big[ \rho Q^{1/2} \Big]^* D_f^2\phi(\rho) \partial_x \Big[ \rho Q^{1/2} \Big]  \Big) .
\end{align*}
Consequently, the contribution ${\cal L}_\rho$ to the limiting generator ${\cal L}$ may finally be rewritten as
\begin{equation}
{\cal L}_\rho \phi(\rho,u) = \blangle {\cal A}^I \rho, D_f \phi(\rho,u) \brangle + 
 \frac{1}{2} 
 \text{Tr} \Big( \partial_x \Big[ \rho Q^{1/2} \Big]^* D_f^2\phi(\rho,u) \partial_x \Big[ \rho Q^{1/2} \Big]  \Big) ,
\label{limiting_L}
\end{equation}
where ${\cal A}^I$ denotes the second order differential operator
\begin{equation}
{\cal A}^I \rho  =  \E \Big[  \partial_x \Big(  \partial_x [  \tld w^2 \rho ] -
M^{-1}I(\tld m_0) \partial_x [\tld w \rho] - \rho u \Big) \Big].
\label{AI}
\end{equation}
This naturally corresponds to the following SPDE, given in Itô form
\begin{align}
d\rho_t = {\cal A}^I \rho_t \; dt + \partial_x \Big[ \rho_t Q^{1/2} dW_t \Big].
\label{SPDE_ito}
\end{align}
Finally, let us determine the Stratonovich form associated with \myref{SPDE_ito}.
Recalling \myref{weget}, the Itô operator ${\cal A}^I$ may be rewritten as
\begin{align*}
{\cal A}^I \rho = \frac{1}{2} \partial^2_x \Big[ k(x,x) \rho \Big] + \partial_x \Big[ \E \Big( \partial_x[M^{-1}I(\tld m_0)] \tld w \Big) \rho \Big]
- \partial_x[ u \rho] .
\end{align*}
The correction from the Itô to the Stratonovich form is given by the formula
\begin{align*}
 & \partial_x \Big[ \rho Q^{1/2} \circ dW_t \Big] =  \partial_x \Big[ \rho Q^{1/2} dW_t \Big]
 +   {\cal A}^{I \to S} \rho
 \end{align*}
 where
 \begin{align*}
{\cal A}^{I \to S} \rho := \frac{1}{2} \partial_x \Big[ \int_y q(x,y) \partial_x \Big( \rho(x) q(x,y) \Big) dy \Big] =
 \frac{1}{2} \partial_x^2 \Big[ k(x,x) \rho \Big] - 
\frac{1}{4} \partial_x \Big[ \Big( \partial_x k(x,x) \Big) \rho \Big] 
\end{align*}
since $k(x,x) = \int_y q(x,y)^2dy$.
Hence, \myref{SPDE_ito} can be expressed in Stratonovich form as
$$
d\rho_t = {\cal A}^S \rho_t \; dt + \partial_x \Big[ \rho Q^{1/2} \circ dW_t \Big]
$$
where the differential operator ${\cal A}^S$ is of order one:
\begin{align*}
{\cal A}^S \rho = {\cal A}^I \rho - {\cal A}^{I \to S} \rho = 
 \partial_x \Big[\Big(  \E \partial_x[M^{-1}I(\tld m_0)] \tld w + \frac{1}{4} \partial_x k(x,x) - u\Big) \rho \Big].
\end{align*}
Using once again expression \myref{weget} and the identities from Proposition \ref{relation}, straight-forward calculations give
$$
{\cal A}^S \rho = \partial_x \Big[ (a-u) \rho \Big]
$$
where $a \equiv a(x)$ is given by \myref{exp_a}.

\end{subsubsection}

\begin{subsubsection}{Second corrector}

Using the form \myref{form} for $\phi$, recalling \myref{L_beta} and \myref{L_0}, calculations lead to the expressions
\begin{equation}
\begin{array}{l l}
{\cal L}_\flat \phi_1(f,u,n) = &
- \blangle J(f) , \partial_x ( M^{-1}I(n)   \partial_x \xi )  \brangle \partial_1 \Phi  
+ \blangle K(f), \partial^2_x \xi \brangle \partial_1 \Phi 
+ \blangle \rho, u   \partial_x \xi \brangle \partial_1 \Phi \\
& - \blangle J(f) , \partial_x \xi \brangle  \blangle \rho , M^{-1}I(n) \partial_x \xi \brangle \partial_1^2 \Phi 
+  \blangle J(f) , \partial_x \xi \brangle^2  \partial_1^2 \Phi 
 \end{array}
 \label{expr1}
\end{equation}
and
\begin{align}
{\cal L}_0 \phi(f,u) = \blangle u, \partial^2_x \zeta \brangle \partial_2 \Phi 
+ \blangle J(f), \zeta \brangle \partial_2 \Phi .
\label{expr2}
\end{align}
We now prove the following result.

\begin{proposition}
The second corrector given by 
$$
\phi_2(f,u,n) = \int_0^\infty Q_t \Big[ {\cal L}_\flat \phi_1 + {\cal L}_0 \phi - {\cal L}\phi \Big](f,u,n) dt
$$
defines a good test function which satisfies
\begin{equation}
{\cal L}_\# \phi_2(f,u,n) = -{\cal L}_\flat \phi_1(f,u,n) - {\cal L}_0 \phi(f,u,n) + {\cal L}\phi(\rho,u),
\label{eq_phi2}
\end{equation}
where ${\cal L}$ is the limiting generator given by \myref{limit generator}. 
\end{proposition}

\begin{proof}[Proof]
Firstly, note that $\theta := {\cal L}_\flat \phi_1 + {\cal L}_0 \phi - {\cal L}\phi \in C_b^{loc}({\cal X}_k^\alpha)$ (recall that $k=3$) has the form of a good test function. In particular, it is in the domain of ${\cal L}_\#$. 
For $B \subset {\cal X}_k^\alpha$ bounded, we shall prove the following points:
\begin{enumerate}
\item \textbf{First point}:
$\displaystyle{
\int_0^\infty  \sup_{(f,u,n) \in B}  \Big| Q_t \Big[ {\cal L}_\flat \phi_1 {\cal L}_0 \phi - {\cal L}\phi  \Big] (f,u,n) \Big| dt < \infty,
}$
\item  \textbf{Second point}:
$\displaystyle{
\int_0^\infty  \sup_{(f,u,n) \in B} \Big| Q_t {\cal L}_\# \Big[ {\cal L}_\flat \phi_1 + {\cal L}_0 \phi - {\cal L}\phi  \Big] (f,u,n) \Big| dt < \infty,
}$
\item  \textbf{Third point}:
$\displaystyle{
\int_0^\infty  \sup_{(f,u,n) \in B} N_k \Big( D_f Q_t \Big[ {\cal L}_\flat \phi_1 + {\cal L}_0 \phi - {\cal L}\phi  \Big] (f,u,n) \Big) dt < \infty ,
}$
\item  \textbf{Fourth point}:
$\displaystyle{
\int_0^\infty  \sup_{(f,u,n) \in B} \| D_u Q_t \Big[ {\cal L}_\flat \phi_1 + {\cal L}_0 \phi - {\cal L}\phi  \Big] (f,u,n) \|_{C^2(\mathbb{T})} dt < \infty,
}$
\end{enumerate}

Point $1$ shows that $\phi_2$ is well defined and in $\B^{loc}({\cal X}_k^\alpha)$.
Once point $2$ is obtained, we can use Corollary \ref{poisson} to conclude immediately that $\phi_2 \in D({\cal L}_\#)$ and satisfies \myref{eq_phi2}.
 It then remains to prove that $\phi_2$ is a good test function. 
 
 Points $3$ and $4$ guarantee that $\phi_2$ meets the last two requirements of Definition \ref{good}. Moreover, using the same decomposition as in the proof of Proposition \ref{generator_g}, we may write
 \begin{align*}
\frac{P_t \phi_2(f,u,n) - \phi_2(f,u,n)}{t}  =
\frac{Q_t \phi_2 (f,u,n) - \phi_2(f,u,n)}{t}  -  \E \Big[ \frac{\phi_2(g_t,u,m_t) - \phi_2(f,u,m_t)}{t}   \Big]  . 
\end{align*}
Considering the b.p.c convergence of each term of the right hand side as $t$ goes to zero, we obtain  
$\phi_2(f, u, \cdot) \in D(M)$ and 
$$
M \phi_2(f, u, n) = {\cal L}_\# \phi_2(f,u,n) - 
\blangle \partial_v[(v-n) f], D_f \phi_2(f,u,n) \brangle
$$
which is in $C_b^{loc}({\cal X}^\alpha_k)$, concluding the proof.
Let us now review the different points.
\paragraph{First point} 
It is clear from the explicit expressions \myref{expr1} and \myref{expr2} that ${\cal L}_\flat \phi_1$ and ${\cal L}_0 \phi$ are linear combinations of functions satisfying the assumptions of Proposition \ref{mix}, i.e essentially "quadratic in $n$". By definition, ${\cal L} \phi(f,u)=  \llangle {\cal L}_\flat \phi_1 + {\cal L}_0 \phi , \mu_\rho \rrangle$, so that
\begin{align*}
  Q_t \Big[ {\cal L}_\flat \phi_1 + {\cal L}_0 \phi - {\cal L}\phi  \Big] = 
  Q_t \Big[ {\cal L}_\flat \phi_1 + {\cal L}_0 \phi \Big] - \llangle {\cal L}_\flat \phi_1 + {\cal L}_0 \phi , \mu_\rho \rrangle .
\end{align*}
One can easily check that the estimates obtained from Proposition \ref{mix} are also locally bounded in $u \in L^\infty_x \subset H^\alpha_x$ (since $\alpha > 1/2$).
\paragraph{Second point}
Let us compute explicitely ${\cal L}_\# \Big[ {\cal L}_\flat \phi_1 + {\cal L}_0 \phi - {\cal L} \phi \Big]$, recalling that
$$
{\cal L}_\# \psi(f,n)  =  \blangle h, D_f \psi(f,n) \brangle + M \psi(f,n) 
$$
where $h = \partial_v[ (v-n) f ]$.
Since ${\cal L} \phi \equiv {\cal L} \phi(\rho,u)$, we have
${\cal L}_\# [{\cal L}_\flat \phi_1 + {\cal L}_0 \phi - {\cal L}\phi] = {\cal L}_\# {\cal L}_\flat \phi_1 + {\cal L}_\# {\cal L}_0 \phi$. 
Using Proposition \ref{mix} once again, we simply need to show that the test functions involved have the right form and are centered with respect to the invariant measure $\mu_\rho$ of the auxiliary process.
Calculations give, for $\chi \equiv \chi(x)$,
\begin{align*}
& \langle h,\chi \rangle = \langle \rho(h),\chi \rangle = 0, \\
& \langle J(h),\chi \rangle = \langle n \rho - J(f),\chi \rangle, \\
& \langle K(h),\chi \rangle = 2 \blangle n   J(f) - K(f),\chi \brangle.
\end{align*}
Let us review the different terms in \myref{expr1} and express the corresponding parts in 
${\cal L}_\#  {\cal L}_\flat \phi_1 $. 
\\
The first term is $\psi_1(f,n) = - \blangle J(f), \partial_x(M^{-1}I(n)   \partial_x \xi) \brangle \partial_1 \Phi$. We get
\begin{align*}
& M\psi_1(f,n) = - \blangle J(f), \partial_x(n.\partial_x \xi) \brangle \partial_1 \Phi, \\
& \blangle h,D\psi_1(f,n) \brangle = \blangle J(f)-n \rho, \partial_x(M^{-1}I(n).\partial_x \xi) \brangle \partial_1 \Phi,
\end{align*}
so that ${\cal L}_\# \psi_1$ has the form required in Proposition \ref{mix}. Calculations give the following expression for $\llangle {\cal L}_\# \psi_1, \mu_\rho \rrangle $:
$$
\E \Big[ - \blangle \tld w \rho , \partial_x(\tld m_0.\partial_x \xi) \brangle
+ \blangle \tld w \rho, \partial_x(M^{-1}I(\tld m_0).\partial_x \xi) \brangle
- \blangle \tld m_0 \rho, \partial_x(M^{-1}I(\tld m_0).\partial_x \xi) \brangle \Big] \partial_1 \Phi.
$$
Using the identity \myref{relation3} from Proposition \ref{relation}, this is indeed $0$.
\\
 The next term is $\psi_2(f,n) = \blangle K(f), \partial_x^2 \xi \brangle \partial_1 \Phi$. We get
\begin{align*}
& M \psi_2(f,n) = 0, \\
& \blangle h, D\psi_2(f,n) \brangle = 2 \blangle n   J(f) - K(f), \partial_x^2 \xi \brangle \partial_1 \Phi,
\end{align*}
so that ${\cal L}_\# \psi_2$ satisfies the assumption of Proposition \ref{mix}. Here, 
\begin{align*}
\llangle {\cal L}_\# \psi_2, \mu_\rho \rrangle = 2 \E \blangle (\tld m_0   \tld w - \tld w^2) \rho, \partial^2_x \xi \brangle \partial_1 \Phi.
\end{align*}
Using the identity \myref{relation1} from Proposition \ref{relation} as well as the stationarity of $\tld m$, we see that $\E (\tld m_0   \tld w) = \E (\tld w^2)$, so that this is indeed $0$.
\\
The next term $\langle \rho, u   \partial_x \xi\rangle \partial_1 \Phi$ is in fact constant along the trajectories of $(g_t,m_t)_{t \ge 0}$, giving no contribution.
The following term is 
$\psi_3(f,n) = - \blangle J(f), \partial_x \xi \brangle \blangle \rho, M^{-1}I(n)  \partial_x \xi \brangle \partial_1^2 \Phi$.
 We get
\begin{align*}
& M \psi_3(f,n) = - \blangle J(f), \partial_x \xi \brangle \blangle \rho, n  \partial_x \xi \brangle \partial_1^2 \Phi, \\
& \blangle h , D\psi_3(f,n) \brangle = - \blangle n \rho - J(f), \partial_x \xi \brangle 
\blangle \rho, M^{-1}I(n).\partial_x \xi \brangle \partial_1^2 \Phi,
\end{align*}
so that ${\cal L}\# \psi_3$ satisfies the assumption of Proposition \ref{mix}. 
Here,
\begin{align*}
\llangle {\cal L}_\# \psi_3, \mu_\rho \rrangle  = & \E \Big[ - \blangle \tld w \rho, \partial_x \xi \brangle \blangle \rho,\tld m_0.\partial_x \xi \brangle - \blangle \tld m_0 \rho,\partial_x \xi \brangle \blangle \rho, M^{-1}I(\tld m_0).\partial_x \xi \brangle \\
& \; \;  + \blangle \tld w \rho, \partial_x \xi \brangle \blangle \rho, M^{-1}I(\tld m_0).\partial_x \xi \brangle \Big] \partial_1^2 \Phi .
\end{align*}
Again, using the identity \myref{relation3} from Proposition \ref{relation}, this is indeed $0$.
 Finally, the last term is $\psi_4(f,n) = \blangle J(f), \partial_x \xi \brangle^2 \partial_1^2 \Phi$. We get
\begin{align*}
& M \psi_4(f,n) = 0, \\
& \blangle h, D\psi_4(f,n) \brangle = 2 \blangle n \rho - J(f), \partial_x \xi \brangle \blangle J(f), \partial_x \xi \brangle \partial_1^2 \Phi,
\end{align*}
so that ${\cal L}_\# \psi_4$ satisfies the assumption of Proposition \ref{mix}. 
Here,
\begin{align}
\llangle {\cal L}_\# \psi_4, \mu_\rho \rrangle = 2 \E \Big[  
\blangle \tld m_0 \rho , \partial_x \xi \brangle \blangle \tld w \rho, \partial_x \xi \brangle - \blangle \tld w \rho, \partial_x \xi \brangle \blangle \tld w \rho, \partial_x \xi \brangle
\Big] \partial_1^2 \Phi. 
\label{this_sym}
\end{align}
Using the identity \myref{relation1} as well as the stationarity of $\tld m$, we see that
$$
\frac{1}{2} \E[ \tld m_0(x) \tld w(y) +  \tld m_0(y) \tld w(x)  ]
=
\E[ \tld w(x) \tld w(y) ]
$$
hence, making use of the symmetry of \myref{this_sym}, it is indeed $0$.
It remains to study ${\cal L}_\# {\cal L}_0 \phi(f,u)$. The first term $\langle u, \partial_x^2 \zeta \rangle \partial_2 \Phi$ is constant. As for the second term $\blangle J(f), \zeta \brangle \partial_2 \Phi$, we get
\begin{align*}
{\cal L}_\# \Big[ \blangle J(f), \zeta \brangle \partial_2 \Phi \Big] = \blangle J(h), \zeta \brangle \partial_2 \Phi = \blangle n \rho - J(f), \zeta \brangle \partial_2 \Phi
\end{align*}
which satisfies the assumption of Proposition \ref{mix} and
\begin{align*}
\llangle \blangle n \rho - J(f), \zeta \brangle \partial_2 \Phi, \mu_\rho \rrangle = \E \blangle \tld m_0 \rho - \tld w \rho , \zeta \brangle \partial_2 \Phi = 0.
\end{align*}
This concludes the proof of the second point.
\paragraph{Third point} 
Once again, we simply use Proposition \ref{mix} (the part regarding the mixing speed of $D_f Q_t \psi$) since, 
as previously mentioned,
\begin{align*}
Q_t \Big[ {\cal L}_\flat \phi_1 + {\cal L}_0 \phi - {\cal L} \phi \Big] =
\Big( Q_t[ {\cal L}_\flat \phi_1 ] - \llangle {\cal L}_\flat \phi_1 , \mu_\rho \rrangle  \Big)
+ \Big( Q_t[ {\cal L}_0 \phi ] - \llangle {\cal L}_0 \phi , \mu_\rho \rrangle  \Big)
\end{align*}
and the different terms in \myref{expr1}, \myref{expr2} are good test functions satisfying the requested assumptions.
\paragraph{Fourth point} 
Reviewing $Q_t\Big[{\cal L}_\flat \phi_1 + {\cal L}_0 \phi - {\cal L}\phi \Big]$, we see that it is a sum of terms of the form
\begin{align*}
\Big( Q_t \psi (f,n) - \llangle \psi, \mu_\rho \rrangle \Big) \; \partial_j^i \Phi((f, \xi) ; (u, \zeta) ) .
\end{align*}
with $i,j=1$ or $2$ and $\psi$ of the form required in Proposition \ref{mix} (the terms involving $u$ in \myref{expr1} and \myref{expr2} are constant along the trajectories of $(g_t,m_t)_{t \ge 0}$ and are hence canceled out). The differential with respect to $u$ of such a term is hence simply
\begin{align*}
\Big( Q_t \psi (f,n) - \llangle \psi, \mu_\rho \rrangle \Big) \partial_2 \partial_j^i \Phi((f,\xi) ; (u,\zeta)) \zeta .
\end{align*}
Since we assumed $\Phi \in C^3$, its $\| \cdot \|_{C^2(\mathbb{T})}$ norm over $B \subset {\cal X}_k^\alpha$ can be bounded by
$$
C \sup_{(f,u,n) \in B} \Big| Q_t \psi (f,n) - \llangle \psi, \mu_\rho \rrangle \Big| \| \zeta \|_{C^2(\mathbb{T})}
$$ 
which is integrable over $\mathbb{R}^+$ according to Proposition \ref{mix}.
\end{proof}

\end{subsubsection}

\end{subsection}

\begin{subsection}{Consequence}
Note that since $\phi, \phi_1, \phi_2$ are good test functions, the equations with a positive order \myref{positive_order} are automatically satisfied since ${\cal L}_\star \phi$, ${\cal L}_\star \phi_i$ are locally bounded for $i=1,2,$ and $\star =0,1, \flat$.
Let us summarize the results we have obtained: for all $\phi$ good test-function of the form
\begin{align*}
\phi(f) = \Phi((\rho,\xi) ; (u, \zeta))
\end{align*}
with $\Phi \in C^3(\mathbb{R}^2)$, $\xi \equiv \xi(x)$, $\zeta \equiv \zeta(x) \in C^2(\mathbb{T})$, there exist correctors $\phi_1$, $\phi_2$ which are good test functions, such that the perturbed test function
$
\phi^\varepsilon = \phi + \varepsilon \phi_1 + \varepsilon^2 \phi_2
$
satisfies
$$
{\cal L}^\varepsilon \phi^\varepsilon(f, u, n) = {\cal L} \phi(\rho, u ) + R^\eps(f,u,n)
$$
where ${\cal L}$ is the limiting generator identified in \myref{limit generator}, and $R^\eps / \eps$ is locally bounded. 
As a result, we know from Proposition \ref{prop_gen} that
\begin{align*}
M^\varepsilon_\phi(t) = \phi^\varepsilon(f^\varepsilon_t, u^\eps_t, m^\varepsilon_t) - \int_0^t {\cal L}^\varepsilon \phi^\varepsilon(f^\varepsilon_s, u^\eps_s, m^\varepsilon_s) ds, \; \; \; t \ge 0
\end{align*}
defines a  $({\cal F}^\eps_t)_{t \ge 0}$-martingale, so that
\begin{align*}
N^\varepsilon_\phi(t) = \phi(\rho^\varepsilon_t, u^\eps_t) - \int_0^t {\cal L} \phi(\rho^\varepsilon_s, u^\eps_s) ds, \; \; \; t \ge 0
\end{align*}
is "almost" a  martingale, up to some $O(\varepsilon)$. This is an essential step towards our diffusion-approximation result.
\end{subsection}

\end{section}

\begin{section}{Diffusion-approximation} \label{6}

\begin{subsection}{Tightness}
We now prove the following result.
\begin{proposition}\label{tight}
Let $\alpha \in (1/2, 3/2]$ and assume that
$$
\sup_\eps \| u_0^\eps \|_{H^\alpha_x} + \sup_\eps \| f_0^\eps \|_{G_3} < \infty.
$$
Then the family of random variables $(\rho^\eps,u^\eps)_{\eps > 0}$ is tight in the space 
$$C([0,T] ; H_x^{-\gamma}) \times C([0,T] ; H_x^\beta), \; \; \gamma > 1/2, \; \; \; \beta < \alpha .$$
\end{proposition}

\begin{proof}

The estimates from Proposition \ref{u_compact1} guarantee that $(u^\eps)_\eps$ lies in a (deterministic) compact subset of $C([0,T] ; H^\beta_x)$. We only need to prove the tightness of $(\rho^\eps)_\eps$ in $C([0,T] ; H^{-\gamma}_x)$. To this intent, we will once again use the perturbed test function method, but this time only up to the first order. 
Firstly, note that since the injection $L^1_x \subset  H^{-\gamma}_x$ is compact for $\gamma > 1/2$, it is sufficient to prove
\begin{align}
& \limsup_\eps \mathbb{P}( \| \rho^\eps \|_{L^\infty_tL^1_x} > M) \xrightarrow[M \to \infty]{} 0,  \label{point1} \\
& \forall r > 0, \; \; \limsup_\eps \mathbb{P}( w_\gamma(\rho^\eps ; \delta) > r) \xrightarrow[\delta \to 0]{} 0,
\label{point2}
\end{align} 
where
$w_\gamma(\rho ; \delta)$ denote the modulus of continuity  
$$w_\gamma(\rho ; \delta) = \sup_{|t-s| \le \delta} \| \rho_t - \rho_s \|_{H_x^{-\gamma}}.$$
 One may refer to \cite{billing}, Chapter 2, for details on this tightness criterion. 
The first point \myref{point1} is obvious since $\| \rho^\eps_t \|_{L^1_x} = \| \rho^\eps_0 \|_{L^1_x}=1 $. Note that for any $\eta > \gamma$ an interpolation inequality gives, for some $\theta \in (0,1)$,
\begin{align*}
w_{\gamma}(\rho ; \delta) \le  \Big( 2 \sup_{t \in [0,T]} \| \rho_t \|_{L^1_x}\Big)^\theta w_{\eta}(\rho ; \delta)^{1-\theta} 
\end{align*}
so that it is in fact sufficient to establish \myref{point2} for $\gamma$ large enough.
Let us define the diagonal operator 
$$
J = (Id - \Delta_x)^{-1/2} \label{chap1-J}
$$
that is, considering the $L^2(\mathbb{T})$ Hilbert-base 
$\left\{ e_j(x) = (2\pi)^{-1/2} e^{ijx}, \;  j \in \mathbb{Z} \right\}$,
\begin{align*}
J \xi = \sum_j (1+|j|^2)^{-1/2} \langle \xi,e_j \rangle e_j,
\end{align*}
so that
$$
 \| J^\gamma \rho \|^2_{L^2} = \sum_j | \langle \rho, J^\gamma e_j \rangle |^2 = \sum_j (1+|j|^2)^{-\gamma} |\langle \rho,e_j \rangle|^2 = \| \rho \|^2_{H^{-\gamma}_x} .
$$
For fixed $j \in \mathbb{Z}$, we consider the good test function
\begin{align}
\phi_j(f) = \langle \rho, J^\gamma e_j \rangle.
\end{align}
which is indeed of the form \myref{form} used in the perturbed test function method, with 
\begin{align*}
& \| J^\gamma e_j \|_{C^2} = (2\pi)^{-1/2} |j|^2(1+|j|^2)^{-\gamma/2}  .
\end{align*}
The associated first corrector
\begin{align}
\phi^1_j(f,n) = \blangle J(f) - M^{-1}I(n) \rho, \partial_x J^\gamma e_j \brangle \label{first_cor}
\end{align}
is designed so that ${\cal L}_\# \phi_j^1(f,n) = - {\cal L}_\flat \phi_j(f)$. Defining the first order perturbation 
$$
\phi_j^\eps(f,n) = \phi_j(\rho) + \eps \phi_j^1(f,n)
$$
it results that 
\begin{equation}
{\cal L}^\eps \phi^\eps_j(f,u,n) = {\cal L}_\flat \phi^1_j(f,u,n) = O(1)
\label{perturbed_order1}
\end{equation}
in the sense that it can be bounded by $C \| J^\gamma e_j \|_{C^2}$ when $(f,n)$ lies in a bounded $B \subset {\cal X}_k^\alpha$. Let us choose $\gamma$ large enough so that
\begin{align*}
\sum_j \| J^\gamma e_j \|_{C^2}^2 < \infty,
\end{align*}
that is in fact, $\gamma > 5/2$. We now define
$$
\theta^\eps(t) = J^{-\gamma} \sum_j \phi_j^\eps(f^\eps_t,m^\eps_t) e_j \in H^{-\gamma}_x
$$
and notice that, using the form \myref{first_cor} and Proposition \ref{moments_eps}, for all $t \in [0,T]$,
\begin{align*}
\| \rho_t^\eps - \theta^\eps(t) \|_{H^{-\gamma}_x} & = \| J^\gamma \rho_t^\eps - J^\gamma \theta^\eps(t) \|_{L^2}
= \sum_j \Big| \phi_j(\rho_t^\eps) - \phi^\eps_j(f_t^\eps,m_t^\eps) \Big|^2
= \eps \sum_j \Big| \phi^1_j(f^\eps_t, m^\eps_t) \Big|^2
\\
& \le \eps \, C(T,C^*, \|u_0^\eps\|_{L^\infty_x}, \| f_0^\eps \|_{G^1})  \sum_j \| J^\gamma e_j \|_{C^1}^2
\lesssim \eps .
\end{align*}
We deduce that \myref{point2} can be derived from the tightness of $(\theta^\eps)_\eps$ in $D([0,T] ; H^{-\gamma}_x)$: indeed, $\rho^\eps$ being continuous,
\begin{align*}
w_\gamma(\rho^\eps ; \delta) \lesssim w'_\gamma(\rho^\eps ; \delta) \lesssim w'_\gamma(\theta^\eps ; \delta) + \eps
\end{align*}
where $w'_\gamma$ denotes the modulus of continuity in $D([0,T] ; H^{-\gamma}_x)$ (see again \cite{billing}, Chapter 3, for details). 
We apply Aldou's criterion to prove the tightness of $(\theta^\eps)_\eps$ and conclude: let $\delta > 0$ and $\tau_1, \tau_2$ be two $({\cal F}^\eps_t)_{t \ge 0}$-stopping time satisfying
\begin{align}
\tau_1 \le \tau_2 \le \tau_1 + \delta, \hspace{10mm} \tau_2 \le T . \label{stop_time}
\end{align}
Introducing the square integrable martingale
\begin{align*}
M^\eps_j(t) = \phi^\eps_j(f^\eps_t, m^\eps_t) - \int_0^t {\cal L}^\eps \phi^\eps_j(f^\eps_s, u^\eps_s, m^\eps_s) ds
\end{align*}
we have
\begin{align*}
\E \| \theta^\eps(\tau_2) - \theta^\eps(\tau_1) \|^2_{H^{-\gamma}_x}
& = \sum_j \Big| \phi_j^\eps(f^\eps_{\tau_2}, m^\eps_{\tau_2})  - \phi_j^\eps(f^\eps_{\tau_1}, m^\eps_{\tau_1})\Big|^2
\\
&
=
\sum_j \E \Big| M_j^\eps(\tau_2) - M_j^\eps(\tau_1) - \int_{\tau_1}^{\tau_2} {\cal L}^\eps \phi^\eps_j(f^\eps_s, u^\eps_s, m^\eps_s) ds \Big|^2 .
\end{align*}
On one hand, using Doob's optional sampling theorem, we have
\begin{align*}
\E \Big| M_j^\eps(\tau_2) - M_j^\eps(\tau_1)  \Big|^2  =
 \E \Big( | M_j^\eps(\tau_2)| ^2 - | M_j^\eps(\tau_1)| ^2   \Big)
 = \E \Big( \Big[M^\eps_j \Big](\tau_2) - \Big[ M_j^\eps \Big](\tau_1)   \Big)
\end{align*}
where $[M^\eps_j](t)$ denotes the predictable quadratic variation of $M^\eps_j(t)$. We know from Proposition \ref{prop_gen} that it is given by
\begin{align*}
\Big[ M_j^\eps(t) \Big] = \frac{1}{\eps^2} \int_0^t \Big( M |\phi^\eps_j|^2 - 2 \phi^\eps_j M \phi^\eps_j \Big) (f^\eps_s, m^\eps_s) ds
= \int_0^t \Big( M |\phi^1_j|^2 - 2 \phi^1_j M \phi^1_j \Big) (f^\eps_s, m^\eps_s) ds
\end{align*}
after some straight-forward calculation. Using expression \myref{first_cor} and the quadratic inequality \myref{chap1-quadratic}, we deduce that the integrand can be bounded over $[0,T]$ by $C \| J^\gamma e_j \|^2_{C^1}$, so that
\begin{align}
\sum_j \E \Big| M_j^\eps(\tau_2) - M_j^\eps(\tau_1)  \Big|^2 \lesssim \delta .
\label{comb1}
\end{align}
On the other hand, making use of \myref{perturbed_order1},
\begin{align*}
\Big| \int_s^t {\cal L}^\eps \phi^\eps_j(f^\eps_s, u^\eps_s, m^\eps_s) ds \Big|^2 \lesssim  \int_s^t 
| {\cal L}^\eps \phi^\eps_j(f^\eps_s, u^\eps_s, m^\eps_s) |^2 ds \lesssim \| J^\gamma e_j \|^2_{C^2} |t-s|
\end{align*}
so that the choice of $\gamma$ guarantees that
\begin{align}
\sum_j \E \Big| \int_{\tau_1}^{\tau_2}  {\cal L}^\eps \phi^\eps_j(f^\eps_s, u^\eps_s, m^\eps_s) ds \Big|^2 \lesssim  \delta . \label{comb2}
\end{align}
Combining \myref{comb1} and \myref{comb2} and using Markov's inequality, we conclude that
\begin{align*}
\forall r > 0, \; \; \limsup_\eps \sup_{\tau_1, \tau_2} \mathbb{P} \Big( \Big| \theta^\eps(\tau_2) - \theta^\eps(\tau_1) \Big| > r \Big) \xrightarrow[\delta \to 0]{} 0,
\end{align*}
where the supremum is taken over all stopping times $\tau_1, \tau_2$ satisfying \myref{stop_time}. Aldou's criterion is met, which concludes the proof. 

\end{proof}

In the following, let us fix $\gamma > 1/2$ and $\beta < \alpha$ and a subsequence $(\rho^{\eps_j}, u^{\eps_j})_{j \ge 1}$ which converges in law to some $(\rho, u)$. Using Skorokhod's representation theorem,  we may introduce some probability space $(\Omega^*, \mathbb{P}^*)$ hosting some random variables $(\rho_*^{\eps_j},  u_*^{\eps_j})_{j \ge 1}$ , $(\rho_*, u_*)$ satisfying
\begin{align*}
(\rho_*^{\eps_j},  u_*^{\eps_j}) \sim (\rho^{\eps_j}, u^{\eps_j}),
\hspace{3mm} (\rho_*, u_*) \sim (\rho, u) \hspace{3mm} \text{in law in } C([0,T] ; H^{-\gamma}_x) \times C([0,T] ; H^\beta_x)
\end{align*}
and
$$
(\rho_*^{\eps_j},  u_*^{\eps_j}) \xrightarrow[j \to \infty]{} (\rho_*, u_*) \text{ a.s in } C([0,T] ;  H^{-\gamma}_x) \times C([0,T] ; H^\beta_x).
$$
Additionally, note that passing to the limit in the equality $\| \rho^\eps_*(t) \|_{L^1_x} = 1$ leads to the inequality
$$
\sup_{[0,T]} \| \rho_*(t) \|_{L^1_x} \le 1 \; \; \; a.s
$$
or equivalently, regarding the limiting law,
\begin{align}
\mathbb{P} \Big[ \sup_{t \in [0,T]}\| \rho_t \|_{L^1_x} \le 1 \Big] = 1.
\label{L1bound}
\end{align}

\end{subsection}

\begin{subsection}{Convergence of the martingale problem}

Let us consider a good test-function $\phi$ of the form \myref{form}. The final result obtained in section \ref{5} can be summarized as follows: for any $s_1 < s_2 < \ldots < s_n = s < t$ and any
$\theta : (H^{-\gamma}_x \times H^\beta_x)^n \to \mathbb{R}$ continuous and bounded, 
\begin{align}
\E\Big[ \Big( \phi(  \rho^\eps_t,   u^\eps_t) - \phi(  \rho^\eps_s,   u^\eps_s) - \int_s^t {\cal L} \phi( \rho^\eps_\sigma,   u^\eps_\sigma) d\sigma \Big) \theta(  \rho^\eps_{s_1},   u^\eps_{s_1}, \ldots,   \rho^\eps_{s_n},   u^\eps_{s_n})  \Big] = O(\eps) .
\label{chap1-mart1}
\end{align}
Regarding the adaptedness of the process, note that since $\omega \mapsto (\rho^\eps_t , u^\eps_t)(\omega) \in L^1_x \times H^\alpha_x$ is ${\cal F}^\eps_t$-measurable, and since the injection $L^1_x \times H^\alpha_x \subset H^{-\gamma}_x \times H^\beta_x $ is continuous, the random variable $\omega \mapsto (\rho^\eps_t , u^\eps_t)(\omega) \in H^{-\gamma}_x \times H^\beta_x$ is also ${\cal F}^\eps_t$-measurable.
The martingale property \myref{chap1-mart1} only depends on the law of $(\rho^\eps, u^\eps) \in C([0,T] ; H^{-\gamma}_x \times H^\beta_x)$, so that we may in fact write, on the new probability space $(\Omega^*, \mathbb{P}^*)$,
\begin{align*}
\E^* \Big[ \Big( \phi(  \rho_*^{\eps_j}(t),   u_*^{\eps_j}(t)) - \phi(  \rho_*^{\eps_j}(s),   u_*^{\eps_j}(s)) - \int_s^t {\cal L} \phi( \rho_*^{\eps_j}(\sigma),   u_*^{\eps_j}(\sigma)) d\sigma \Big) \theta_* \Big] = O(\eps),
\end{align*}
where $\theta_*=\theta(  \rho_*^{\eps_j}(s_1),    \ldots,  u_*^{\eps_j}(s_n)) $.
From the expressions \myref{contr1}, \myref{contr2} and \myref{limit generator}, we easily deduce that ${\cal L}\phi$ is continuous and locally bounded on $H^{-\gamma}_x \times H^\beta_x$ as soon as the functions $\xi$ and $\zeta$ involved in \myref{form} have enough regularity (namely
$\xi \in H_x^{\gamma + 2}$, $\zeta \in H_x^{2 - \beta}$). We may then apply the dominated convergence theorem on the probability space $\Omega^*$ to send $\eps_j$ to zero. Using once again the invariance of the law, we obtain
$$
\E\Big[ \Big( \phi(  \rho_t,   u_t) - \phi(  \rho_s,   u_s) - \int_s^t {\cal L} \phi( \rho_\sigma,   u_\sigma) d\sigma \Big) \theta(  \rho_{s_1},   u_{s_1}, \ldots,   \rho_{s_n},   u_{s_n})  \Big] = 0.
$$
This equality holding true for any $s_1 < \ldots < s_n=s < t$ and $\theta$, we deduce that
$$
N_\phi(t) = \phi(\rho_t, u_t) - \phi(\rho_0, u_0) - \int_0^t {\cal L}\phi(\rho_s, u_s) ds, \; \; t \ge 0
$$
defines a continuous martingale, starting at $0$, with respect to the natural filtration 
\begin{align*}
{\cal F}_t = \sigma \Big( (\rho_s, u_s) \in H^{-\gamma}_x \times H^\beta_x, \; s \in [0,t] \Big), \; \; t \ge 0.
\end{align*}
This reasoning also holds for the good test function $|\phi|^2$ which is still of the form \myref{form}, so that $N_{|\phi|^2}(t)$ defines a continuous martingale as well. As mentioned in the proof of Proposition \ref{prop_gen} (see again \cite{dv2}, Theorem B.3) we derive that the quadratic variation of $N_\phi(t)$ is given by
$$
\Big[ N_\phi \Big](t) = \int_0^t \Big( {\cal L} | \phi |^2 - 2 \phi {\cal L} \phi \Big)(\rho_s, u_s) ds .
$$
Let us consider $\phi(\rho,u) = (\rho, \xi) + (u, \zeta)$. Simple calculations from the expression \myref{limiting_L} give
\begin{align*}
& {\cal L}\phi(\rho,u) = ({\cal A}^I \rho, \xi) + (\partial^2_x u , \zeta),
\\
& \Big( {\cal L} | \phi |^2 - 2 \phi {\cal L} \phi \Big)(\rho,u) 
= \int_x \int_y k(x,y) \rho(x) \partial_x \xi(x) \rho(y) \partial_y \xi (y) dx dy.
\end{align*}
where the operator ${\cal A}^I$ and the kernel $k(x,y)$ are defined in \myref{AI} and \myref{exp_Q}. 
Let us introduce the $H_x^{-\gamma-2} \times H_x^{\beta-2}$-valued process
$$
N(t) = \Big( \rho_t -\rho_0 - \int_0^t {\cal A}^I \rho_s ds \; , \; u_t - u_0 - \int_0^t \partial^2_x u_s ds \Big),
\; \; \; t \ge 0.
$$
Then for all $(\xi, \zeta) \in H_x^{\gamma + 2} \times H_x^{2-\beta}$,
\begin{align*}
\blangle N(t) , (\xi , \zeta ) \brangle, \; \; t \ge 0
\end{align*}
defines a continuous real-valued $({\cal F}_t)_{t \ge 0}$-martingale of quadratic variation
\begin{align}
\Big[ \blangle N , (\xi, \zeta) \brangle \Big](t) = \int_0^t  \int_x \int_y k(x,y) \rho_t(x) \partial_x \xi(x) \rho_t(y) \partial_y \xi (y) dx dy .
\label{quad}
\end{align}
Using the polarization formula 
$\langle N , h_1 \rangle \langle N , h_2 \rangle = \frac{1}{4} \Big( \langle N, h_1 + h_2\rangle - \langle N,h_1-h_2 \rangle \Big)$,
 we derive from \myref{quad} that
\begin{align*}
\blangle N(t) , (\xi_1 n \zeta_1) \brangle \blangle N(t) ,(\xi_2 , \zeta_2 ) \brangle - \blangle V(t) \xi_1,\xi_2 \brangle , \; \; t \ge 0
\end{align*}
is a continuous real-valued $({\cal F}_t)_{t \ge 0}$-martingale, where the operator $V(t)$ is given by
\begin{align*}
\blangle V(t) \xi_1,\xi_2 \brangle & = \int_0^t \int_x \int_y k(x,y) \rho_t(x) \partial_x \xi_1(x) \rho_t(y) \partial_y \xi_2(y) dx dy
\\
& = \int_z \Big( \int_x \partial_x(\rho_t(x) q(x,z) ) \xi_1(x) dx \Big) \Big( \int_x \partial_x(\rho_t(y) q(y,z) ) \xi_2(y) dy \Big) dz
\\
& = \int_z \Big( \partial_x[ \rho_t Q^{1/2} ]^* \xi_1(z) \Big) \Big(  \partial_x[ \rho_t Q^{1/2} ]^* \xi_2 (z)  \Big) dz.
\end{align*}
We have used the notation introduced in \myref{chap1-operator}. That is in fact
$$
V(t) = \partial_x[\rho_t Q^{1/2} ] \partial_x[ \rho_t Q^{1/2} ]^* .
$$
We may therefore apply the martingale representation theorem \cite[Theorem 8.2]{da_prato}:
there exists some filtered probability space $(\widehat \Omega, (\widehat{\cal F}_t)_{t \ge 0})$ equipped with a Wiener process $\hat W$ and a process $(\widehat \rho, \widehat u)$ defined on $\widehat \Omega$ whose law is that of $(\rho,u)$, such that the process
$$
\widehat N(t) = \Big( \widehat \rho_t -\rho_0 - \int_0^t {\cal A}^I \widehat \rho_s ds \; , \; \widehat u_t - u_0 - \int_0^t \partial^2_x \widehat u_s ds \Big),
\; \; \; t \ge 0.
$$
can be represented as the Wiener integral
$$
\widehat N(t) = \Big( \int_0^t \partial_x \Big[ \widehat \rho_t Q^{1/2} d\widehat W_t \Big] ; 0 \Big).
$$
This shows  that the limiting process $(\rho_t, u_t)_{t \ge 0}$ is a weak solution of the expected SPDE \myref{edps}, which concludes the proof of Theorem \ref{chap1-thm1}.

\end{subsection}

\end{section}


\begin{section}{Uniqueness for the limiting SPDE} \label{7}

We now prove the path-wise uniqueness for the limiting SPDE stated in Theorem \ref{chap1-thm2}. 
For conveniency, let us drop the hats on the quantities $\widehat \rho, \widehat u$ introduced earlier and consider the limiting process $(\rho,u) \in C([0,T] ; H^{-\gamma}_x \times H^\beta_x)$, where $\rho$ satisfies the SPDE of \myref{edps} given on a filtered probability space $(\Omega, ({\cal F}_t)_{t \ge 0})$ equipped with a Wiener process $W=(W_t)_{t \ge 0}$.
Of course, $u = (u_t)_{t \in [0,T]}$ is completely determined as the solution the deterministic heat equation with initial data $u_0$. 
Let us denote 
$$
b_t(x) = a(x) - u_t(x),
$$
 with $a$ given by \myref{exp_a}. Diagonalizing the compact operator $Q^{1/2}$ involved in \myref{edps}, we may introduce functions $(\phi_k)_{k \ge 0}$ such that
$$
Q^{1/2} \circ dW_t = \sum_k \phi_k \circ d\beta^k_t, \hspace{10mm} (\beta^k) \text{ independent Brownian motions}.
$$
As a consequence, equation \myref{edps} can be re-written as
$$
\left\{
\begin{array}{l}
d \rho_t + \partial_x(b_t \rho_t) dt
 = \partial_x\Big( \rho_t \circ \sum_k \phi_k d\beta^k_t \Big), 
 \\
  \rho(0) = \rho_0 \in L^1_x,
\end{array}
\right.
$$
or, equivalently, in Itô form,
\begin{equation}
\left\{
\begin{array}{l}
 d \rho_t + \Big( \partial_x(b_t \rho_t) - \frac{1}{2} \sum_k \partial_x(\phi_k \partial_x(\phi_k \rho_t)) \Big) dt
 = \sum_k \partial_x(\phi_k \rho_t) d\beta^k_t \label{the_spde} \\
  \rho(0) = \rho_0 \in L^1_x,
\end{array}
\right.
\end{equation}
Recall that, as mentioned in \myref{L1bound}, the solution $\rho$ constructed as a weak limit of $(\rho^\eps)_{\eps}$ in Theorem~\ref{chap1-thm1} satisfies the additional bound
\begin{align}
\sup_{t \in [0,T]} \| \rho_t \|_{L^1_x} \le 1, \; \; a.s.
\label{thebound}
\end{align}
Under the additional regularity assumptions of Theorem \ref{chap1-thm2}, it is clear that $ b \in L^\infty([0,T] ; W^{4,\infty}_x)$. 
Regarding the regularity of $\phi_k$, one may note that, for $\sigma > 4 + 1/2$, 
$$
\sum_k \| \phi_k \|^2_{W_x^{4,\infty}} \le \sum_k \| \phi_k \|^2_{H_x^\sigma} = \| Q^{1/2} \|^2_{{\cal L}_2(L^2 ; H_x^\sigma)} = \text{Tr}_{L^2} \Big( J^{-2 \sigma} Q \Big)
$$
where $J$ is the positive regularizing diagonal operator introduced in \myref{chap1-J}.
Applying Corollary C.2 and Proposition C.3 from appendix C of \cite{da_prato}, we deduce, for any $\theta > 1$,
$$
\text{Tr}_{L^2_x} \Big( J^{-2 \sigma} Q \Big) =  \text{Tr}_{L^2_x} (J^\theta J^{-\theta} J^{-2 \sigma} Q)
\le \text{Tr}_{L^2_x}(J^\theta) \; \| J^{-(\theta + 2 \sigma)} Q \|_{{\cal L}(L^2_x)}.
$$
On one hand, $\text{Tr}_{L^2_x}(J^\theta) = \sum_k (1 + |k|^2)^{-\theta/2} < \infty$ since $\theta > 1$. 
On the other hand, for $f \in L^2_x$, we have
\begin{align*}
J^{-(\theta + 2 \sigma)} Q f(x) = \int_y \Big( Id - \partial^2_x \Big)^{(\theta + 2 \sigma)/2} k(x,y) f(y) dy
\end{align*}
so that the bound $\| J^{-(\theta + 2 \sigma)} Q \|_{{\cal L}(L^2_x)} < \infty$ is ensured as soon as 
\begin{align*}
\sup_{y \in \mathbb{T}} \| k( \cdot , y) \|_{W_x^{\theta + 2 \sigma, \infty}} < \infty .
\end{align*}
The number $\theta + 2 \sigma$ can be chosen as close to $1 + 2 \times (4+1/2) = 10$ as desired. Hence, the regularity assumption \myref{regularity} of Theorem \ref{chap1-thm2} leads to
\begin{align*}
\sum_k \| \phi_k \|^2_{W_x^{4,\infty}} < \infty.
\end{align*}
It is now clear that the proof of Theorem \ref{chap1-thm2} boils down to that of the following result.

\begin{proposition} \label{thm_unique}
Assume that $\rho_0 \in L^1_x$ and
\begin{align}
 b \in L^\infty([0,T] ; W^{4,\infty}_x),
 \hspace{10mm}
 \sum_k \| \phi_k \|^2_{W^{4,\infty}_x} < \infty . \label{ass}
\end{align}
Then a solution $\rho \in C([0,T] ; H^{-\gamma}_x)$ a.s of \myref{the_spde} satisfying \myref{thebound} is path-wise unique.
\end{proposition}

We are naturally led to consider the dual BSPDE (Backward SPDE) associated with \myref{the_spde}, of unknown $(\psi, Z)$:
\begin{equation}
\left\{
\begin{array}{l}
 d \psi_t + \Big( - b_t \partial_x \psi_t + \frac{1}{2} \sum_k \phi_k \partial_x(\phi_k \partial_x \psi_t)  - \sum_k \phi_k \partial_x Z_t^k \Big) dt = \sum_k Z^k_t d\beta^k_t  \label{BSPDE}
\\
 \psi(T) = \psi_T .
\end{array}
\right.
\end{equation}
For $\sigma > 0$, let us define the spaces (where "prog." stands for "progressively measurable")
\begin{align*}
& \mathbb{H}_\psi^\sigma = \Big\{ \psi : \Omega \times [0,T] \to H^\sigma_x \text{ prog.}, \; \| \psi \|^2_{\mathbb{H}^\sigma} :=   \E \int_0^T \| \psi_t \|^2_{H_x^\sigma} dt < \infty  \Big\},
\\
& \mathbb{H}_Z^\sigma = \Big\{ Z=(Z^k)_k : \Omega \times [0,T] \to H^\sigma_x \text{ prog.}, \; \| Z \|^2_{\mathbb{H}^\sigma} :=  \sum_k \E \int_0^T \| Z_t \|^2_{H_x^\sigma} dt < \infty  \Big\}.
\end{align*}
The existence theory for \myref{BSPDE} is linked to the uniqueness for \myref{the_spde} in the following way.

\begin{lemma} \label{le_lem}
Assume that for some $s > 1/2$, for all ${\cal F}_T$-measurable $\psi_T : \Omega \to C^\infty(\mathbb{T}_x)$, \myref{BSPDE} has a solution $(\psi, Z) \in \mathbb{H}_\psi^{s+2} \times \mathbb{H}_Z^{s+1}$. Then the conclusion of Proposition \ref{thm_unique} holds.
\end{lemma}

\begin{proof}[Proof of the lemma]

Considering two solutions $\rho^1, \rho^2$ of \myref{the_spde}, let us define $\rho = \rho^1 - \rho^2$, so that $\rho(0) = 0$.
We start by regularizing \myref{the_spde} using an appropriate kernel $k_\delta \in C^\infty(\mathbb{T}_x), \delta > 0$ (one may for instance consider the Féjer kernel). Denoting $\rho_\delta = k_\delta * \rho$ and $\partial_\delta = k_\delta * \partial_x$, we have
\begin{equation*}
\left\{
\begin{array}{l}
 d \rho^\delta_t + \Big( \partial_\delta(b_t \rho_t) 
- \frac{1}{2} \sum_k \partial_\delta(\phi_k \partial_x(\phi_k \rho_t)) \Big) dt
 = \sum_k \partial_\delta(\phi_k \rho_t) d\beta^k_t \\
  \rho^\delta(0) = k_\delta * \rho_0 := \rho^\delta_0
\end{array}
\right.
\end{equation*}
in the "strong sense" that $(\rho^\delta_t)_{t \ge 0}$ is an $L^2_x$-valued semi-martingale. Itô's formula is now licit (one may verify that the assumptions of \cite{da_prato} Theorem 4.17 are satisfied) and gives
\begin{align*}
d \langle \rho^\delta_t, \psi_t \rangle & = \Big[ -\blangle \partial_\delta(b_t \rho_t) , \psi_t \brangle - \blangle \rho^\delta_t , b \partial_x \psi_t \brangle  \Big] dt
\\
& + \Big[ \frac{1}{2} \sum_k  \blangle \partial_\delta(\phi_k \partial _x\phi_k \rho_t)) , \psi_t \brangle - \blangle \rho^\delta_t , \phi_k \partial_x(\phi_k \partial_x \psi_t) \brangle \Big] dt
\\
& + \Big[ \sum_k \blangle \rho^\delta_t , \phi_k \partial_x Z^k_t \brangle + \blangle \partial_\delta (\phi_k \rho_t), Z^k_t \brangle \Big] dt
\\
& + \sum_k \Big[ \blangle \partial_\delta(\phi_k \rho_t) , \psi_t \brangle + \blangle \rho^\delta_t, Z^k_t \brangle \Big] d\beta^k_t .
\end{align*}
The assumptions easily guarantee that the last term defines a square-integrable martingale. We may hence write
\begin{align}
\E \Big[ \langle \rho_T , k_\delta * \psi_T \rangle \Big] & = - \E \int_0^T \blangle \rho_t , b_t \partial_x( k_\delta * \psi_t) - k_\delta * [ b_t \partial_x \psi_t] \brangle dt  \nonumber
\\
& - \frac{1}{2} \E \sum_k \int_0^T \blangle \rho_t , \phi_k \partial_x(\phi_k \partial_x (k_\delta * \psi_t)) - k_\delta * [ \phi_k \partial_x(\phi_k \partial_x \psi_t)]  \brangle dt \label{2nd}
\\
& - \E \sum_k \int_0^T \blangle \rho_t, k_\delta * [ \phi_k \partial_x Z^k_t] - \phi_k \partial_x(k_\delta * Z^k_t) \brangle \label{3rd}
dt .
\end{align}
The left-hand side easily goes to $\E[ \langle \rho_T, \psi_T \rangle ]$,while
the right-hand side vanishes as $\delta \to 0$. For instance, using \myref{thebound}, one may control the second term \myref{2nd} with
\begin{align*}
\sum_k \E \int_0^T \| \phi_k \partial_x(\phi_k \partial_x (k_\delta * \psi_t)) - k_\delta * [ \phi_k \partial_x(\phi_k \partial_x \psi_t)]  \|_{L_x^\infty} dt
\end{align*}
The assumptions guarantee that $\psi_t \in H_x^{s+2} \subset W_x^{2,\infty}$ a.s, so that the integrand tends to $0$ for all fixed $k, t$ and almost every $\omega$. With the uniform bound 
\begin{align*}
\| \phi_k \partial_x(\phi_k \partial_x (k_\delta * \psi_t)) - k_\delta * [ \phi_k \partial_x(\phi_k \partial_x \psi_t)]  \|_{L_x^\infty}
\le C \|\phi_k \|^2_{W_x^{1,\infty}} \| \psi_t \|_{W_x^{2,\infty}} \le C \|\phi_k \|^2_{W_x^{1,\infty}} \| \psi_t \|_{H_x^{s+2}}
\end{align*}
and the regularity assumption \myref{ass}, we may apply dominated convergence to conclude. Note that the bound for the third term \myref{3rd} is
\begin{align*}
\| k_\delta * [ \phi_k \partial_x Z^k_t] - \phi_k \partial_x(k_\delta * Z^k_t)  \|_{L^\infty_x} \le \| \phi_k \|_{L^\infty} \| Z^k_t \|_{W_x^{1,\infty}} \le \| \phi_k \|_{L^\infty}^2 + \| Z^k_t \|_{H_x^{s+1}}^2 .
\end{align*}
Since $\| \rho_T \|_{L^1_x} \le 1$ a.s, the equality $\E [ \langle \rho_T, \psi_T \rangle ] = 0$ for all ${\cal F}_T$-measurable $\psi_T : \Omega \to C^\infty(\mathbb{T})$ leads to 
$\rho_T(x) = 0$ for almost every $(x, \omega) \in \mathbb{T} \times \Omega$. This holds for all $T > 0$ and since $\rho \in C([0,T] ; H^{-\gamma}_x)$ almost surely, we  conclude $\rho = 0$ a.s.
\end{proof}

It now remains to prove the existence of sufficiently regular solutions to the BSPDE \myref{BSPDE}. Let us consider the equation perturbed with an additional $\frac{\eps}{2} \partial_x^2 \psi$:
\begin{equation}
\hspace{-4mm}
\left\{
\hspace{-1mm}
\begin{array}{l}
 d \psi_t + \Big(
 \left(- b_t  + \frac{1}{2} \sum_k  \phi_k \partial_x \phi_k  \right) \partial_x \psi_t  - \sum_k \phi_k \partial_x Z_t^k \Big) dt 
+  \frac{1}{2} \Big( \sum_k \phi_k^2 + \eps \Big) \partial_x^2 \psi_t dt  = \sum_k Z^k_t d\beta^k_t  \label{RSPDE_eps}
\\
 \psi(T) = \psi_T .
 \end{array}
 \right.
\end{equation}
In this super-parabolic setup, in view of the assumptions \myref{ass} on the coefficients, assuming furthermore that 
$$
\psi_T \in L^2 \Big( (\Omega, {\cal F}_T) ; H_x^4 \Big),
$$
Theorem 2.3 of \cite{du_meng} guarantees the existence of a solution $(\psi^\eps, Z^\eps)$ of \myref{RSPDE_eps} such that
\begin{align}
& (\psi^\eps, Z^\eps) \in {\mathbb H}^5 \times {\mathbb H}^4
\; \; \text{ and } \; \;
\E \Big[ \sup_{t \in [0,T]} \| \psi_t^\eps \|_{H_x^4}^2 \Big] < \infty.
\label{meng}
\end{align} 
Of course, the corresponding estimates in these spaces may depend on $\eps > 0$.
 
\begin{proposition} \label{big_calc}
For all $0 \le m \le 4$, the solution $(\psi^\eps, Z^\eps) \equiv (\psi, Z)$ of \myref{RSPDE_eps} satisfies the estimate
\begin{align}
&  \frac{1}{2} \| \partial_x^m \psi_t \|_{L^2_x}^2 + \frac{1}{2} \int_t^T \Big( \eps \| \partial_x^{m+1} \psi_s \|_{L^2_x}^2 +
 \sum_k \| \partial_x^m Z^k_s + m(\partial_x^m \psi_s)(\partial_x \phi_k) - \partial_x(\partial_x^m \psi_s \phi_k) \|_{L^2_x}^2 \Big) ds \nonumber
 \\
&   \le \frac{1}{2} \| \partial_x^m \psi_T \|_{L^2_x}^2 + C \int_t^T  \Big(  \| \psi_s \|_{H_x^m}^2 + \sum_k \| Z^k_s \|^2_{H_x^{m-1}}  \Big)ds - \sum_k \int_t^T \blangle \partial_x^m \psi_s, \partial_x^m Z^k_s \brangle d\beta^k_s ,
\label{big_est}
\end{align}
with the convention $\| Z \|_{H_x^{m-1}} = 0$ when $m=0$, where the constant $C$ depends only on the quantities involved in \myref{ass}.
\end{proposition}

\begin{remark} \label{L1_mart}
Note that \myref{meng} guarantees that, for $0 \le m \le 4$,
\begin{align*}
M_t =  \sum_k \int_0^t \blangle \partial_x^m \psi_s, \partial_x^m Z^k_s \brangle d\beta^k_s, \; \; \; t \ge 0
\end{align*}
defines a real-valued $L^1$ martingale. Indeed, 
\begin{align*}
V_T := \sum_k \int_0^T \Big| \blangle \partial_x^m \psi_t, \partial_x^m Z_t^k  \brangle \Big|^2 dt
\le \sup_{t \in [0,T]} \| \psi_t \|_{H_x^m}^2 \sum_k \int_0^T \| Z_t^k \|^2_{H_x^m} dt < \infty \; \; a.s
\end{align*}
ensures that $M_t$ is a local martingale. Burkholder-Davis-Gundy's inequality \cite{bdg} then gives
\begin{align*}
\E \Big[ \sup_{t \in [0,T]} |M_t| \Big] \lesssim \E \Big[ V_T^{1/2} \Big] 
\lesssim \E \Big[ \sup_{t \in [0,T]} \| \psi_t \|^2_{H_x^m} \Big] + \sum_k \E \int_0^T \| Z_t^k \|^2_{H_x^m} dt   < \infty.
\end{align*}
\end{remark}

Let us now give a proof of Proposition \ref{big_calc}.

\begin{proof}[Proof]
We apply Itô's formula to $\| \partial_x^m \psi_t \|_{L^2_x}^2$. To make it rigorous (so that the assumptions of \cite{da_prato} Theorem 4.17 are again satisfied), one may start by regularizing the coefficients $b$ and $(\phi_k)_k$ involved in \myref{RSPDE_eps}, so that \mycite{du_meng} provides the regularity $(\psi, Z) \in \mathbb{H}^6 \times \mathbb{H}^5$. We obtain
\begin{align*}
 \frac{1}{2} d \| \partial_x^{m} \psi_t \|_{L^2_x}^2 & =  \Big( A^1_t + A^2_t  + A^3_t + A^4_t + A^5_t \Big) dt + \blangle \partial_x^m \psi_t, \partial_x^m Z_t dW_t \brangle,
\end{align*}
with
\begin{align*}
& A^1_t := - \frac{\eps}{2} \blangle \partial_x^m \psi_t, \partial_x^{m+2} \psi_t \brangle ,
\hspace{5mm}
 A^2_t := \blangle \partial_x^m \psi_t, \partial_x^m(b_t \partial_x \psi_t) \brangle  ,
 \hspace{5mm}
A^3_t := \sum_k \blangle \partial_x^m \psi_t , \partial_x^m (\phi_k \partial_x Z^k_t) \brangle ,
\\ & A^4_t :=  \frac{1}{2} \sum_k \| \partial_x^m Z_t^k \|^2_{L^2_x} ,
\hspace{5mm}
 A^5_t := - \frac{1}{2} \sum_k \blangle \partial_x^m \psi_t, \partial_x^m (\phi_k \partial_x(\phi_k \partial_x \psi_t)) \brangle dt.
\end{align*}
Let us review the different terms. The first term $A^1_t$
appears untouched in \myref{big_est}. We may expand the second term into
\begin{align*}
A^2_t &  = 
 \blangle \partial_x^m  \psi_t,  b_t \partial_x^{m+1} \psi_t \brangle
+ \sum_{l=0}^{m-1}\binom{m}{l} \blangle \partial_x^m \psi_t, \partial_x^{m-l} b_t \partial_x^{l+1} \psi_s \brangle \nonumber
\\
& = - \frac{1}{2} \blangle | \partial_x^m \psi_t|^2, \partial_x b_t \brangle
+ \sum_{l=0}^{m-1} \binom{m}{l}  \blangle \partial_x^m \psi_t, \partial_x^{m-l} b_t \partial_x^{l+1} \psi_s \brangle,
\end{align*}
which can be bounded by $C \| b \|_{L^\infty_t W_x^{m,\infty}} \| \psi_t \|_{H_x^m}^2$.
Similarly, the third term can be expanded into
\begin{align*}
A_t^3  & = \Big[ \sum_k 
 \blangle \partial_x^m \psi_t, \phi_k \partial_x^{m+1} Z^k_t \brangle
 + m \blangle \partial_x^m \psi_t, \partial_x \phi_k \partial_x^m Z^k_t \brangle \Big] \\
 & \hspace{5mm} + \Big[ \sum_k \ \sum_{l=0}^{m-2} \binom{m}{l}
 \blangle \partial_x^m \psi_t \partial_x^{m-l} \phi_k, \partial_x^{l+1} Z^k_t \brangle
 \Big]
 \hspace{5mm} := B_t^{1} + B_t^{2}.
\end{align*}
The term $B_t^2$ can be bounded by 
$$
C \sum_k \| \psi_t \|_{H_x^m} \| \phi_k \|_{W_x^{m,\infty}} \| Z^k_t \|_{H_x^{m-1}}
\lesssim  \| \psi_t \|_{H_x^m}^2 \sum_k \| \phi_k \|_{W_x^{m,\infty}}^2 + \sum_k  \| Z^k_t \|_{H_x^{m-1}}^2.
$$
Using a polarizing identity, the term $B^1_t$ may be rewritten as
\begin{align*}
B^1_t & =  - \sum_k 
\blangle \partial_x(\partial_x^m \psi_t \phi_k) - m \partial_x^m \psi_t \partial_x \phi_k, \partial_x^m Z^k_t  \brangle
\\
& = - \frac{1}{2} \sum_k  \| \partial_x^m Z^k_t \|^2_{L^2_x} -
 \frac{1}{2} \sum_k
\|  \partial_x(\partial_x^m \psi_t \phi_k) - m \partial_x^m \psi_t \partial_x \phi_k \|_{L^2_x}^2
\\
& \hspace{6mm} +  \frac{1}{2} \sum_k \| \partial_x^m Z^k _t + m \partial_x^m \psi_t \partial_x \phi_k - \partial_x(\partial_x^m \psi_t \phi_k) \|_{L^2_x}^2
\\
& =: C^1_t + C^2_t + C^3_t.
\end{align*}
The term $C^1_t$ cancels out $A^4_t$. The term $C^3_t$ appears untouched in \myref{big_est}. We are hence left to study
\begin{align*}
D_t := A^5_t + C^2_t = 
- \frac{1}{2} \sum_k  \Big( \partial_x^m \psi_t, \partial_x^m(\phi_k \partial_x(\phi_k \partial_x \psi_t) ) \Big) 
- \frac{1}{2} \sum_k 
\|  \partial_x(\partial_x^m \psi_t \phi_k) - m \partial_x^m \psi_t \partial_x \phi_k \|_{L^2_x}^2
 .
\end{align*}
By expanding the derivatives and integrating by parts as done previously, we get the expression
\[
-\frac{1}{2} \sum_k \Big[ \Big( | \partial_x^m \psi_t |^2, \partial_x(\phi_k \partial_x \phi_k) \Big) - \Big( | \partial_x^{m+1} \psi_t|^2, | \phi_k |^2 \Big) 
  + \sum_{l=0}^{m-2} \binom{m}{l}
\Big( \partial_x^m \psi_t, \partial_x^{m-l} \phi_k \partial_x^{l+1}(\phi_k \partial_x \psi_t) \Big) \Big] 
\]
for the term $A^5_t$, and
\begin{align*}
C_t^2 = - \frac{1}{2} \sum_k \Big[
\Big( | \partial_x^m \psi_t|^2 , (m-1)^2 |\partial_x \phi_k|^2 + (m-1) \partial_x(\phi_k \partial_x \phi_k) \Big)
+ \Big( | \partial_x^{m+1} \psi_t|^2, | \phi_k |^2 \Big) \Big] .
\end{align*}
We see that the unwanted terms involving $\partial_x^{m+1} \psi_t$ therefore cancel out in $D_t$, and the remaining terms can once again be bounded by $C \| \psi_t \|_{H_x^m}^2 \sum_k \| \phi_k \|^2_{W_x^{m,\infty}}$.
\end{proof}

From Proposition \ref{big_calc}, we derive the following uniform bound.
\begin{corollary} \label{le_cor}
The solution $(\psi^\eps, Z^\eps)$ of \myref{RSPDE_eps} satisfies the bound
\begin{align*}
\| \psi^\eps \|_{\mathbb{H}^4}^2 + \| Z^\eps \|_{\mathbb{H}^3}^2  \lesssim 1.
\end{align*}
The constant involved in $\lesssim$ depends only on $T$ and the quantities of \myref{ass}.
\end{corollary}
\begin{proof}[Proof]
Let us simply denote $(\psi, Z) = (\psi^\eps, Z^\eps)$.
As mentioned in Remark \ref{L1_mart}, we may take the expectation in \myref{big_est}, which gives, in particular, for $0 \le m \le 4$, the estimate
\begin{align}
& \E \|  \psi_t \|^2_{H_x^m} \le  \E \| \psi_T \|_{H_x^m}^2
+ C \Big( \int_t^T \E \| \psi_s \|^2_{H_x^m} ds + \sum_k \int_t^T \E \| Z^k_s \|^2_{H_x^{m-1}} ds \Big).
\label{expect1}
\end{align}
Writing
$$
\| \partial_x^m Z^k_s \|_{L^2_x} \le  \| \partial_x^m Z^k_s + m(\partial_x^m \psi_s)(\partial_x \phi_k) - \partial_x(\partial_x^m \psi_s \phi_k) \|_{L^2_x} + \| m(\partial_x^m \psi_s)(\partial_x \phi_k) - \partial_x(\partial_x^m \psi_s \phi_k) \|_{L^2_x},
$$
and using \eqref{big_est} again, we also obtain
\begin{align}
& \sum_k \int_t^T \E \| Z^k_s \|^2_{H_x^m} ds  \le
 \E \|  \psi_T \|_{H_x^m}^2 + C \Big( \int_t^T \E \| \psi_s \|^2_{H_x^{m+1}} ds + \sum_k \int_t^T \E \| Z^k_s \|^2_{H_x^{m-1}} ds \Big) .
 \label{expect2}
\end{align}
Recalling the convention $\| Z \|_{H_x^{-1}} = 0$, in Proposition \ref{big_calc}, inequality \myref{expect2} yields, for $0 \le m \le 3$,
\begin{align}
\sum_k \int_t^T \E \| Z^k_s \|^2_{H_x^m} ds  \lesssim 
\E \| \psi_T \|^2_{H_x^3} + \int_t^T \E \| \psi_s \|^2_{H_x^4} . 
\label{expect3}
\end{align}
Coming back to \myref{expect1}, we get, for $0 \le m \le 4$,
\begin{align*}
& \E \|  \psi_t \|^2_{H_x^4} \lesssim \E \| \psi_T \|_{H_x^4}^2
 + \int_t^T \E \| \psi_s \|^2_{H_x^4} ds  
\end{align*}
and Grönwall's Lemma provides the desired estimate for $\psi$. The estimate for $Z$ is obtained by taking $t=0$ in \myref{expect3}.
\end{proof}

Now, let us consider $\eps_1, \eps_2 > 0$ and denote $(\psi, Z) = (\psi^{\eps_1} - \psi^{\eps_2}, Z^{\eps_1} - Z^{\eps_2})$. In the same fashion as Proposition \ref{big_calc}, one can prove the following estimate, for $0 \le m \le 3$:
\begin{align*}
& \frac{1}{2} \| \partial_x^m \psi_t \|_{L^2_x}^2 + \frac{1}{2} \int_t^T 
 \sum_k \| \partial_x^m Z^k_s + m(\partial_x^m \psi_s)(\partial_x \phi_k) - \partial_x(\partial_x^m \psi_s \phi_k) \|_{L^2_x}^2  ds
 \\
 & \le C | \eps_2 - \eps_1 |  \int_0^T \Big( \| \psi^{\eps_1}_s \|^2_{H_x^{m+1}}  + \| \psi^{\eps_2}_s \|^2_{H_x^{m+1}}  \Big) ds  + C \int_t^T  \Big(  \| \psi_s \|_{H_x^m}^2 + \sum_k \| Z^k_s \|^2_{H_x^{m-1}}  \Big)ds \\
 & \hspace{10mm}- \sum_k \int_t^T \Big( \partial_x^m \psi_s, \partial_x^m Z^k_s \Big) d\beta^k_s  .
\end{align*}
Relying on the bound from Corollary \ref{le_cor} and applying similar calculations as in its proof, we deduce
$$
\| \psi^{\eps_1} - \psi^{\eps_2} \|^2_{\mathbb{H}^3} + \| Z^{\eps_1} - Z^{\eps_2}  \|^2_{\mathbb{H}^2} \lesssim |\eps_1 - \eps_2|.
$$
As a consequence, $(\psi^\eps, Z^\eps)$ converges to some $(\psi, Z)$ in the Banach space $\mathbb{H}^3 \times \mathbb{H}^2$ as $\eps$ goes to zero. It is now easy to send $\eps$ to zero in \myref{RSPDE_eps} and verify that $(\psi, Z) \in {\mathbb{H}}^3 \times \mathbb{H}^3$ is indeed a solution of \myref{BSPDE}. 
Lemma \ref{le_lem} may be applied, which concludes the proof of Proposition \ref{thm_unique}, and therefore that of Theorem \ref{chap1-thm2}.

\end{section}





\begin{section}{Appendix: Well-posedness of the path-wise kinetic system} \label{app}

For the sake of completeness, we wish to prove the well-posedness of the deterministic PDE system
\begin{equation}
\left\{
\begin{array}{l }
\partial_t f + v   \partial_x f + \partial_v \Big[ (m + u - v) f \Big] = 0, \\
\partial_t u - \partial^2_x u = J(f) - \rho u.
\end{array}
\right.
\label{S}
\end{equation}
Here, we consider a fixed $\omega \in \Omega$, so that $m \equiv m(\omega)$ belongs to $D([0,T] ; E)$, space of càdlàg functions taking values in the separable complete space $E$.

\begin{subsection}{\texorpdfstring{$L^1$}{} and \texorpdfstring{$L^2$}{} estimates for a conservation equation with Lipschitz coefficients}
Recall that, whenever $w \in D([0,T] ; E)$ is given, the solution $f[w]$ to the linear conservation equation
\begin{equation*}
\left\{
\begin{array}{l}
\partial_t f + v   \partial_x f + \partial_v \Big[ (w - v) f \Big] = 0, \\
f_0 \in L^1_{x,v},
\end{array}
\right. 
\end{equation*}
is naturally expressed as
$$
f_t[w](x,v) = e^t f_0 \circ \Phi_0^t(x,v), \label{char1}
$$
where $ \Phi_s^t(x,v) = (X_s^t(x,v) , V_s^t(x,v))$ is the flow associated to the characteristics:
\begin{equation}
\left\{
\begin{array}{l l}
\displaystyle{ \frac{d}{ds} X^t_s = V^t_s, } &  X^t_t(x,v) = x 
\vspace{2mm}
\\
\displaystyle{ 
 \frac{d}{ds} V^t_s = w_s(X^t_s) - V^t_s, } & V^t_t(x,v) = v
\end{array}
\right.
\label{flow1}
\end{equation}
Grönwall's lemma immediately gives the following result.
\begin{lemma} \label{stab_flow}
Let $w, \tld w \in D([0,T] ; E)$ and let $\Phi[w],  \Phi[\tld w]$ denote the respective flows. Then,
$$
\forall s,t \in [0,T], \; \;  \Big\| \Phi_s
^t[w] - \ \Phi_s^t[\tld w] \Big\|_{L^\infty_{x,v}} \lesssim \| w - \tld w \|_{L^\infty_t L^\infty_x} .
$$
The constant involved in $\lesssim$ depends on $T, \| \partial_x  w \|_{L^\infty_{t,x}},  \| \partial_x \tld w \|_{L^\infty_{t,x}}$ only.
\label{stability}\end{lemma}
Let us now fix some $w \in D([0,T] ; E)$ and give some estimates regarding the $v$-moments of $f_t[w]$. Recall the notation, for $f \in G^1$,
$$
\rho(f)(x) = \int f(x,v)dv, \hspace{5mm} J(f)(x) = \int v f(x,v) dv.
$$ 

\begin{lemma}[$L^1$ moments estimates] \label{lemL1}
Let $f_0 \in L^1_{x,v}$. Then for all $t \in [0,T]$,
$$
\| f_t[w] \|_{L^1_{x,v}} = \| f_0 \|_{L^1_{x,v}} .
$$
Assume furthermore that $f_0 \in {G^k}$, $k \ge 1$. Then for all $t \in [0,T]$, 
$$
 \| f_t[w] \|_{G^k} \lesssim \| f_0 \|_{{G^k}} .
$$
The constant involved in $\lesssim$ depends on $T,k$ and $\| w \|_{L^\infty_{t,x}}$ only. \\
The estimate for the first moment is precisely:
$$
\| J(f_t[w]) \|_{L^1_{x,v}} \le \Big( \| J(f_0) \|_{L^1_x} + t \sup_{s \in [0,t]} \| w_s \|_{L^\infty_x} \| f_0\|_{L^1_{x,v}} \Big)e^{-t}
$$
\end{lemma}

\begin{proof}[Proof]
Without loss of generality, we may assume $f_0 \ge 0$. 
When $w \in D([0,1] ; E)$ is smooth, it is clear that the flow $\Phi[\tld w]$  is a globally-defined $C^\infty$-diffeomorphism, whose Jacobian determinant is given by
$
J \Phi_0^t[\tld w] = \det(D\Phi_0^t[\tld w]) = e^t.
$
The conservation of the $L^1$-norm is then obtained by a change of variable using the form \myref{char}. The result is generalized to $w \in D([0,1] ; E)$ by density using Lemma \ref{stability}.

For simplicity, we only prove the estimate in $G^1$: the estimates on higher moments can be obtained by induction using the same process. Assuming $f_0 \in C_c^\infty(\mathbb{T} \times \mathbb{R})$, the following calculations are justified:
\begin{align*}
\frac{d}{dt} \| J(f_t) \|_{L^1_x} & = \frac{d}{dt} \int_x \int_v | v | f_t(x,v) dx dv
= - \int_x \int_v | v | \Big( v   \partial_x f + \partial_v[(w_t(x) -v)f_t(x,v)] \Big) dx dv \\
&
 = - \int_v \sign(v) w_t(x) f_t(x,v) dv + \int_v |v| f_t(x,v) dv,
\end{align*}
so that
\begin{align*}
\frac{d}{dt} \| J(f_t) \|_{L^1_x}  \le \| w_t \|_{L^\infty_x} \| f_0 \|_{L^1_x}  - \| J(f_t) \|_{L^1_x}
\le \sup_{s \in [0,t]} \| w_s \|_{L^\infty_x} \| f_0 \|_{L^1_x}  - \| J(f_t) \|_{L^1_x},
\end{align*}
that is, integrating over $[0,t]$,
\begin{align*}
\| J(f_t) \|_{L^1_x} \le \Big( \|J(f_0) \|_{L^1_x} + t \sup_{s \in [0,t]} \| w_s \|_{L^\infty_x} \| f_0 \|_{L^1_x} \Big) - \int_0^t \| J(f_s) \|_{L^1_x} ds.
\end{align*}
Grönwall's lemma gives the expected result. Again, the inequality holds for a general initial data $f_0 \ge 0$ by a density argument, using Fatou's Lemma.
\end{proof}

\begin{corollary} \label{cor}
Let $f_0 \in L^2_{x,v}$. Then for all $t \in [0,T]$, 
$$
 \| f_t[w] \|_{L^2_{x,v}} = e^{t/2} \| f_0 \|_{L^2_{x,v}}.
$$
Assume furthermore that $|f_0|^2 \in {G^k}$, $k \ge 1$. Then for all $t \in [0,T]$, 
$$
\| f_t[w]^2 \|_{G^k} \lesssim \| f_0^2 \|_{{G^k}} .
$$
The constant involved in $\lesssim$ depends on $T$, $k$ and $\| w \|_{L^\infty_{t,x}}$ only.
\end{corollary}

\begin{proof}[Proof]
Simply note that $h_t := e^{-t} (f_t[w])^2 =  e^t f_0^2 \circ \Phi_0^t$ and apply Lemma \ref{lemL1} to $h$.
\end{proof}

\begin{corollary}[$L^2$ moments estimates] \label{lemL2}
Let $f_0 \in L^2_{x,v}$ and $f_0^2 \in G_4$. Then, for all $t \in [0,T]$, 
\begin{align*}
& \| \rho(f_t[w]) \|_{L^2_x} \lesssim \| f_0^2 \|_{G_2}^{1/2}, \\
& \| J(f_t[w]) \|_{L^2_x}  \lesssim  \| f_0^2 \|_{G_4}^{1/2} . 
\end{align*}
The constants involved in $\lesssim$ depends on $T$ and $\| w \|_{L^\infty_{t,x}}$ only.
\end{corollary}

\begin{proof}[Proof]
For instance, for the second estimate, Hölder's inequality gives
\begin{align*}
 \Big( \int_v | v | f dv \Big)^2 & = \Big( \int_v (1+|v|) |v| f \times \frac{1}{1+|v|} dv \Big)^2  
 \le \Big( \int_v (1+|v|)^2 |v|^2 f^2 dv \Big) \Big( \int_v \frac{1}{(1+|v|)^2} dv \Big)
\end{align*}
from which we deduce $ \int_x J(f)^2 dx \le C \| f^2 \|_{G_4} $, and we use the estimate from Corollary \ref{cor}.
\end{proof}

We will now establish some kind of stability estimate for $f$ as a function of $w$. Let us start with a lemma regarding the gradient.

\begin{lemma}[$L^1$ gradient estimates] \label{lemgrad}
Assume that $\| \nabla_{x,v} f_0 \|_{L^1_{x,v}} < \infty$. Then, 
$$
\| \nabla_{x,v} f_t[w] \|_{L^1_{x,v}} \le C(T, \| \partial_x w \|_{L^\infty_{t,x}} ) \| \nabla_{x,v} f_0 \|_{L^1_{x,v}} .
$$
Assume furthermore that $\| v \nabla_{x,v} f_0 \|_{L^1_{x,v}} < \infty$. Then,
$$
 \| v \nabla_{x,v} f_t[w] \|_{L^1_{x,v}} \le C(T, \| w \|_{L^\infty_{t,x}}, \| \partial_x w \|_{L^\infty_{t,x}}, \| \nabla_{x,v} f_0 \|_{L^1_{x,v}}, \| v \nabla_{x,v} f_0 \|_{L^1_{x,v}}).
$$
\end{lemma}

\begin{proof}[Proof]
Let us assume that $f_0 \in C^\infty_c$, so that we are dealing with smooth, compact-supported, strong solutions. The general case is once again deduced by density arguments. From 
\begin{align*}
\partial_t f + v \partial_x f +(w-v) \partial_v f - f= 0
\end{align*}
we derive that $g_1=\partial_x f, g_2 = \partial_v f$  satisfy
\begin{align}
& \partial_t g_1 + v \partial_x g_1 + \partial_x w \; g_2 + (w-v) \partial_v g_1 - g_1 = 0, \label{eq g1}\\
& \partial_t g_2 + g1 + v \partial_x g_2 + (w-v) \partial_v g_2 - 2 g_2 = 0 \label{eq g2}
\end{align}
Multiplying \myref{eq g1} by $\sign(g_1)$ and integrating over time
(to be rigorous, one should replace $\sign(x)$ by $s_\eta(x) = x / (|x| + \eta)$ and carefully let $\eta$ go to zero), we are easily led to 
\begin{align*}
\frac{d}{dt} \int_x \int_v |g_1| dx dv &  
\le  \| \partial_x w \|_{L^\infty_{t,x}} \int_x \int_v |g_2| dx dv .
\end{align*}
Similarly, multiplying \myref{eq g2} by $\sign(g_2)$ we get 
\begin{align*}
\frac{d}{dt} \int_x \int_v |g_2| dx dv  
 \le  \int_x \int_v | g_1 | dx dv + \int_x \int_v |g_s| dx dv .
\end{align*}
Summing these and applying Gr\"onwall's inequality gives the expected estimate.
Let us now establish the estimate on $\| v \nabla_{x,v} f \|_{L^1_{x,v}}$. From \myref{eq g1}, \myref{eq g2}, we derive that $h_1 := v \partial_x f$ and $h_2 := v \partial_v f$ satisfy
\begin{align*}
& \partial_t h_1 + v \partial_x h_1 + (\partial_x w) h_2 + v(w-v) \partial_v g_1 - h_1 = 0, \\
& \partial_t h_2 + h_1 + v \partial_x h_2 + v(w-v) \partial_v g_2 - 2h_2 = 0 
\end{align*}
Now since $(w-v) \partial_v h_i = (w-v) g_i + v(w-v) \partial_v g_i$, we can rewrite
\begin{align*}
 v(w-v) \partial_v g_i = (w-v) \partial_v h_i +h_i - w g_i
\end{align*}
and get 
\begin{align*}
& \partial_t h_1 + v \partial_x h_1 + (\partial_x w) h_2 +(w-v) \partial_v h_1 - wg_1 = 0, \\
& \partial_t h_2 +h_1 + v \partial_x h_2 +(w-v) \partial_v h_2  - h_2 -wg_2 = 0.
\end{align*}
We then use the same  approach to obtain
\begin{align*}
& \frac{d}{dt} \| h_1(t) \|_{L^1_{x,v}} \le \|\partial_x w \|_{L^\infty_{t,x}} \| h_2(t) \|_{L^1_{x,v}} - \| h_1(t) \|_{L^1_{x,v}} + \| w \|_{L^\infty_{t,x}} \| g_1(t) \|_{L^1_{x,v}} 
\\
&\frac{d}{dt} \| h_2(t) \|_{L^1_{x,v}} \le \| h_1(t) \|_{L^1_{x,v}} + \| w \|_{L^\infty_{t,x}} \| g_2(t) \|_{L^1_{x,v}}
\end{align*}
and conclude by applying the previous estimates and Gr\"onwall's inequality.
\end{proof}

\begin{corollary}[$L^1$ stability] \label{corstab} 
Let $w, \tld w \in D([0,T] ; E)$.
Assume 
$$
\| (1+|v|) \nabla_{x,v} f_0 \|_{L^1_{x,v}} <+\infty.
$$
Then the following estimates hold:
\begin{align*}
& \| \rho_t[w] -  \rho_t[\tld w] \|_{L^1_x} \le \| f_t[w] - f_t [\tld w]\|_{L^1_{x,v}} \lesssim \| w - \tld w \|_{L^\infty_{t,x}}  ,
\\
& \| J(f_t[w]) - J( f_t[ \tld w]) \|_{L^1_x} \lesssim \| w - \tld w \|_{L^\infty_{t,x}} .
\end{align*}
The constants involved in $\lesssim$ depend on $T$, $\| (1+|v|) \nabla_{x,v} f_0 \|_{L^1_{x,v}}$, $\|w \|_{L^\infty_{t,x}}$, $\| \tld w \|_{L^\infty_{t,x}}$, 
$\| \partial_x w \|_{L^\infty_{t,x}}$, and  $\| \partial_x \tld w \|_{L^\infty_{t,x}}$ 
 only.
\end{corollary}

\begin{proof}[Proof]
Let us once again assume that $f_0 \in C^\infty_c$. Let us denote $f_t = f_t[w]$, $\tld f_t = f_t[\tld w]$, $F = f - \tld f$ and $W = w - \tld w$. These quantities satisfy
\begin{align*}
\partial_t F + v   \partial_x F + \partial_v \Big[ (w-v) F + W \tld f \Big] = 0 
\end{align*}
which can be rewritten as
$$
\partial_t F + v   \partial_x F + (w-v)   \partial_v F - F + W   \partial_v \tld f = 0. \label{eq F}
$$
Multiplying by $\sign(F)$, we get
\begin{align*}
\partial_t | F | + v   \partial_x | F | + (w-v)   \partial_v |F| - |F| + \sign(F) \; W   \partial_v \tld f = 0
\end{align*}
and we deduce
\begin{align*}
\frac{d}{dt} \| F(t) \|_{L^1_{x,v}} \le \| W \|_{L^\infty_{t,x}}  \| \nabla_{x,v} \tld f \|_{L^1_{x,v}} .
\end{align*}
Using the $L^1$ gradient estimates on $\tld f$ from Lemma \ref{lemgrad}, we get the first inequality. 
Now let $G = vF$. Multiplying \myref{eq F} by $v$, we get
\begin{align*}
\partial_t G + v   \partial_x G + v(w-v)   \partial_v F - G + v W   \partial_v \tld f = 0.
\end{align*}
Since $(w-v) \partial_v G = (w-v) F + v(w-v) \partial_v F$ this can be rewritten as
\begin{align*}
\partial_t G + v \partial_x G+ (w-v) \partial_v G  -wF + v W \partial_v \tld f = 0.
\end{align*}
Multiplying by $\sign(G)$, we get
\begin{align*}
\partial_t |G| + v \partial_x |G| + (w-v) \partial_v |G| - \sign(G) w F + \sign(G) W v \partial_v \tld f = 0
\end{align*}
and we deduce
\begin{align*}
\frac{d}{dt} \| G(t) \|_{L^1_{x,v}} & \le - \| G(t) \|_{L^1_{x,v}} + \| w \|_{L^\infty_{t,x}} \| F(t) \|_{L^1_{x,v}} + \| W \|_{L^\infty_{t,x}} \| v \nabla_{x,v} \tld f(t) \|_{L^1_{x,v}}  .
\end{align*}
Using the $L^1$ gradient estimates on $\tld f$ and the inequality we just showed, Gr\"onwall's lemma gives the expected estimate.
\end{proof}

\end{subsection}

\begin{subsection}{Existence and uniqueness}

We are now ready to prove the following result.

\begin{proposition}  Assume that
\begin{align*} 
\int_x \int_v (1 + |v|^4) |f_0|^2 dx dv + \int_x \int_v (1+|v|) | \nabla_{x,v} f_0 | dx dv < \infty, 
\hspace{5mm}
u_0  \in H^\eta_x \text{ for } \eta \in (3/2,2).
\end{align*}
Then, for all $\beta \in (3/2 , \eta)$, there exists a unique couple $(f,u)$ with $u \in C([0,T] ; H^\beta_x)$ solution of the system \myref{system1}
in the sense that
\begin{equation*}
\left\{
\begin{array}{l}
 f_t = f_t[m+u] \; \; \text{ as defined in \myref{char1}},
\vspace{1mm} \\ 
\displaystyle{
 u_t = S(t) u_0 + \int_0^t S(t-s) \Big[ J(f_s) - \rho_s u_s \Big] ds,
 }
\end{array}
\right.
\end{equation*}
where $S(t) = e^{t \partial^2_x}$ denotes the semigroup associated to the heat equation.
For any $k \ge 0$, assuming $f_0 \in G^k$, we have additionally $f \in C([0,T] ; G^k)$.
\end{proposition}

\begin{remark}
The "numerical constants" used in the following proof may all depend on the quantity $\sup_{t \in [0,T]} \| m_t(\omega) \|_{W_x^{1,\infty}}$, which is in any case assumed to be bounded in Assumption \ref{ass1}.
\end{remark}
Let us  define  the recursive sequences
\begin{align*}
f^n \in C([0,T] ; L^1_{x,v}), \; \; n \ge 0 \\
u^n \in C([0,T] ; H^\eta_x), \; \; n \ge 0
\end{align*}
as follows:  $(f^0_t, u^0_t) = (f_0, u_0)$ and
\begin{align}
& f_t^{n+1} = f_t[m+u^n] \; \;  \text{ as defined in } \myref{char1}, \\
& u_t^{n+1} = S(t) u_0 + \int_0^t S(t-s) \Big[ J(f^{n+1}_s) - \rho^{n+1}_s u^n_s \Big] ds.
\end{align}

\begin{subsubsection}*{Claim 1: these sequences are well-defined}
Since $\eta > 3/2$, and we work in dimension $1$, the embedding $H^\eta_x \subset W_x^{1,\infty}$ guarantees that $f_t[m+u^n]$ can be properly defined as long as $u^n$ takes values in $H^\eta_x$. We have 
\begin{align*}
\| u_t^{n+1} \|_{H^\eta_x} \le \| u_0 \|_{H^\eta_x} + \int_0^t \left\| S(t-s) \Big[ J(f^{n+1}_s) - \rho^{n+1}_s u^n_s \Big] \right\|_{H^\eta_x} ds 
\end{align*}
with the classical estimate
\begin{align*}
\forall w \in L^2, \; \; \; \| S(t) w \|_{H^\eta_x} \lesssim (1+t^{-\eta/2}) \| w \|_{L^2_x}.
\end{align*}
The $L^2$ estimates for $f^{n+1}_t = f_t[m +u^n]$ from Lemma \ref{lemL2} give
\begin{align*}
\| J(f_s^{n+1}) - \rho_s^{n+1} u_s^n \|_{L^2_x} \le C\Big( \sup_{t \in [0,T]} \|u^n_t \|_{L^\infty_x} \Big) \Big(1+  \sup_{t \in [0,T]} \|u^n_t \|_{L^\infty_x}) \Big)
\end{align*}
from which we deduce
\begin{align}
\| u_t^{n+1} \|_{H^\eta_x} \le \| u_0 \|_{H^\eta_x} + C\Big( \sup_{t \in [0,T]} \|u^n_t \|_{L^\infty_x}) \Big) \int_0^t  (1+s^{-\eta/2}) ds .
\label{estim}
\end{align}
Since $\eta < 2$, it follows that, for all $n \ge 0$, $u^n \in C([0,T] ; H^\eta_x)$ and $f^n$ is well defined. 

\begin{remark}
This calculation explains why we have to restrict ourselves to the one-dimensional case $x \in \mathbb{T}$.
\label{well_posedness_remark}
\end{remark}

\end{subsubsection}

\begin{subsubsection}*{Claim 2: $(u^n)_n$ is bounded in $C([0,T]; L^\infty_x)$}
We have 
\begin{align*}
\| u_t^{n+1} \|_{L^\infty_x} \le  \|u_0\|_{L^\infty_x} + \int_0^t \Big \| S(t-s)\Big[ J(f^{n+1}_s) - \rho^{n+1}_s u^n_s \Big]  \Big \|_{L^\infty_x} ds
\end{align*}
and since $H^{1/2+\delta}_x \subset L^\infty_x$ for $\delta > 0$, we get $L^1_x \subset H^{-1/2-\delta}_x$ and the estimate
\begin{align*}
\forall w \in L^1_x, \; \; \; \| S(t) w \|_{L^\infty_x} \le \| S(t) w \|_{H^{1/2 + \delta}_x} \lesssim (1+t^{-1/2 - \delta}) \| w \|_{H^{-1/2-\delta}_x}
\lesssim (1+t^{-1/2-\delta}) \| w \|_{L^1_x} .
\end{align*}
The $L^1$ estimates for $f^{n+1}_t = f_t[m+u^n]$ from Lemma \ref{lemL1} give
\begin{align*}
\| J(f_s^{n+1}) - \rho_s^{n+1} u_s^n \|_{L^1_x} 
\lesssim 1+ \sup_{\sigma \in [0,s]} \| u_\sigma^n \|_{L^\infty_x} +  \|u^n_s\|_{L^\infty_x} \lesssim  1+ \sup_{\sigma \in [0,s]} \| u^n_\sigma \|_{L^\infty_x} .
\end{align*}
It follows that (forgetting the $\delta > 0$ for simplicity)
\begin{align*}
\| u^{n+1}_t \|_{L^\infty_x} \le  \|u_0\|_{L^\infty_x} + C \int_0^t  \Big( 1+(t-s)^{-1/2} \Big) \Big(1 + \sup_{\sigma \in [0,s]} \| u^n_\sigma \|_{L^\infty_x} \Big) ds .
\end{align*}
Introducing, for $\mu > 0$, the norm
\begin{align}
\forall w \in C([0,T], L^\infty_x), \; \; \| w \|_\mu = \sup_{t \in [0,T]} e^{-\mu t} \| w_t \|_{L^\infty_x} ,\label{norm}
\end{align}
is is easy to check that
$
\| u^{n+1} \|_\mu \le a_\mu \| u^n \|_\mu + b
$
with $a_\mu \in [0,1)$ for $\mu$ chosen large enough. It follows that $\sup_n \| u^n \|_\mu < \infty$, hence the result claimed.
\end{subsubsection}

\begin{subsubsection}*{Claim 3: $(u^n)_n$ is bounded in $C([0,T]; H^\eta_x)$}

Since $\sup_{t \in [0,T]} \| u^n \|_{L^\infty_x} \lesssim 1$, this follows directly from \myref{estim}.
\end{subsubsection}

\begin{subsubsection}*{Claim 4: $(u^n)_n$ is compact in $C([0,T]; H^\beta_x)$}

Noting that, with $g^n = J(f^{n}) - \rho^{n} u^{n-1}$, 
\begin{align*}
\forall \xi \in C^\infty_c(\mathbb{T}), \; \; \langle u^{n}_t, \xi \rangle - \langle u^{n}_0, \xi \rangle = \int_0^t \langle u^{n}_s, \partial^2_x \xi \rangle + \langle g^n_s, \xi \rangle ds,
\end{align*}
we are led to
\begin{align*}
\Big| \langle u^{n}_{t_2} - u^{n}_{t_1}, \xi \rangle \Big| & \le \int_{t_1}^{t_2} | \langle u^{n}_s, \partial^2_x \xi \rangle | + | \langle g^n_s, \xi\rangle | ds \\
& \le |t_2 - t_1| \Big( \| u^{n} \|_{C([0,T] ; H^\eta_x)}\| \xi \|_{H_x^{2-\eta}} + \| g^n \|_{C([0,T] ; L^2_{x,v})}  \| \xi \|_{L^2_x} \Big)
\\ & \lesssim  |t_2- t_1| \| \xi \|_{H^{2-\eta}_x},
\end{align*}
that is in fact $u^n \in C^{0,1}([0,T] ; H^{\eta-2}_x)$ with a Lipschitz constant uniform in $n$. Ascoli's theorem therefore guarantees the compactness of $(u^n)_{n} \in C([0,T] ; H^{\eta-2}_x)$.  Recalling the bound from Claim $3$, we get the compactness of $(u^n)_{n} \in C([0,T] ; H^\beta_x)$ by interpolation. 

\end{subsubsection}

\begin{subsubsection}*{Conclusion}

We simply note that the mapping $u_n \mapsto u_{n+1}$ is a contraction on $C([0,T] ; L^\infty_x)$:
\begin{align}
& \Big\| u^{n+1}_t - u_t^n \Big\|_{L^\infty_x} 
= \Big\| \int_0^t S(t-s) \Big[ J(f_s^{n+1}) - J(f_s^{n}) - (\rho^{n+1}_s u^n_s - \rho_s^{n} u_s^{n-1}) \Big] ds \Big\|_{L^\infty_x}  \label{contraction} \\
& \lesssim \int_0^t \Big(1 + (t-s)^{-1/2} \Big) \Big( \| J(f_s^{n+1}) - J(f_s^{n}) \|_{L^1_x} + \| \rho^{n+1}_s u^{n+1}_s - \rho_s^{n} u_s^{n-1} \|_{L^1_x} \Big) ds. \nonumber
\end{align}
Since $f^{n+1} = f[m+u^n]$ and $f^n = f[m + u^{n-1}]$, using Corollary \myref{corstab} and the fact that $(u^n)_n$ is bounded in $C([0,T] ; W_x^{1,\infty})$, we get for all $s \in [0,T]$
\begin{align*}
\| J(f_s^{n+1}) - J(f_s^{n}) \|_{L^1_x} \lesssim \sup_{\sigma \in [0,s] } \| u^{n}_\sigma - u^{n-1}_\sigma \|_{L^\infty_x}  ,
\end{align*}
and
\begin{align*}
\| \rho^{n+1}_s u^n_s - \rho_s^{n} u_s^{n-1} \|_{L^1_x}
&  \le 
\| \rho^{n+1}_s(u^n_s - u_s^{n-1}) \|_{L^1_x} + \|(\rho^{n+1}_s - \rho_s^{n}  )u_s^{n-1} \|_{L^1_x} 
\\ &
\lesssim
\sup_{\sigma \in [0,s]} \| u^n_\sigma - u^{n-1}_\sigma \|_{L^\infty_x}.
\end{align*}
Therefore,
\begin{align*}
\| u_t^{n+1} - u_t^n \|_{L^\infty_x} \lesssim \int_0^t \Big(1+(t-s)^{-1/2} \Big) \sup_{\sigma \in [0,s]} \| u^n_\sigma - u^{n-1}_\sigma \|_{L^\infty_x} ds .
\end{align*}
Using once again the norm $\| . \|_\mu$ defined in \myref{norm}, we are easily led to
\begin{align*}
\| u^{n+1} - u^n \|_\mu \le k_\mu \| u^n - u^{n-1} \|_\mu
\end{align*}
with $k_\mu \in [0, 1)$ for $\mu$ large enough.
As a result, $u^n$ converges to some $u$ in the Banach space $C([0,T] ; L^\infty_x)$. 
Since $(u^n)_n$ is also compact in $C([0,T] ; H^\beta_x)$ we deduce that 
$u^n \to u$ in $C([0,T]; H^\beta_x)$. 
We may now introduce $f = f[m+u]$. Using the same estimates as before (only replacing $u^{n+1}$ by $u$ in \myref{contraction}), we easily derive, as $n$ goes to infinity,
\begin{align*}
u_t = S(t) u_0 + \int_0^t S(t-s) \Big[ J(f_s) - \rho_s u_s \Big] ds.
\end{align*}
Hence, $(f,u)$ is indeed a solution to \myref{S}.
If $(\tld f, \tld u)$ is another solution with $\tld u \in C([0,T] ; H^\beta_x)$ and $\tld f = f[m + \tld u]$, the same estimates lead to
\begin{align*}
\| u - \tld u \|_\mu \le k_\mu \| u - \tld u \|_\mu,
\end{align*}
so that $\tld u = u$ and $\tld f = f$. This concludes the proof of the well-posedness of \myref{S}. 
\vspace{3mm}

Finally, let us assume that $f_0 \in G^k$ and prove that $f \in C([0,T] ; G^k)$. 
It is clearly enough to show that
$\| f_t - f_0 \|_{{G^k}} \to 0$ as $t$ goes to $0$. Given the form \myref{char1} and Lemma \ref{stab_flow}, 
this result is easily obtained by dominated convergence when $f_0 \in C_c(\mathbb{R}^{2d})$. It may then be generalized by density to any $f_0 \in {G^k}$, since, given $\tld f_0 \in C_c(\mathbb{R}^{2d})$
\begin{align*}
\forall t \in [0,1], \; \; \| f_t - \tld f_t \|_{{G^k}} = \int (1 + | V_0^t(z) |^k ) | f_0(z) - \tld f_0(z) | dz \lesssim \| f_0 - \tld f_0 \|_{{G^k}}.
\end{align*}
We have used the estimate $|V_0^t(x,v)|^k \lesssim 1+|v|^k$ easily derived from the sub-linearity of \myref{flow1}, with $\| m + u \|_{L^\infty_{t,x}} < \infty$.

\end{subsubsection}

\end{subsection}

\end{section}

\begin{subsection}*{Acknowledgment}

A. Debussche and A. Rosello are partially supported by the French government thanks to the "Investissements d'Avenir"
program ANR-11-LABX-0020-01. J. Vovelle is partially supported by the ANR project STAB and the "Investissements d'Avenir" program LABEX MILYON ANR-10-LABX-0070.

\end{subsection}


\bibliographystyle{abbrv}
\bibliography{biblio_these}

\begin{thebibliography}{10}

\bibitem{Ball77}
J.~M. Ball.
\newblock Strongly continuous semigroups, weak solutions, and the variation of
  constants formula.
\newblock {\em Proc. Amer. Math. Soc.}, 63(2):370--373, 1977.

\bibitem{billing}
P.~Billingsley.
\newblock {\em Convergence of probability measures}.
\newblock Wiley Series in Probability and Statistics: Probability and
  Statistics. John Wiley \& Sons Inc., New York, second edition, 1999.

\bibitem{bdg}
D.~L. Burkholder, B.~J. Davis, and R.~F. Gundy.
\newblock Integral inequalities for convex functions of operators on
  martingales.
\newblock In {\em Proceedings of the Sixth Berkeley Symposium on Mathematical
  Statistics and Probability, Volume 2: Probability Theory}, pages 223--240,
  Berkeley, Calif., 1972. University of California Press.

\bibitem{da_prato}
G.~Da~Prato and J.~Zabczyk.
\newblock {\em Stochastic Equations in Infinite Dimensions}.
\newblock Cambridge University Press, Cambridge, 2014.

\bibitem{ddv}
A.~Debussche, S.~De~Moor, and J.~Vovelle.
\newblock {Diffusion limit for the radiative transfer equation perturbed by a
  Markovian process}.
\newblock {\em {Asymptotic Analysis}}, 98(1-2):31--58, 2016.

\bibitem{dv1}
A.~Debussche and J.~Vovelle.
\newblock {Diffusion limit for a stochastic kinetic problem}.
\newblock {\em {Communications on Pure and Applied Mathematics}},
  11(6):2305--2326, 2012.

\bibitem{dv2}
A.~Debussche and J.~Vovelle.
\newblock Diffusion-approximation in stochastically forced kinetic equations,
  2017.

\bibitem{du_meng}
K.~Du and Q.~Meng.
\newblock A revisit to w2n-theory of super-parabolic backward stochastic
  partial differential equations in rd.
\newblock {\em Stochastic Processes and their Applications}, 120(10):1996 --
  2015, 2010.

\bibitem{wave}
J.-P. Fouque, J.~Garnier, G.~Papanicolaou, and K.~S{\o}lna.
\newblock {\em Wave Propagation and Time Reversal in Randomly Layered Media},
  volume~56 of {\em Stochastic Modelling and Applied Probability}.
\newblock Springer-Verlag New York, 1 edition, 2007.

\bibitem{papa}
D.~S. G.~Papanicolaou and S.~Varadhan.
\newblock Martingale approach to some limit theorems.
\newblock {\em Duke University Series}, 3, 1977.

\bibitem{JacodShiryaev03}
J.~Jacod and A.~N. Shiryaev.
\newblock {\em Limit theorems for stochastic processes}, volume 288 of {\em
  Grundlehren der Mathematischen Wissenschaften [Fundamental Principles of
  Mathematical Sciences]}.
\newblock Springer-Verlag, Berlin, second edition, 2003.

\bibitem{roch}
S.~Roch.
\newblock Modern discrete probability, an essential toolkit.
  \url{https://www.math.wisc.edu/~roch/mdp/roch-mdp-chap3.pdf}, 2015.

\end{thebibliography}

\end{document}